\pdfoutput=1

\documentclass[a4paper,10pt]{article}
\usepackage[english]{babel}
\usepackage[latin1]{inputenc}
\usepackage[T1]{fontenc}
\usepackage{amsmath}
\usepackage{amssymb}
\usepackage{multirow}
\usepackage{alltt}
\usepackage{graphicx}
\usepackage{bbm}
\usepackage{amsthm}
\usepackage{mathrsfs}
\usepackage{textcomp}
\usepackage{fullpage}
\usepackage[usenames,dvipsnames]{color}

\definecolor{darkblue}{RGB}{0,0,170}
\definecolor{brickred}{RGB}{200,0,0}
\usepackage[pdfauthor={Nicola Abatangelo}, pdftitle={large $s$-harmonic functions and boundary blow-up solutions for the fractional laplacian}, hyperindex]{hyperref}
\hypersetup{colorlinks=true, linkcolor=darkblue, citecolor=brickred}

\newtheorem{defi}{Definition}[subsection]

\newtheorem{theo}[defi]{Theorem}
\newtheorem{prop}[defi]{Proposition}
\newtheorem{lem}[defi]{Lemma}
\newtheorem{rmk}[defi]{Remark}

\providecommand{\tfaname}{\it Proof.}
\newenvironment{tfa}{\begin{proof}[\tfaname]}{\end{proof}}

\newcommand{\R}{\mathbb{R}}
\newcommand{\N}{\mathbb{N}}
\newcommand{\C}{\mathcal{C}}
\newcommand{\A}{\mathcal{A}}
\newcommand{\T}{\mathcal{T}}
\newcommand{\eps}{\varepsilon}
\newcommand{\dist}{\hbox{dist}}
\newcommand{\Ds}{{\left(-\Delta\right)}^s}

\begin{document}


\begin{center}
\Large{\bf\scshape large $s$-harmonic functions
and boundary blow-up solutions for the fractional laplacian}
\end{center}

\begin{center}
\scshape nicola abatangelo$\,$\footnote{
Laboratoire Ami\'enois de Math\'ematique Fondamentale et Appliqu\'ee,
UMR CNRS 7352, Universit\'e de Picardie Jules Verne,
33 rue Saint-Leu, 80039 Amiens (France) -- 
\texttt{nicola.abatangelo@u--picardie.fr}}$^{,\,}$\footnote{
Dipartimento di Matematica Federigo Enriques,
Universit\`a degli Studi di Milano,
via Saldini 50, 20133 Milano (Italy)}
\end{center}

\begin{center}
\begin{minipage}{425pt}
\begin{small}
\bf Abstract. \rm We present a notion of weak solution for the Dirichlet problem
driven by the fractional Laplacian, following the Stampacchia theory.
Then, we study semilinear problems of the form
$$
\left\lbrace\begin{array}{ll}
\Ds u=\pm\,f(x,u) & \hbox{ in }\Omega \\ 
u=g & \hbox{ in }\R^n\setminus\overline{\Omega}\\
Eu=h & \hbox{ on }\partial\Omega
\end{array}\right.
$$
when the nonlinearity $f$ and the boundary data $g,h$ are positive, 
but allowing the right-hand side to be both positive or negative
and looking for solutions
that blow up at the boundary.
The operator $E$ is a weighted limit to the boundary:
for example, if $\Omega$ is the ball $B$, there exists a constant $C(n,s)>0$
such that 
$$
Eu(\theta)=C(n,s)\lim_{\stackrel{\hbox{\scriptsize $x\!\rightarrow\!\theta$}}{x\in B}}u(x)\,{\dist(x,\partial B)}^{1-s},
\hbox{ for all }\theta\in\partial B.
$$

Our starting observation is the existence of 
$s$-harmonic functions which explode at the boundary:
these will be used both as supersolutions in the case 
of negative right-hand side
and as subsolutions in the positive case.
\end{small}
\end{minipage}
\end{center}

\tableofcontents


\section{Introduction}

In this paper we study a suitable notion of weak solution
to semilinear problems driven by the fractional Laplacian $\Ds$, i.e.
the integral operator defined as (see e.g. Di Nezza, Palatucci and Valdinoci \cite{hitchhiker} for an introduction)
\begin{equation}
\Ds f(x)=\A(n,s)\,PV\int_{\R^n}\frac{f(x)-f(y)}{{|x-y|}^{n+2s}}\;dy.
\end{equation}
where $\mathcal{A}(n,s)$ is a normalizing constant\footnote{
$\displaystyle\A(n,s)= 
 \frac{\Gamma\left(\frac{n}{2}+s\right)}{{\pi}^{n/2+2s}\,\Gamma(2-s)}\cdot s(1-s)$, see \cite[formula (1.1.2)]{landkof};
 this constant is chosen in such a way that, for any $f$ in the Schwartz class $\mathcal{S}$,
 $\mathcal{F}[\Ds f](\xi)={|\xi|}^{2s}\,\mathcal{F}f(\xi)$
 where $\mathcal{F}$ is the Fourier transform, see \cite[formulas (1.1.6), (1.1.12') and (1.1.1)]{landkof}}.

In order to do this, we will need to develop a theory
for the Dirichlet problem for the fractional Laplacian with measure data
(see Karlsen, Petitta and Ulusoy \cite{karlsen} and Chen and V\'eron \cite{chen-veron}, for earlier results
in this direction).
We pay particular attention to those solutions having an explosive behaviour at 
the boundary of the prescribed domain,
known in the literature as \it large solutions \rm or also
\it boundary blow-up solutions\rm.
\smallskip

Let us recall that in the classical setting (see Axler, Bourdon and Ramey \cite[Theorem 6.9]{axler}),
to any nonnegative Borel measure $\mu\in\mathcal{M}(\partial B)$ on $\partial B$ it is possible to associate,
via the representation through the Poisson kernel,
a harmonic function in $B$ with $\mu$ as its trace on the boundary.
Conversely,
any positive harmonic function on the ball $B$
has a trace on $\partial B$ that is a nonnegative Borel measure
(see \cite[Theorem 6.15]{axler}).

When studying the semilinear problem for the Laplacian,
solutions can achieve the boundary datum
$+\infty$ on the whole boundary. More precisely,  take $\Omega$ a bounded smooth domain
and $f$ nondecreasing such that $f(0)=0$. According to the works of
Keller \cite{keller} and Osserman \cite{osserman},
the equation 
$$
\left\lbrace\begin{array}{l}
\displaystyle
\Delta u=f(u) \ \hbox{ in }\Omega\subseteq\R^n, \\
\displaystyle
\lim_{\stackrel{\hbox{\scriptsize $x\!\rightarrow\!\partial\Omega$}}{x\in\Omega}}u(x)=+\infty.
\end{array}\right.
$$
has a solution if and only if $f$ satisfies the so called \it Keller-Osserman condition\rm,
that is
$$
\int^\infty\frac{dt}{\sqrt{F(t)}}<+\infty,\quad F'(t)=f(t),
$$
see also Dumont, Dupaigne, Goubet and R{\u{a}}dulescu \cite{KO-dupaigne} for the case of oscillating nonlinearity.
The case of nonsmooth domains is delicate, see in particular the work of Dhersin and Le Gall \cite{dherslegall}
for the case $f(u)=u^2$.
For the same nonlinearity, and for $\Omega=B$, Mselati \cite{mselati} completely classified
positive solutions in terms of their boundary trace,
which can be $+\infty$ on one part of the boundary and a measure that doesn't
charge sets of zero boundary capacity on the remaining part.
See the upcoming book by Marcus and V\'eron for further developments in this direction.
\smallskip

In the fractional context, our starting point is that
large solutions arise even in linear problems.
In particular, it is possible to provide
\it large $s$-harmonic functions\rm, i.e. functions satisfying
\begin{equation}\label{problem}
\left\lbrace\begin{array}{l}
\displaystyle
\Ds u=0 \ \hbox{ in some open bounded region }\Omega\subseteq\R^n \\
\displaystyle
\lim_{\stackrel{\hbox{\scriptsize $x\!\rightarrow\!\partial\Omega$}}{x\in\Omega}}u(x)=+\infty
\end{array}\right.
\end{equation}
see for example Remark \ref{largesharm} below.
An example of a large $s$-harmonic function on the unit ball $B$ 
is
$$
u_\sigma(x)=\left\lbrace\begin{array}{ll}
\displaystyle
\frac{c(n,s)}{\left(1-|x|^2\right)^\sigma} & \hbox{ in }B \\
\displaystyle
\frac{c(n,s+\sigma)}{\left(|x|^2-1\right)^\sigma} & \hbox{ in }\R^n\setminus\overline{B}
\end{array}\right.\qquad\sigma\in(0,1-s),\ \ c(n,s)=\frac{\Gamma(n/2)\,\sin(\pi s)}{\pi^{1+n/2}}.
$$
See Lemma \ref{expl-sol} below.
Moreover, letting $\sigma\rightarrow1-s$ we recover
the following example found in Bogdan, Byczkowski, Kulczycki,
Ryznar, Song, and Vondra{\v{c}}ek \cite{sharm},
qualitatively different from the previous one:
$$
u_{1-s}(x)=\left\lbrace\begin{array}{ll}
\displaystyle
\frac{c(n,s)}{\left(1-|x|^2\right)^{1-s}} & \hbox{ in }B, \\
\displaystyle
0 & \hbox{ in }\R^n\setminus\overline{B}.
\end{array}\right.
$$
The function $u_{1-s}$ so defined satisfies
$$
\left\lbrace\begin{array}{ll}
\Ds u=0 & \hbox{ in }B \\
u=0 & \hbox{ in }\R^n\setminus\overline{B}
\end{array}\right.
$$
and shows how problems where only outer values are prescribed
are ill-posed in the classical sense. Different kinds of boundary conditions have
to be taken into account:
indeed, in the first case we have an $s$-harmonic function
associated to the prescribed data of $u_\sigma$ outside $B$;
in the second case all the mass of the boundary datum 
concentrates on $\partial B$.
This means that we need a notion of weak solution
that can deal at the same time with these two different boundary data,
one on the complement of the domain
and the other one on its boundary.
\smallskip

Recently, Felmer and Quaas \cite{felmer-quaas} and Chen, Felmer and Quaas \cite{chen-felmer}
have shown the
existence of large solutions to problems of the form
$$
\left\lbrace\begin{array}{ll}
\Ds u+{u}^{\,p}=f & \hbox{ in }\Omega, \\
u=g & \hbox{ in }\R^n\setminus\overline{\Omega}
\end{array}\right.
$$
both under assumptions of the explosion of the datum $g$ at $\partial\Omega$
and when $g=0$.
In all cases they need to assume $p>2s+1$.
Our approach allows to deal with general equations of the form
$$
\Ds u=f(x,u) \quad\hbox{ in }\Omega
$$
with no necessary assumptions on the sign and the growth of the nonlinearity $f$
and to provide large solutions in both cases where $u$ is
prescribed outside of $\Omega$ or only at $\partial\Omega$ as a measure.

\subsection{The notion of $s$-harmonicity}

Our starting point is the definition of $s$-harmonicity, found in Landkof \cite[\textsection 6.20]{landkof}.
Denote by
$$
\eta(x)=\left\lbrace\begin{array}{ll}
\displaystyle
\frac{c(n,s)}{{|x|}^n\,\left(|x|^2-1\right)^s} & |x|>1 \\
0 & |x|\leq 1
\end{array}\right.
$$
and by
$$
\eta_r(x)=\frac{1}{r^n}\;\eta\left(\frac{x}{r}\right)=\left\lbrace\begin{array}{ll}
\displaystyle
\frac{c(n,s)\,r^{2s}}{{|x|}^n\,\left(|x|^2-r^2\right)^s} & |x|>r, \\
0 & |x|\leq r.
\end{array}\right.
$$
The constant $c(n,s)$ above is chosen in such a way that
\begin{equation}\label{cns}
\int_{\R^n}\eta(x)\;dx=\int_{\R^n}\eta_r(x)\;dx=1
\end{equation}
and therefore, see \cite[\textsection 6.19]{landkof}
\begin{equation}\label{cns2}
c(n,s)=\frac{\Gamma(n/2)\,\sin(\pi s)}{\pi^{1+n/2}}.
\end{equation}

The definition of $s$-harmonicity is given via a
mean value property, namely

\begin{defi}\label{sharm-def} We say that a measurable nonnegative function 
$u:\R^n\rightarrow[0,+\infty]$ is
$s$-harmonic on an open set $\Omega\subseteq\R^n$ if $u\in C(\Omega)$ and
for any $x\in\Omega$ and $0<r<\dist(x,\partial\Omega)$
\begin{equation}\label{sharm-eq}
u(x)=\int_{\C B_r(x)}\frac{c(n,s)\,r^{2s}}{|y-x|^n\,\left(|y-x|^2-r^2\right)^s}\:u(y)\;dy
=\left(\eta_r*u\right)(x).
\end{equation}
\end{defi}

\subsection{Hypotheses and main results}

\subsection*{An integration by parts formula}\label{intparts-sec}

For any two functions $u,\ v$ in the Schwartz class $\mathcal{S}$,
the self-adjointness of the operator $\Ds$ entails
$$
\int_{\R^n}u\,\Ds v\ =\ \int_{\R^n}\Ds u\,v,
$$
this follows from the representation of the fractional Laplacian
via the Fourier transform,
see \cite[Paragraph 3.1]{hitchhiker}. By splitting $\R^n$ into two domains of integration 
\begin{equation}\label{intparts}
\int_{\Omega}u\,\Ds v-\int_{\Omega}\Ds u\,v\ =\ 
\int_{\R^n\setminus\Omega}\Ds u\,v-\int_{\R^n\setminus\Omega}u\,\Ds v.
\end{equation}

\begin{prop}\label{integrbypartsform}
Let $\Omega\subseteq\R^n$ open and bounded. 
Let $C^{2s+\eps}(\Omega)=\{v:\R^n\rightarrow\R \hbox{ such that } v\in C(\Omega)\hbox{ and for any }K\hbox{ compactly supported
in }\Omega, \hbox{ there exists }\alpha=\alpha(K,v)\hbox{ such that }v\in C^{2s+\alpha}(K) \}$.
If $u\in C^{2s+\eps}(\Omega)\cap L^\infty(\R^n)$ and 
$v=0$ in $\R^n\setminus\Omega$, $v\in C^{2s+\eps}(\Omega)\cap C^s(\R^n)$
and $\Ds v\in L^1(\R^n)$, then 
\begin{equation}\label{intpartsprop}
\int_{\Omega}u\,\Ds v-\int_{\Omega}\Ds u\,v\ =\ -\int_{\R^n\setminus\Omega}u\,\Ds v.
\end{equation}
\end{prop}
The proof can be found in the Appendix.

From now on the set $\Omega\subseteq\R^n$ will be an
open bounded domain with $C^{1,1}$ boundary.
More generally, we will prove in Section \ref{liner-sec} the following
\begin{prop}\label{intparts-prop} Let 
$\delta(x)=\dist(x,\partial\Omega)$ for any $x\in\R^n$,
and $u\in C^{2s+\eps}_{loc}(\Omega)$ such that
\begin{equation}\label{gg}
{\delta(x)}^{1-s}u(x)\in C(\overline{\Omega})
\qquad\hbox{ and }\qquad
\int_{\R^n}\frac{|u(x)|}{1+{|x|}^{n+2s}}\;dx<+\infty.
\end{equation}
Let
\begin{description}
\item[\rm --] $G_\Omega(x,y)$, $x,y\in\Omega,\,x\neq y$, be the Green function 
of the fractional Laplacian on $\Omega$, that is
$$
G_\Omega(x,y)=\frac{\A(n,-s)}{{|x-y|}^{n-2s}}-H_\Omega(x,y)
$$ 
and $H_\Omega$ is the unique function in $C^s(\R^n)$ solving
$$
\left\lbrace\begin{array}{ll}
\Ds H_\Omega(x,\cdot)=0 & \hbox{ in }\Omega \\
H_\Omega(x,y)=\frac{\A(n,-s)}{{|x-y|}^{n-2s}} & \hbox{ in }\R^n\setminus\overline{\Omega}
\end{array}\right.
$$
pointwisely,\footnote{the construction of $H$ can be found in the proof of Theorem \ref{meanvalue} below}
\item[\rm --] $P_\Omega(x,y)=-\Ds G_\Omega(x,y)$, $x\in\Omega,y\in\R^n\setminus\overline{\Omega}$,
be the corresponding Poisson kernel,
\item[\rm --] for $x\in\Omega,\theta\in\partial\Omega$,\footnote{this is a readaptation of the Martin kernel of $\Omega$} 
$$
M_\Omega(x,\theta)=\lim_{\stackrel{\hbox{\scriptsize $y\!\rightarrow\!\theta$}}{y\in\Omega}}\frac{G(x,y)}{{\delta(y)}^s},
$$
\item[\rm --] for $\theta\in\partial\Omega$, 
$$
Eu(\theta)=
\lim_{\stackrel{\hbox{\scriptsize $x\!\rightarrow\!\theta$}}{x\in\Omega}}
\frac{{\delta(x)}^{1-s}\,u(x)}{\int_{\partial\Omega}M_\Omega(x,\theta')\;d\mathcal{H}(\theta')},
$$ 
where the limit is well-defined in view of Lemma \ref{Eu} below.
\end{description}
Then the integration by parts formula
\begin{equation}\label{byparts}
\int_{\Omega}u\,\Ds v\ =\ 
\int_{\Omega}\Ds u\,v-\int_{\R^n\setminus\overline{\Omega}}u\,\Ds v
+\int_{\partial\Omega}Eu(\theta)\,D_sv(\theta)\;d\mathcal{H}(\theta)
\end{equation}
holds, where
$$
D_sv(\theta)=\lim_{\stackrel{\hbox{\scriptsize $x\!\rightarrow\!\theta$}}{x\in\Omega}}\frac{v(x)}{{\delta(x)}^s},
\quad\theta\in\partial\Omega
$$
for any $v\in C^s(\R^n)$ such that $\Ds v|_\Omega\in C^\infty_c(\Omega)$ 
and $v\equiv 0$ in $\R^n\setminus\Omega$.
Such a limit exists and is continuous in $\theta$
in view of Lemma \ref{testspace}.
In addition, we have the representation formula
\begin{equation}\label{representation}
u(x)=\int_\Omega G_\Omega(x,y)\Ds u(y)\; dy
-\int_{\R^n\setminus\overline{\Omega}}\Ds G_\Omega(x,y)\,u(y)\;dy
+\int_{\partial\Omega}D_sG_\Omega(x,\theta)\, Eu(\theta)\;d\mathcal{H}(\theta).
\end{equation}
\end{prop}

The paper is organized as follows. 
Section \ref{meanvalue-sec} relates Definition \ref{sharm-def}
with the fractional Laplacian $\Ds$.
Section \ref{liner-sec} recalls some facts on Green functions and
Poisson kernels and it studies the linear Dirichlet problem
both in the pointwise and in a weak sense.
Section \ref{nonlin-sec} deals with the nonlinear problem.
Now, let us outline the main results in Section \ref{liner-sec}
and Section \ref{nonlin-sec}.

\subsection*{The Dirichlet problem}

For a fixed $x\in\Omega$ the Poisson kernel satisfies 
$$
0\leq-\Ds G_\Omega(x,y)\leq c\,\min\left\lbrace{\delta(y)}^{-s},{\delta(y)}^{-n-2s}\right\rbrace,
\qquad y\in\R^n\setminus\overline{\Omega}
$$
for some constant $c>0$ independent of $y$,
as it will be shown later on.
In particular any Dirichlet condition $u=g$ in $\R^n\setminus\overline{\Omega}$
satisfying \eqref{gintro} below is admissible
in the representation formula \eqref{representation}. 
We prove the following

\begin{theo}\label{pointwise} Let $f\in C^\alpha(\overline{\Omega})$
for some $\alpha\in(0,1)$, 
$g:\R^n\setminus\overline{\Omega}\rightarrow\R$ be any measurable function 
satisfying\footnote{compare also with equation \eqref{g} below}
\begin{equation}\label{gintro}
\int_{\R^n\setminus\overline{\Omega}}|g(y)|\min\left\lbrace{\delta(y)}^{-s},{\delta(y)}^{-n-2s}\right\rbrace\;dy<+\infty
\end{equation}
and $h\in C(\partial\Omega)$.
Then, the function defined by setting
\begin{equation}\label{linearsol}
u(x)=
\left\lbrace\begin{array}{ll}
\displaystyle \int_\Omega f(y)\,G_\Omega(y,x)\;dy
-\int_{\R^n\setminus\overline{\Omega}}g(y)\,\Ds G_\Omega(x,y)\;dy
+\int_{\partial\Omega}h(\theta)\,M_\Omega(x,\theta)\;d\mathcal{H}(\theta), & x\in\Omega \\
g(x), & x\in\R^n\setminus\overline\Omega 
\end{array}\right.
\end{equation}
belongs to $C^{2s+\eps}_{loc}(\Omega)$, fulfills \eqref{gg}, and $u$ is the only pointwise solution of 
$$
\left\lbrace\begin{array}{ll}
\Ds u = f & \hbox{ in }\Omega, \\
u = g & \hbox{ in }\R^n\setminus\overline{\Omega,} \\
Eu = h & \hbox{ on }\partial\Omega.
\end{array}\right.
$$
Here $Eu$ has been defined in Proposition \ref{intparts-prop}.
Moreover, if $g\in C(\overline{V_\eps})$ for some $\eps>0$,
where $V_\eps=\{x\in\R^n\setminus\overline{\Omega}:\delta(x)<\eps\}$
and $h=0$, then $u\in C(\overline{\Omega})$.
\end{theo}

\begin{rmk}\rm Even if it is irrelevant to write $\int_{\R^n\setminus\Omega}$
or $\int_{\R^n\setminus\overline{\Omega}}$ in formula \eqref{linearsol}, with the latter notation
we would like to stress on the fact that the boundary $\partial\Omega$
plays an important role in this setting.
\end{rmk}

\begin{rmk}\rm Since $\partial\Omega\in C^{1,1}$, we will exploit the behaviour
of the Green function and the Poisson kernel
described by \cite[Theorem 2.10 and equation (2.13) resp.]{chen}:
there exist $c_1=c_1(\Omega,s)>0$, $c_2=c_2(s,\Omega)$ such that
\begin{equation}\label{est-poisson}
\frac{{\delta(x)}^s}{c_1{\delta(y)}^s\left(1+\delta(y)\right)^s{|x-y|}^n}
\ \leq\ -\Ds G_\Omega(x,y)\ \leq\ 
\frac{c_1{\delta(x)}^s}{{\delta(y)}^s\left(1+\delta(y)\right)^s{|x-y|}^n},
\quad x\in\Omega,\ y\in\R^n\setminus\overline\Omega,
\end{equation}
and
\begin{equation}\label{est-green}
\frac{1}{c_2\,|x-y|^n}\left(|x-y|^2\,\wedge\,\delta(x)\delta(y)\right)^s\ 
\leq\ G_\Omega(x,y)\ \leq\ 
\frac{c_2}{|x-y|^n}\left(|x-y|^2\,\wedge\,\delta(x)\delta(y)\right)^s\ 
\qquad x,y\in\Omega.
\end{equation}
\end{rmk}

\begin{rmk}[Construction of large $s$-harmonic functions]\label{largesharm}\rm 
The case when $f=0$ in Theorem \ref{pointwise} corresponds to 
$s$-harmonic functions: when $h\not\equiv 0$ then $u$ automatically 
explodes somewhere on the boundary (by definition of $E$), while
if $h\equiv 0$ then large $s$-harmonic functions can be built
as follows.
Take any positive $g$ satisfying \eqref{gintro} with
$$
\lim_{\delta(x)\downarrow 0}g(x)=+\infty
$$
and let $g_N=\min\{g,N\}$, $N\in\N$.
By Theorem \ref{pointwise}, the corresponding solutions $u_N\in C(\overline{\Omega})$.
In particular, $u_N=N$ on $\partial\Omega$ and by the Maximum Principle
${\{u_N\}}_N$ is increasing. Hypotheses \eqref{gintro} guarantees a uniform bound 
on ${\{u_N\}}_N$. Then 
$$
u_N(x)=-\int_{\R^n\setminus\Omega}g_N(y)\Ds G_\Omega(x,y)\;dy
$$
increases to a $s$-harmonic function $u$ such that
$$
\liminf_{\stackrel{\hbox{\scriptsize $x\in\Omega$}}{x\rightarrow\partial\Omega}}u(x)\geq
\liminf_{\stackrel{\hbox{\scriptsize $x\in\Omega$}}{x\rightarrow\partial\Omega}}u_N(x)=N,
\qquad \hbox{ for any }N\in\N.
$$
\end{rmk}

Next, in view of Theorem \ref{pointwise},
we introduce the test function space
$$
\T(\Omega)=\left\lbrace\phi\in C^\infty(\Omega):
\hbox{ there exists }\psi\in C^\infty_c(\Omega)
\hbox{ such that }
\left\lbrace\begin{array}{ll}
\Ds \phi=\psi & \hbox{ in }\Omega \\
\phi=0 & \hbox{ in }\R^n\setminus\overline\Omega \\
E\phi=0 & \hbox{ in }\partial\Omega
\end{array}\right.\right\rbrace.
$$
and starting from the integration by parts formula \eqref{byparts},
we introduce the following notion of weak solution  

\begin{defi}\label{weakdefiintro} Given three Radon measures $\lambda\in\mathcal{M}(\Omega)$, 
$\mu\in\mathcal{M}(\R^n\setminus\Omega)$ and $\nu\in\mathcal{M}(\partial\Omega)$, such that
$$
\int_\Omega\delta(x)^s\;d|\lambda|(x)<+\infty,\quad
\int_{\R^n\setminus\overline\Omega}\min\{\delta(x)^{-s},\delta(x)^{-n-2s}\}\;d|\mu|(x)<+\infty,
\quad |\nu|(\Omega)<+\infty
$$
we say that a function $u\in L^1(\Omega)$
is a solution of 
$$
\left\lbrace\begin{array}{ll}
\Ds u=\lambda & \hbox{ in }\Omega \\
u=\mu & \hbox{ in }\R^n\setminus\overline{\Omega} \\
Eu=\nu & \hbox{ on }\partial\Omega
\end{array}\right.
$$
if for every $\phi\in\T(\Omega)$ it is
\begin{equation}\label{weakdefi-eq}
\int_\Omega u(x)\Ds \phi(x)\;dx\ =\ \int_\Omega\phi(x)\;d\lambda(x)
-\int_{\R^n\setminus\overline{\Omega}}\Ds \phi(x)\;d\mu(x)+\int_{\partial\Omega}D_s\phi(\theta)\;d\nu(\theta).
\end{equation}
\end{defi}
The integrals in the definition are finite for any $\phi\in\T(\Omega)$
in view of Lemma \ref{testspace}.

Thanks to the representation formula \eqref{representation},
we prove

\begin{theo}\label{existence-weak2}
Given two Radon measures $\lambda\in\mathcal{M}(\Omega)$
and 
$\mu\in\mathcal{M}(\R^n\setminus\overline{\Omega})$ such that
$$
\int_\Omega G_\Omega(x,y)\;d|\lambda|(y)<+\infty,
\qquad\hbox{and}\qquad
-\int_{\R^n\setminus\overline{\Omega}}\Ds G_\Omega(x,y)\;d|\mu|(y)<+\infty,\qquad
\hbox{ for a.e. }x\in\Omega,
$$
and a Radon measure $\nu\in\mathcal{M}(\partial\Omega)$ such that $|\nu|(\partial\Omega)<+\infty$,
the problem
$$
\left\lbrace\begin{array}{ll}
\Ds u=\lambda & \hbox{ in }\Omega \\
u=\mu & \hbox{ in }\R^n\setminus\overline{\Omega} \\
Eu=\nu & \hbox{ on }\partial\Omega
\end{array}\right.
$$
admits a unique solution $u\in L^1(\Omega)$ in the weak sense.

In addition, we have the representation formula 
$$
u(x)=\int_\Omega G_\Omega(x,y)\;d\lambda(y)-\int_{\C\Omega}\Ds G_\Omega(x,y)\;d\mu(y)
+\int_{\partial\Omega}M_\Omega(x,y)\;d\nu(y)
$$
and
$$
\Arrowvert u\Arrowvert_{L^1(\Omega)}\leq C\left(
\Arrowvert{\delta(x)}^s\lambda\Arrowvert_{\mathcal{M}(\Omega)}+
\Arrowvert\min\{{\delta(x)}^{-s},{\delta(x)}^{-n-2s}\}\,\mu\Arrowvert_{\mathcal{M}(\R^n\setminus\overline{\Omega})}+
\Arrowvert\nu\Arrowvert_{\mathcal{M}(\partial\Omega)}\right).
$$
for some constant $C=C(n,s,\Omega)>0$.
\end{theo}

We will conclude Section \ref{liner-sec} by showing that
\begin{prop} The weak solution of
$$
\left\lbrace\begin{array}{ll}
\Ds u(x)=\frac{1}{{\delta(x)}^\beta} & \hbox{ in }\Omega \\
u=0 & \hbox{ in }\R^n\setminus\overline{\Omega} \\
Eu=0 & \hbox{ on }\partial\Omega
\end{array}\right.
$$
satisfies 
\begin{eqnarray}
\displaystyle
c_1{\delta(x)}^s
\leq u(x)\leq
c_2{\delta(x)}^s & &
\hbox{ for }0<\beta<s, \nonumber \\ 
\displaystyle
c_3{\delta(x)}^s\log\frac{1}{\delta(x)}
\leq u(x)\leq
c_4{\delta(x)}^s\log\frac{1}{\delta(x)} & &
\hbox{ for }\beta=s, \label{udelta} \\
\displaystyle
c_5{\delta(x)}^{-\beta+2s}
\leq u(x) \leq
c_6{\delta(x)}^{-\beta+2s} & &
\hbox{ for }s<\beta<1+s. \nonumber
\end{eqnarray}
Moreover, there exist a constant $\underline{c}=\underline{c}(n,s,\Omega)>0$
and $\overline{c}=\overline{c}(n,s,\Omega)$
such that the solution of 
$$
\left\lbrace\begin{array}{ll}
\Ds u=0 & \hbox{ in }\Omega \\
u=g & \hbox{ in }\R^n\setminus\overline{\Omega} \\
Eu=0 & \hbox{ on }\partial\Omega
\end{array}\right.
$$
satisfies
\begin{equation}\label{uleqg}
\underline{c}\,\underline{g}(\delta(x)):=
\underline{c}\,\inf_{\delta(y)=\delta(x)}g(y)
\leq u(x)\leq 
\overline{c}\sup_{\delta(y)=\delta(x)}g(y)
=:\overline{c}\,\overline{g}(\delta(x))
\qquad x\in\Omega,
\end{equation}
for any $g$ satisfying \eqref{gintro} such that $\underline{g},\,\overline{g}$
are decreasing functions in $0^+$ and
$$
\lim_{t\downarrow 0}\underline{g}(t)=+\infty.
$$
\end{prop}

\subsection*{The nonlinear problem}

We consider nonlinearities $f:\Omega\times\R\rightarrow\R$ 
satisfying hypotheses 
\begin{description}
\item[\it f.1)] $f\in C(\Omega\times\R)$, $f\in L^\infty(\Omega\times I)$ for any bounded $I\subseteq\R$
\item[\it f.2)] $f(x,0)=0$ for any $x\in\Omega$, and
$f(x,t)\geq 0$ for any $x\in\Omega,\,t>0$,
\end{description}
and all positive boundary data $g$ that satisfy \eqref{gintro}.

After having constructed large $s$-harmonic functions,
we first prove the following preliminary

\begin{theo}\label{nl-cs}
Let $f:\Omega\times\R\rightarrow\R$ be a function 
satisfying f.1).
Let $g:\R^n\setminus\overline{\Omega}\rightarrow\R$ be a measurable bounded function.
Assume the nonlinear problem
$$
\left\lbrace\begin{array}{ll}
\Ds u=-f(x,u) & \hbox{ in }\Omega \\
u=g & \hbox{ in }\R^n\setminus\overline{\Omega} \\
Eu=0 & \hbox{ on }\partial\Omega
\end{array}\right.
$$
admits a subsolution $\underline{u}\in L^1(\Omega)$ 
and a supersolution $\overline{u}\in L^1(\Omega)$ in the weak sense
$$
\left\lbrace\begin{array}{ll}
\Ds\underline{u}\leq -f(x,\underline{u}) & \hbox{ in }\Omega \\
\underline{u} \leq g & \hbox{ in }\R^n\setminus\overline{\Omega}\\
E\underline{u}=0 & \hbox{ on }\partial\Omega
\end{array}\right.
\qquad\hbox{and}\qquad
\left\lbrace\begin{array}{ll}
\Ds\overline{u}\geq -f(x,\overline{u}) & \hbox{ in }\Omega \\
\overline{u} \geq g & \hbox{ in }\R^n\setminus\overline{\Omega}\\
E\overline{u}=0 & \hbox{ on }\partial\Omega
\end{array}\right.
$$
i.e. for any $\phi\in\T(\Omega)$, $\Ds\phi|_\Omega\geq 0$, it is
$$
\int_\Omega\underline{u}\,\Ds\phi\leq
-\int_\Omega f(x,\underline{u})\,\phi
-\int_{\R^n\setminus\overline{\Omega}} g\,\Ds\phi,
\qquad
\int_\Omega\overline{u}\,\Ds\phi\geq
-\int_\Omega f(x,\overline{u})\,\phi
-\int_{\R^n\setminus\overline{\Omega}} g\,\Ds\phi.
$$
Assume also $\underline{u}\leq\overline{u}$ in $\Omega$, 
and $\underline{u},\overline{u}\in L^\infty(\Omega)\cap C(\Omega)$.
Then the above nonlinear problem has a weak solution $u$ 
in the sense of Definition \ref{weakdefiintro}
satisfying
$$
\underline{u}\leq u\leq\overline{u}.
$$
In addition,
\begin{itemize}
\item if $f$ is increasing in the second variable,
i.e. $f(x,s)\leq f(x,t)$ whenever $s\leq t$, for all $x\in\overline{\Omega}$,
then there is a unique solution,
\item if not, there is a unique minimal solution $u_1$, that is
a solution $u_1$ such that $\underline{u}\leq u_1\leq v$ for any other supersolution $v\geq\underline{u}$.
\end{itemize} 
\end{theo}

In case our boundary datum $g$ is a nonnegative bounded function,
then Theorem \ref{nl-cs} provides a unique solution,
since we may consider $\overline{u}=\sup g$
and $\underline{u}=0$.
Then we attack directly the problem with unbounded boundary values,
and we are especially interested in those data
exploding on $\partial\Omega$.
The existence of large $s$-harmonic functions
turns out to be the key ingredient to prove all the following theorems,
that is,

\begin{theo}[Construction of large solutions]\label{large-building}
Let $f:\Omega\times\R\rightarrow\R$ be a function 
satisfying f.1) and f.2). Then there exist $u,\,v:\R^n\rightarrow[0,+\infty]$
such that
$$
\left\lbrace\begin{array}{l}
\Ds u= -f(x,u) \quad \hbox{ in }\Omega \\ \displaystyle
\lim_{\stackrel{\hbox{\scriptsize $x\!\rightarrow\!\partial\Omega$}}{x\in\Omega}}u(x)=+\infty,
\end{array}\right.
\qquad\hbox{ and }\qquad
\left\lbrace\begin{array}{l}
\Ds v= f(x,v) \quad \hbox{ in }\Omega \\ \displaystyle
\lim_{\stackrel{\hbox{\scriptsize $x\!\rightarrow\!\partial\Omega$}}{x\in\Omega}}v(x)=+\infty.
\end{array}\right.
$$
\end{theo}

Depending on the nature of the nonlinearity $f$
one can be more precise about the Dirichlet values of $u$.
Namely,

\begin{theo}[Damping term]\label{-sign}
Let $f:\Omega\times\R\rightarrow\R$ be a function 
satisfying f.1) and f.2), and
$g:\R^n\setminus\overline{\Omega}\rightarrow[0,+\infty]$ a measurable function
satisfying \eqref{gintro};
let also be $h\in C(\partial\Omega)$, $h\geq 0$.
The semilinear problem 
$$
\left\lbrace\begin{array}{ll}
\Ds u(x)=-f(x,u(x)) & \hbox{ in }\Omega \\
u=g & \hbox{ in }\R^n\setminus\overline{\Omega}\\
Eu=h & \hbox{ on }\partial\Omega
\end{array}\right.
$$
satisfies the following:
\begin{description}
\item[\it i)] if $h\equiv 0$, the equation has a weak solution for any admissible $g$,
\item[\it ii)] if $h\not\equiv 0$ then
\begin{itemize}
\item the problem has a weak solution if there exist $a_1,\,a_2\geq0$ and $p\in[0,\frac{1+s}{1-s})$
such that 
$$
f(x,t)\leq a_1+a_2{t}^p,\quad \hbox{ for }t>0,
$$
\item the problem doesn't admit any weak solution if there exist $b_1,T>0$
such that 
$$
f(x,t)\geq b_1{t}^{\frac{1+s}{1-s}},\quad \hbox{ for }t>T.
$$
\end{itemize}
\end{description}
If, in addition, $f$ is increasing in the second variable then the problem admits 
only one positive solution.
\end{theo}

\begin{theo}[Sublinear source]\label{sublinear}
Let $f:\Omega\times\R\rightarrow\R$ be a function 
satisfying f.1) and f.2), and
$g:\R^n\setminus\overline{\Omega}\rightarrow[0,+\infty]$ a measurable function
satisfying \eqref{gintro};
let also be $h\in C(\partial\Omega)$, $h\geq 0$.
Suppose also that 
$$
f(x,t)\leq\Lambda(t),\quad\hbox{ for all }x\in\overline{\Omega},\,t\geq 0
$$ 
where $\Lambda(t)$ is concave and $\Lambda'(t)\xrightarrow[]{t\uparrow\infty}0$.
Then there exists a positive weak solution $u$
to the semilinear problem 
$$
\left\lbrace\begin{array}{ll}
\Ds u(x)=f(x,u(x)) & \hbox{ in }\Omega \\
u=g & \hbox{ in }\R^n\setminus\overline{\Omega}\\
Eu=h & \hbox{ on }\partial\Omega.
\end{array}\right.
$$
\end{theo}

\begin{theo}[Superlinear source]\label{superlinear}
Let $f:\Omega\times\R\rightarrow\R$ be a function 
satisfying f.1) and f.2).
For $0<\beta<1-s$, consider problems
$$
(\star)\ \left\lbrace\begin{array}{ll}
\Ds u(x)=\lambda f(x,u(x)) & \hbox{ in }\Omega \\
u(x)={\delta(x)}^{-\beta} & \hbox{ in }\R^n\setminus\overline{\Omega},\\
Eu=0 & \hbox{ on }\partial\Omega.
\end{array}\right.\quad\quad
(\star\star)\ \left\lbrace\begin{array}{ll}
\Ds u(x)=\lambda f(x,u(x)) & \hbox{ in }\Omega \\
u(x)=0 & \hbox{ in }\R^n\setminus\overline{\Omega},\\
Eu=1 & \hbox{ on }\partial\Omega.
\end{array}\right.
$$

\underline{Existence}. If
there exist $a_1,\,a_2,\,T>0$ and $p\geq 1$, such that
$$
f(x,t)\leq a_1+a_2\,{t}^p,\qquad x\in\Omega,\ t>T.
$$
and $p\beta<1+s$, then there exists $L_1>0$ depending on $\beta$ and $p$ such that
problem $(\star)$ admits a weak solution $u\in L^1(\Omega)$ for any $\lambda\in[0,L_1]$.
Similarly, if $p(1-s)<1+s$, then there exists $L_2>0$ depending on $p$ such that
problem $(\star\star)$ admits a weak solution $u\in L^1(\Omega)$ for any $\lambda\in[0,L_2]$.

\underline{Nonexistence}. If
there exist $,b,\,T>0$ and $q>0$, such that
$$
b\,{t}^q\leq f(x,t),\qquad x\in\Omega,\ t>T.
$$
and $q\beta>1+s$, then problem $(\star)$ admits a weak solution only for $\lambda=0$.
Similarly, if $q(1-s)>1+s$, then problem $(\star\star)$ admits a weak solution only for $\lambda=0$.
\end{theo}

We finally note that, with the definition of weak solution we are dealing with, 
the nonexistence of a weak solution implies \it complete blow-up\rm, meaning that:

\begin{defi} If for any nondecreasing sequence ${\{f_k\}}_{k\in\N}$
of bounded functions such that $f_k\uparrow f$ pointwisely as $k\uparrow+\infty$,
and any sequence ${\{u_k\}}_{k\in\N}$ of positive solutions to 
$$
\left\lbrace\begin{array}{ll}
\Ds u_k=f_k(x,u_k) & \hbox{ in }\Omega \\
u_k=g & \hbox{ in }\R^n\setminus\overline{\Omega} \\
Eu_k=h & \hbox{ on }\partial\Omega, 
\end{array}\right.
$$
there holds 
$$
\lim_{k\uparrow+\infty}\frac{u_k(x)}{{\delta(x)}^s}=+\infty,\quad
\hbox{ uniformly in }\,x\in\Omega,
$$
then we say there is complete blow-up.
\end{defi}

\begin{theo}\label{complete-blowup} Let $f:\Omega\times\R\rightarrow\R$ be a function 
satisfying f.1) and f.2)
and $g:\R^n\setminus\overline{\Omega}\rightarrow[0,+\infty]$ a measurable function
satisfying \eqref{gintro}; let also be 
$h\in C(\partial\Omega)$, $h\geq 0$ and $\lambda\geq 0$.
If there is no weak solution to 
$$
\left\lbrace\begin{array}{ll}
\Ds u=f(x,u) & \hbox{ in }\Omega \\
u=g & \hbox{ in }\R^n\setminus\overline{\Omega} \\
Eu=h & \hbox{ on }\partial\Omega, 
\end{array}\right.
$$
then there is complete blow-up.
\end{theo}

\subsection*{Notations}

In the following we will always use the following notations:
\begin{description}
\item $\C\Omega$, when $\Omega\subseteq\R^n$ is open, for $\R^n\setminus\overline{\Omega}$,
\item $\delta(x)$ for $\dist(x,\partial\Omega)$ once $\Omega\subseteq\R^n$ has been fixed,
\item $\mathcal{M}(\Omega)$, when $\Omega\subseteq\R^n$, for the space of measures on $\Omega$,
\item $\mathcal{H}$, for the $n-1$ dimensional Hausdorff measure, 
dropping the ``$n-1$'' subscript whenever there is no ambiguity,
\item $f\wedge g$, when $f,g$ are two functions, for the function $\min\{f,g\}$,
\item $C^{2s+\eps}(\Omega)=\{v:\R^n\rightarrow\R\hbox{ and for any }K\hbox{ compactly supported
in }\Omega, \hbox{ there exists }
\alpha=\alpha(K,v)\hbox{ such that }v\in C^{2s+\alpha}(K) \}$.
\end{description}

\section{A mean value formula}\label{meanvalue-sec}

Definition \ref{sharm-def} of
$s$-harmonicity turns out to be equivalent to 
have a null fractional Laplacian.
Since we couldn't find a precise reference for this,
we provide a proof.
Indeed, on the one hand we have that
any function $u$ which is $s$-harmonic in an open set $\Omega$ solves
$$
\Ds u(x)=0\qquad\hbox{ in }\Omega,
$$
indeed condition \eqref{sharm-eq}
can be rewritten, using \eqref{cns2},
$$
\int_{\C B_r(x)}\frac{u(y)-u(x)}{|y-x|^n\,\left(|y-x|^2-r^2\right)^s}\;dy =0
\qquad \hbox{ for any }r\in(0,\dist(x,\partial\Omega))
$$
and therefore
$$
0=\lim_{r\downarrow 0}\int_{\C B_r(x)}\frac{u(y)-u(x)}{|y-x|^n\,\left(|y-x|^2-r^2\right)^s}\;dy
=PV \int_{\R^n}\frac{u(y)-u(x)}{|y-x|^{n+2s}}\;dy
=-\frac{\Ds u(x)}{\A(n,s)}.
$$
Indeed, by dominated convergence, far from $x$ it is
$$
\int_{\C B_1(x)}\frac{u(y)-u(x)}{|y-x|^n\,\left(|y-x|^2-r^2\right)^s}\;dy
\xrightarrow{r\downarrow 0}\int_{\C B_1(x)}\frac{u(y)-u(x)}{|y-x|^{n+2s}}\;dy.
$$
Now, any function $u$ $s$-harmonic in $\Omega$ is smooth in $\Omega$:
this follows from the representation through the Poisson kernel on balls,
given in Theorem \ref{pointwise},
and the smoothness of the Poisson kernel,
see formula \eqref{pois-ball} below.
Since $u$ is a smooth function, a Taylor expansion when $|y-x|<1$
$$
u(y)-u(x)=\langle\,\nabla u(x),y-x\,\rangle+\theta(y-x),
\quad\hbox{where}\quad|\theta(y-x)|\leq C\,|y-x|^2
$$
implies
\begin{eqnarray}
 & & \left|\int_{B(x)\setminus B_r(x)}\frac{u(y)-u(x)}{|y-x|^n\,\left(|y-x|^2-r^2\right)^s}\;dy-
\int_{B(x)\setminus B_r(x)}\frac{u(y)-u(x)}{\left|y-x\right|^{n+2s}}\;dy\right| \nonumber \\
 & & \leq \int_{B(x)\setminus B_r(x)}\frac{|\theta(y-x)|}{{|y-x|}^n}\left(\frac{1}{\left(|y-x|^2-r^2\right)^s}-\frac{1}{\left|y-x\right|^{2s}}\right)\;dy 
\nonumber \\
 & & \leq C\int_{B(x)\setminus B_r(x)}\frac{1}{{|y-x|}^{n-2+2s}}\left[\left(1-\frac{r^2}{|y-x|^2}\right)^{-s}-1\right]\;dy \nonumber \\
 & & \leq C\,s\int_{B(x)\setminus B_r(x)}\frac{r^2}{{|y-x|}^{n+2s}}\;dy=C\,s\,r^2\int_r^1\rho^{-1-2s}\;d\rho\leq \frac{C{r}^{2-2s}}{2}
\xrightarrow[r\downarrow 0]{}0. \nonumber
\end{eqnarray}

\begin{theo}\label{meanvalue} Let $u:\R^n\rightarrow\R$ a measurable function such that 
for some open set $\Omega\subseteq\R^n$ is $u\in C^{2s+\eps}(\Omega)$.
Also, suppose that
$$
\int_{\R^n}\frac{|u(y)|}{{1+|y|}^{n+2s}}\;dy<+\infty.
$$
Then for any $x\in\Omega$ and $r>0$ such that $\overline{B_r(x)}\subseteq\Omega$, one has
\begin{equation}\label{meanform}
u(x)\ =\ \int_{\C B_r}\eta_r(y)\:u(x-y)\;dy + \gamma(n,s,r)\Ds u(z), 
\quad\quad
\gamma(n,s,r)=\frac{\Gamma(n/2)}{2^{2s}\,\Gamma\left(\frac{n+2s}{2}\right)\,\Gamma(1+s)}\;r^{2s}
\end{equation}
for some $z=z(x,s)\in\overline{B_r(x)}$.
\end{theo}

\begin{tfa}\rm 
Suppose, without loss of generality, that $x=0$.
Let $v=\Gamma_s-H$, where
$\Gamma_s$ is the fundamental solution\footnote{one possible construction 
and the explicit expression of the fundamental solution
can be found in \cite[paragraph 2.2]{extension}} of the fractional Laplacian 
$$
\Gamma_s(x)=\frac{-\A(n,-s)}{{|x|}^{n-2s}}
$$
and $H$ solves in the pointwise sense
$$
\left\lbrace\begin{array}{ll}
\Ds H=0 & \hbox{ in }B_r \\
H=\Gamma_s & \hbox{ in }\C B_r.
\end{array}\right.
$$
We claim that $H$ satisfies equality
$$
H(x)=H_1(x):=\int_{\R^n}\Gamma_s(x-y)\,\eta_r(y)\;dy=(\Gamma_s*\eta_r)(x).
$$
Indeed
\begin{enumerate}
\item $H_1$ is $s$-harmonic in $B_r$ because $H_1=\Gamma_s*\eta_r$,
$\Ds\Gamma_s=\delta_0$ in the sense of distributions
and $\eta_r=0$ in $B_r$,
\item since (cf. Appendix in \cite[p. 399ss]{landkof})
$$
\int_{\R^n}\Gamma_s(x-y)\,\eta_r(y)\;dy=
\int_{\C B_r}\Gamma_s(x-y)\,\eta_r(y)\;dy=
\Gamma_s(x),\qquad |x|>r,
$$
then $H_1=\Gamma_s$ in $\C B_r$, as desired.
\end{enumerate}
Finally, as in 1., note that 
$$
\Ds H(x)=\eta_r(x),\quad\hbox{when }|x|>r.
$$
Since $u\in C^{2s+\eps}(\Omega)$, $\Ds u\in C(\Omega)\subseteq C(\overline{B_r})$,
see \cite[Proposition 2.4]{silvestre} or Lemma \ref{lem-H} below.
Mollify $u$ in order to obtain a sequence $\{u_j\}_j\subseteq C^\infty(\R^n)$.
Hence,
\begin{multline}\label{73}
\int_{B_r}v\cdot\Ds u_j=\int_{B_r}\Gamma_s\cdot\Ds u_j-\int_{B_r}H\cdot\Ds u_j=\\=
\int_{\R^n}\Gamma_s\cdot\Ds u_j-\int_{\R^n}H\cdot\Ds u_j
=u_j(0)-\int_{\C B_r}u_j\,\eta_r,
\end{multline}
where we have used the integration by parts formula \eqref{integrbypartsform}
and the definition of $\Gamma_s$.
On the one hand we have now that $\Ds u_j\xrightarrow[]{j\uparrow+\infty}\Ds u$ uniformly in $B_r$, since
\begin{multline*}
\sup_{x\in B_r}\left|\Ds v_j(x)-\Ds v(x)\right|=\sup_{x\in B_r}\left|\A(n,s)\,PV\int_{\R^n}\frac{v_j(x)-v_j(y)-v(x)+v(y)}{{|x-y|}^{n+2s}}\;dy\right|\leq\\
\leq \A(n,s)\Arrowvert u_j-u\Arrowvert_{C^{2s+\eps}(B_{r+\delta})}
\sup_{x\in B_r}\int_{B_{r+\delta}}\frac{dy}{{|x-y|}^{n-\eps}}\;+\A(n,s)\,\sup_{x\in B_r}
\int_{\R^n\setminus B_{r+\delta}}\frac{2\,\Arrowvert v_k-v\Arrowvert_{L^\infty(\Omega_k)}}{{|x-y|}^{n+2s}}\;dy,
\end{multline*}
while on the other hand
$$
u_j(0)\xrightarrow[]{j\uparrow+\infty} u(0),\qquad
\int_{\C B_r}u_j\,\eta_r\xrightarrow[]{j\uparrow+\infty}\int_{\C B_r}u\,\eta_r
$$
so that we can let $j\rightarrow+\infty$ in equality \eqref{73}.
Collecting the information so far, we have
\begin{equation}\label{mv-first}
u(0)=
\int_{\C B_r} u\:\eta_r + \int_{B_r} v\;\Ds u=
\int_{\C B_r} u\:\eta_r + \Ds u(z)\cdot\int_{B_r}v
\end{equation}
for some $|z|\leq r$, by continuity of $\Ds u$ in $\overline{B_r}$
and since $v>0$ in $\overline{B_r}$.
The constant $\gamma(n,s,r)$ appearing in the statement 
equals to $\int_{B_r}v$.

Let us compute $\gamma(n,s,r)=\int_{B_r}v$.
If we consider the solution $\varphi_\delta$ to
$$
\left\lbrace\begin{array}{ll}
\Ds\varphi_\delta=1 & \hbox{ in }B_{r+\delta}, \\
\varphi_\delta=0 & \hbox{ in }\C B_{r+\delta}.
\end{array}\right.
$$
and we apply formula \eqref{mv-first} to $\varphi_\delta$ in place of $u$ to entail
$$
\varphi_\delta(0)=
\int_{\C B_r} \varphi_\delta\,\eta_r + \Ds \varphi_\delta(z)\cdot\int_{B_r}v
=\int_{B_{r+\delta}\setminus B_r} \varphi_\delta\,\eta_r +\int_{B_r}v
$$
The solution $\varphi_\delta$ is explicitly known (see \cite[equation (1.4)]{rosserra} and references therein) 
and given by
$$
\varphi_\delta(x)=\frac{\Gamma(n/2)}{2^{2s}\,\Gamma\left(\frac{n+2s}{2}\right)\,\Gamma(1+s)}
\left((r+\delta)^2-|x|^2\right)^s.
$$
Hence, by letting $\delta\downarrow 0$,
$$
\gamma(n,s,r)=\int_{B_r}v=\lim_{\delta\downarrow 0}\varphi_\delta(0)=
\frac{\Gamma(n/2)}{2^{2s}\,\Gamma\left(\frac{n+2s}{2}\right)\,\Gamma(1+s)}\;r^{2s}.
$$
\end{tfa}

\begin{rmk}\rm
The asymptotics as $s\uparrow1$ of \eqref{meanform}
are studied in Appendix \ref{meanform-app}.
\end{rmk}

\section{Linear theory for the fractional Dirichlet problem}\label{liner-sec}

Assume $\Omega\subseteq\R^n$ is open and bounded, with $C^{1,1}$ boundary.

\subsection{Preliminaries on fractional Green functions, Poisson kernels and Martin kernels}

Consider the function $G_\Omega:\Omega\times\R^n\rightarrow\R$
built as the family of solutions to the problems
$$
\forall\;x\in\Omega\qquad
\left\lbrace\begin{array}{ll}
\Ds G_\Omega(x,\cdot)=\delta_x & \hbox{ in }\Omega, \\
G_\Omega(x,y)=0 & \hbox{ in }\C\Omega.
\end{array}\right.
$$
This function can be written as the sum
$$
G_\Omega(x,y)=\Gamma_s(y-x)-H(x,y),
$$
where $\Gamma_s$ is the fundamental solution to the fractional 
Laplacian,
and 
$$
\forall\;x\in\Omega\qquad
\left\lbrace\begin{array}{ll}
\Ds H(x,\cdot)=0 & \hbox{ in }\Omega, \\
H(x,y)=\Gamma_s(y-x) & \hbox{ in }\C\Omega.
\end{array}\right.
$$

\begin{lem}\label{lem-H} Fix $x\in\Omega$. Then $H(x,\cdot)\in C^{2s+\eps}(\Omega)\cap C(\overline{\Omega})$.
\end{lem}
\begin{tfa}\rm
Take $r=r(x)>0$ such that $\overline{B_r(x)}\subseteq\Omega$
and $\kappa$ a cutoff function
$$
0\leq \kappa\leq 1\hbox{ in }\R^n, \quad
\kappa\in C^\infty(\R^n), \quad
\kappa=1 \hbox{ in }\overline{\C\Omega}, \quad
\kappa=0 \hbox{ in }\overline{B_r(x)}
$$
and define $\gamma_s(y):=\kappa(y)\Gamma_s(y-x)$:
$\gamma_s\in C^\infty(\R^n),\gamma_s=\Gamma_s$ in $\C\Omega$, $\gamma_s=0$ in $B_r(x)$.
Then establish the equivalent problem
$$
\forall\;x\in\Omega\qquad
\left\lbrace\begin{array}{ll}
\Ds H(x,\cdot)=0 & \hbox{ in }\Omega \\
H(x,y)=\gamma_s(y-x) & \hbox{ in }\C\Omega
\end{array}\right.
$$
and by setting $h(x,y):=H(x,y)-\gamma_s(x,y)$ we obtain
$$
\forall\;x\in\Omega\qquad
\left\lbrace\begin{array}{ll}
\Ds h(x,\cdot)=-\Ds\gamma_s & \hbox{ in }\Omega, \\
h(x,y)=0 & \hbox{ in }\C\Omega.
\end{array}\right.
$$
Note that $\Ds\gamma_s\in L^\infty(\R^n)\cap C^\infty(\R^n)$ (see \cite[Proposition 2.7]{silvestre}):
from this we deduce that $h(x,\cdot)\in C^{2s+\eps}(\Omega)\cap C^s(\R^n)$
and then also $H(x,\cdot)\in C^{2s+\eps}(\Omega)\cap C^s(\R^n)$.
\end{tfa}

\begin{lem} The function $G_\Omega(x,y)$ we have just obtained
satisfies the following properties:
\begin{description}
\item[\it i)] $G_\Omega$ is continuous in $\Omega\times\Omega$ except on the diagonal,
where its singularity is inherited by the singularity in $0$ of $\Gamma_s$,
\item[\it ii)] $\Ds G_\Omega(x,\cdot)\in L^1(\C\Omega)$ for any $x\in\Omega$,
\item[\it iii)] for any $u\in C^{2s+\eps}(\Omega)\cap L^\infty(\R^n)$
and $x\in\Omega$
\begin{equation}\label{green-repr}
u(x)=\int_\Omega\Ds u(y)\:G_\Omega(x,y)\;dy-\int_{\C\Omega}u(y)\:\Ds G_\Omega(x,y)\;dy,
\end{equation}
\item[\it iv)] $\Ds G_\Omega(x,y)$, $x\in\Omega,\,y\in\C\Omega$
is given by the formula
\begin{equation}\label{green-pois}
\Ds G_\Omega(x,y)=
-\A(n,s)\,\int_{\Omega}\frac{G_\Omega(x,z)}{|z-y|^{n+2s}}\;dz,
\end{equation}
\item[\it v)] $G_\Omega(x,y)\geq 0$ for a.e. $(x,y)\in\Omega\times\Omega$, $x\neq y$,
and $\Ds G_\Omega(x,y)\leq 0$ for any $x\in\Omega,\,y\in\C\Omega$,
\item[\it vi)] it is
\begin{eqnarray}
\Gamma_s(y-x)=-\int_{\C\Omega}\Gamma_s(y-z)\cdot\Ds G_\Omega(x,z)\;dz
\qquad \hbox{for }x\in\Omega\hbox{ and }y\in\C\Omega, \label{gammas-eq}\\
G_\Omega(x,y)-\Gamma_s(y-x)=-\int_{\C\Omega}\Gamma_s(y-z)\cdot\Ds G_\Omega(x,z)\;dz
\qquad \hbox{for }x\in\Omega\hbox{ and }y\in\Omega. 
\end{eqnarray}
\end{description}
\end{lem}
\begin{tfa}\rm We prove all conclusions step by step.

\it Proof of ii). \rm
First of all, we use the estimate
$$
|h(x,y)|\leq C\Arrowvert\Ds\gamma_s(x,\cdot)\Arrowvert_\infty\;\delta(y)^s,
\hbox{ where }h\hbox{ solves }\forall\;x\in\Omega\quad
\left\lbrace\begin{array}{ll}
\Ds h(x,\cdot)=-\Ds\gamma_s & \hbox{ in }\Omega \\
h(x,y)=0 & \hbox{ in }\C\Omega
\end{array}\right.
$$
and $\gamma_s$ is a regularization of $\Gamma_s$ as in Lemma \ref{lem-H};
for the inequality we refer to \cite[Proposition 1.1]{rosserra}. We deduce that,
for $y$ sufficiently close to $\partial\Omega$, it is
$$
|G_\Omega(x,y)|=|\Gamma_s(y-x)-H(x,y)|=|\gamma_s(y-x)-H(x,y)|=|h(x,y)|
\leq C\Arrowvert\Ds\gamma_s(x,\cdot)\Arrowvert_\infty \delta(y)^s.
$$
Then, for $y\in\C\Omega$,
\begin{equation}\label{est-pois}
\begin{split}
|\Ds G_\Omega(x,y)|=\left|\A(n,s)\,PV\int_{\R^n}\frac{G_\Omega(x,z)-G_\Omega(x,y)}{|y-z|^{n+2s}}\;dz\right|
=\qquad\qquad\qquad\\
=\A(n,s)\,\int_{\Omega}\frac{G_\Omega(x,z)}{|y-z|^{n+2s}}\;dz\leq
C\int_{\Omega}\frac{\delta(z)^s}{|y-z|^{n+2s}}\;dz
\end{split}
\end{equation}
and
\begin{equation}
\begin{split}
\int_{\C\Omega}|\Ds G_\Omega(x,y)|\;dy
\leq C\int_{\C\Omega}\int_{\Omega}\frac{\delta(z)^s}{|y-z|^{n+2s}}\;dz\;dy\leq\qquad\qquad\qquad\\
\leq C\int_{\Omega}\delta(z)^s\int_{\delta(z)}^{+\infty}\rho^{-1-2s}\;d\rho\;dz\leq
\frac{C}{2s}\int_\Omega\frac{dz}{\delta(z)^s}<+\infty.
\end{split}
\end{equation}

\it Proof of iii). \rm 
The function $G_\Omega(x,y)=\Gamma_s(y-x)-H(x,y)$ can be 
integrate by parts against any $u\in C^{2s+\eps}(\R^n\setminus\Omega)\cap L^\infty(\R^n)$,
since, according to Lemma \ref{lem-H}, $H(x,\cdot)\in C^{2s+\eps}_{loc}(\Omega)\cap C(\overline{\Omega})$ 
and so we can apply Proposition \ref{integrbypartsform}. Hence,
\begin{equation*}
u(x)=\int_\Omega\Ds u(y)\:G_\Omega(x,y)\;dy-\int_{\C\Omega}u(y)\:\Ds G_\Omega(x,y)\;dy.
\end{equation*}
This formula is a Green's representation formula:
$G_\Omega$ is the Green function
while its fractional Laplacian is the Poisson kernel.

\it Proof of iv). \rm We point how the computation 
of $\Ds G_\Omega(x,y)$, $x\in\Omega,\,y\in\C\Omega$
reduces to a more readable formula:
\begin{equation*}
-\Ds G_\Omega(x,y)=\A(n,s)\,PV\int_{\R^n}\frac{G_\Omega(x,z)-G_\Omega(x,y)}{|z-y|^{n+2s}}\;dz
=\A(n,s)\,\int_{\Omega}\frac{G_\Omega(x,z)}{|z-y|^{n+2s}}\;dz,
\end{equation*}
by simply recalling that $G_\Omega(x,z)=0$, when $z\in\C\Omega$.

\it Proof of v). \rm $G_\Omega(x,y)\geq 0$ for $(x,y)\in\Omega\times\Omega$, $x\neq y$,
in view of Lemma \ref{maxprinc} below applied to the function $H(x,\cdot)$.
Also, from this and \eqref{green-pois} we deduce that
\begin{equation}\label{pois-pos}
-\Ds G_\Omega(x,y)=\A(n,s)\,
\int_{\Omega}\frac{G_\Omega(x,z)}{|z-y|^{n+2s}}\;dz\geq 0.
\end{equation}

\it Proof of vi). \rm It suffices to apply \eqref{green-repr}
to the solution $H(x,y)$ of
$$
\left\lbrace\begin{array}{ll}
\Ds H(x,\cdot)=0 & \hbox{ in }\Omega \\
H(x,y)=\Gamma_s(y-x) & \hbox{ in }\C\Omega
\end{array}\right.
$$
to infer 
\begin{eqnarray*}
 & \displaystyle\Gamma_s(y-x)\ =\ -\int_{\C\Omega}\Gamma_s(y-z)\cdot\Ds G_\Omega(x,z)\;dz
\qquad \hbox{for }x\in\Omega\hbox{ and }y\in\C\Omega, & \\
 & \displaystyle H(x,y)\ =\ -\int_{\C\Omega}\Gamma_s(y-z)\cdot\Ds G_\Omega(x,z)\;dz
\qquad \hbox{for }x\in\Omega\hbox{ and }y\in\Omega. &
\end{eqnarray*}
\end{tfa}

\begin{lem} For any $x\in\Omega,\,\theta\in\partial\Omega$ the function
$$
M_\Omega(x,\theta)=\lim_{\stackrel{\hbox{\scriptsize $y\in\Omega$}}{y\rightarrow\theta}}
\frac{G_\Omega(x,y)}{{\delta(y)}^s}
$$
is well-defined in $\Omega\times\partial\Omega$ and
for any $h\in C(\partial\Omega),\,\psi\in C^\infty_c(\Omega)$ one has
$$
\int_\Omega\left(\int_{\partial\Omega}M_\Omega(x,\theta)\,h(\theta)\;d\mathcal{H}(\theta)\right)\psi(x)\;dx
=\int_{\partial\Omega}h(\theta)\,D_s\phi(\theta)\;d\mathcal{H}(\theta),
$$
where
$$
D_s\phi(\theta)=\lim_{\stackrel{\hbox{\scriptsize $y\in\Omega$}}{y\rightarrow\theta}}
\frac{\phi(y)}{{\delta(y)}^s}
\qquad\hbox{ and }\qquad
\phi(y)=\int_\Omega G_\Omega(x,y)\,\psi(x)\;dx.
$$
\end{lem}
\begin{tfa}\rm
For a small parameter $\eps>0$ define $\Omega_\eps=\{x\in\overline{\Omega}:0\leq\delta(x)<\eps\}$:
associate at any $x\in\Omega$ a couple $(\rho,\theta)$ where $\rho=\delta(x)$
and $\theta\in\partial\Omega$ satisfies $|x-\theta|=\rho$: such a $\theta$
is uniquely determined for small $\eps$ since $\partial\Omega\in C^{1,1}$.
Take also $\varphi\in C^\infty(\R)$, with $\varphi(0)=1$ and supported in $[-1,1]$.
With a slight abuse of notation define
\begin{equation}\label{h-approx}
f_\eps(x)=f_\eps(\rho,\theta)=h(\theta)\,\frac{\varphi(\rho/\eps)}{K_\eps},
\quad K_\eps=\frac{1}{1+s}\int_0^\eps\varphi(r/\eps)\;dr.
\end{equation}
Consider the functions
$$
u_\eps(x)=\int_{\Omega}f_\eps(y)\,G_\Omega(x,y)\;dy,
$$
and any function $\psi\in C^\infty_c(\Omega)$:
\begin{eqnarray*}
\displaystyle \int_\Omega u_\eps\psi & = & 
\int_{\Omega}f_\eps(y)\left[\int_\Omega\psi(x)\,G_\Omega(x,y)\;dx\right]\;dy = 
\int_{\Omega_\eps}h(\theta(y))\,\frac{\varphi(\rho(y)/\eps)}{K_\eps}\left[\int_\Omega\psi(x)\,G_\Omega(x,y)\;dx\right]\;dy \\
& = & \int_{\Omega_\eps}h(\theta(y))\,\frac{\varphi(\rho(y)/\eps)\,{\delta(y)}^s}{K_\eps}
\left[\int_\Omega\psi(x)\,\frac{G_\Omega(x,y)}{{\delta(y)}^s}\;dx\right]\;dy.
\end{eqnarray*}
Note that the function $\psi_1(y)=\int_\Omega\psi(x)\,G_\Omega(x,y)\,{\delta(y)}^{-s}\;dx$ is 
continuous in $\overline{\Omega_\eps}$, as a consequence of the boundary
behaviour of $G_\Omega$, see \cite[equation (2.13)]{chen}.
Then
\begin{multline*} 
\int_{\Omega_\eps}h(\theta(y))\,\frac{\varphi(\rho(y)/\eps)\,{\delta(y)}^s}{K_\eps}\,\psi_1(y)\;dy=
\frac{1}{(1+s)\,K_\eps}\int_0^\eps\varphi(r/\eps)\left(\int_{\{{\delta}^{1+s}=r\}}h(\theta)\,\psi_1(r,\theta)\;d\mathcal{H}(\theta)\right)\;dr \\
\xrightarrow{\eps\downarrow0}\int_{\partial\Omega}h(\theta)\,\psi_1(0,\theta)\;d\mathcal{H}(\theta)
=\int_{\partial\Omega}h(\theta)\int_\Omega \psi(x)\lim_{y\rightarrow\theta}\frac{G_\Omega(x,y)}{{\delta(y)}^{s}}
\;dx\;d\mathcal{H}(\theta)
\end{multline*}
where the limit has been computed using the coarea formula.
Hence
\begin{equation}\label{000}
\int_\Omega u_\eps\psi\xrightarrow{\eps\downarrow0}\int_\Omega
\left(\int_{\partial\Omega}M_\Omega(x,\theta)\,h(\theta)\;d\mathcal{H}(\theta)\right)\psi(x)\;dx
\end{equation}
where
$$
M_\Omega(x,\theta)=\lim_{\stackrel{\hbox{\scriptsize $y\in\Omega$}}{y\rightarrow\theta}}\frac{G_\Omega(x,y)}{{\delta(y)}^{s}}.
$$
In addition,
\begin{eqnarray}
\int_\Omega u_\eps\psi & = & 
\int_{\Omega_\eps}h(\theta(y))\,\frac{\varphi(\rho(y)/\eps)}{K_\eps}\left[\int_\Omega\psi(x)\,G_\Omega(x,y)\;dx\right]\;dy 
=
\int_{\Omega_\eps}h(\theta(y))\,\frac{\varphi(\rho(y)/\eps)\,\delta(y)^s}{K_\eps\,\delta(y)^s}\,\phi(y)\;dy \nonumber \\
& \xrightarrow[]{\eps\downarrow 0} &
\int_{\partial\Omega}h(\theta)\,D_s\phi(\theta)\;d\mathcal{H}(\theta), \label{001}
\end{eqnarray}
again by applying the coarea formula as before.
So the limits \eqref{000} and \eqref{001} must coincide.
\end{tfa}

\begin{rmk}\rm The function $M_\Omega(x,\theta)$ we have just introduced
is closely related to the \it Martin kernel \rm 
$$
\widetilde M_\Omega(x,\theta)=\lim_{\stackrel{\hbox{\scriptsize $y\in\Omega$}}{y\rightarrow\theta}}
\frac{G_\Omega(x,y)}{G_\Omega(x_0,y)}.
$$
For this reason we borrow
the usual notation of the Martin kernel.
\end{rmk}

\begin{lem}\label{Eu} For any $h\in C(\partial\Omega)$ define
$$
u(x)=\int_{\partial\Omega}M_\Omega(x,\theta)h(\theta)\;d\mathcal{H}(\theta),
\qquad x\in\Omega.
$$
Then for any $\theta^*\in\partial\Omega$
$$
Eu(\theta^*):=\lim_{\stackrel{\hbox{\scriptsize $x\in\Omega$}}{x\rightarrow\theta^*}}
\frac{{\delta(x)}^{1-s}\,u(x)}{\int_{\partial\Omega} M_\Omega(x,\theta)\;d\mathcal{H}(\theta)}=h(\theta^*).
$$
\end{lem}
\begin{tfa}\rm
Denote by
$$
L(x)={\delta(x)}^{1-s}\int_{\partial\Omega} M_\Omega(x,\theta)\;d\mathcal{H}(\theta),
\qquad x\in\Omega
$$
for which we have
$$
\int_{\partial\Omega}\frac{\delta(x)}{c\,{|x-\theta|}^n}\;d\mathcal{H}(\theta)
\leq L(x)\leq
\int_{\partial\Omega}\frac{c\,\delta(x)}{{|x-\theta|}^n}\;d\mathcal{H}(\theta),
$$
i.e. $L$ is a bounded quantity.
Indeed, referring to estimates on the Green function in \cite[equation (2.13)]{chen}, we have inequalities
\begin{equation}\label{chen-martin}
\frac{{\delta(x)}^s}{c\,{|x-\theta|}^n}\leq M_\Omega(x,\theta)\leq\frac{c\,{\delta(x)}^s}{{|x-\theta|}^n},
\qquad x\in\Omega,\theta\in\partial\Omega.
\end{equation} 
Thus,
\begin{multline*}
\left|\frac{{\delta(x)}^{1-s}\,u(x)}{L(x)}-h(\theta^*)\right|=
\left|\frac{{\delta(x)}^{1-s}}{L(x)}\int_{\partial\Omega}M_\Omega(x,\theta)h(\theta)\;d\mathcal{H}(\theta)
-h(\theta^*)\frac{{\delta(x)}^{1-s}}{L(x)}\int_{\partial\Omega}M_\Omega(x,\theta)\;d\mathcal{H}(\theta)\right|\leq\\
\leq \frac{{\delta(x)}^{1-s}}{L(x)}\int_{\partial\Omega}
M_\Omega(x,\theta)\left|h(\theta)-h(\theta^*)\right|\;d\mathcal{H}(\theta)
\leq C\delta(x)\int_{\partial\Omega}\frac{\left|h(\theta)-h(\theta^*)\right|}{{|x-\theta|}^n}\;d\mathcal{H}(\theta) 
\end{multline*}
Describe $\partial\Omega$ as a graph in a neighborhood of $0$,
i.e.
$$
\Gamma\subseteq\partial\Omega\;\hbox{ open},\quad
\Gamma\ni\theta=(\theta',\phi(\theta'))\hbox{ for some }\phi:B_r'\subseteq\R^{n-1}\rightarrow\R,\ \phi(0)=\theta^*,\ \phi\in C^{1,1}(B_r(0)).
$$
Let us now write, 
$$
\int_{\Gamma}\frac{\left|h(\theta)-h(\theta^*)\right|}{{|x-\theta|}^n}\;d\mathcal{H}(\theta)
\leq \int_{\Gamma}\frac{\left|h(\theta)-h(\theta^*)\right|}
{\left[|x-\theta^*|^2+|\theta^*-\theta|^2-2\langle x-\theta^*,\theta^*-\theta\rangle\right]^{n/2}}
\;d\mathcal{H}(\theta)
$$
where
$$
2\langle x-\theta^*,\theta^*-\theta\rangle=2|x-\theta^*|\cdot|\theta^*-\theta|
\left\langle\frac{x-\theta^*}{|x-\theta^*|},\frac{\theta^*-\theta}{|\theta^*-\theta|}\right\rangle
\leq 2\mu|x-\theta^*|\cdot|\theta^*-\theta|,\quad\theta\in\Gamma
$$
for some $\mu<1$. Then
$$
\int_{\Gamma}\frac{\left|h(\theta)-h(\theta^*)\right|}{{|x-\theta|}^n}\;d\mathcal{H}(\theta)\leq
\frac{1}{(1-\mu)^{n/2}}\int_\Gamma\frac{\left|h(\theta)-h(\theta^*)\right|}
{\left[|x-\theta^*|^2+|\theta^*-\theta|^2\right]^{n/2}}\;d\mathcal{H}(\theta)
$$
Suppose without loss of generality that $\theta^*=0$ and denote by $\omega$ the modulus of continuity of $h$:
\begin{multline*}
\int_{\Gamma}\frac{\left|h(\theta)-h(\theta^*)\right|}{{|x-\theta|}^n}\;d\mathcal{H}(\theta)\leq
\frac{\sup_\Gamma\omega(|\theta|)}{(1-\mu)^{n/2}}\int_{B_r'}\frac{d\theta'}{\left[|x|^2+|\theta'|^2+|\phi(\theta')|^2\right]^{n/2}} \leq \\
\leq  
\frac{\sup_\Gamma\omega(|\theta|)}{(1-\mu)^{n/2}}\int_0^r\frac{\rho^{n-1}\;d\rho}{\left[|x|^2+\rho^2\right]^{n/2}} 
\leq
\frac{\sup_\Gamma\omega(|\theta|)}{(1-\mu)^{n/2}}\int_0^r\frac{d\rho}{|x|^2+\rho^2}
\leq\frac{\sup_\Gamma\omega(|\theta|)}{|x|\,(1-\mu)^{n/2}}
\end{multline*}
and therefore
\begin{multline*}
\delta(x)\int_{\partial\Omega}\frac{\left|h(\theta)-h(\theta^*)\right|}{{|x-\theta|}^n}\;d\mathcal{H}(\theta)=\\=
\delta(x)\int_{\Gamma}\frac{\left|h(\theta)-h(\theta^*)\right|}{{|x-\theta|}^n}\;d\mathcal{H}(\theta)+
\delta(x)\int_{\partial\Omega\setminus\Gamma}\frac{\left|h(\theta)-h(\theta^*)\right|}{{|x-\theta|}^n}\;d\mathcal{H}(\theta)
\leq\\\leq 
\frac{\sup_\Gamma\omega(|\theta|)}{(1-\mu)^{n/2}}+o(\delta(x)).
\end{multline*}
Now, since $\sup_\Gamma\omega(|\theta|)\leq\omega(\hbox{diam}\Gamma)$,
where diam$\Gamma=\sup_{\theta\in\Gamma}|\theta|$, and $\Gamma$ is arbitrary,
we deduce
$$
\delta(x)\int_{\partial\Omega}\frac{\left|h(\theta)-h(\theta^*)\right|}{{|x-\theta|}^n}\;d\mathcal{H}(\theta)
\xrightarrow[x\rightarrow\theta^*]{}0.
$$
\end{tfa}

\subsection{Linear theory for smooth data: proof of Theorem \ref{pointwise}}

We start by stating

\begin{lem}[Maximum Principle]\label{maxprinc} 
Let $u:\R^n\rightarrow\R$ be a function in
$C^{2s+\eps}(\Omega)\cap C(\overline{\Omega})$,
and
$$
\int_{\R^n}\frac{|u(y)|}{1+|y|^{n+2s}}\;dy<+\infty,\qquad\hbox{and}\qquad
\left\lbrace\begin{array}{ll}
\Ds u\leq 0 & \hbox{ in }\Omega \\
u\leq 0 & \hbox{ in }\C\Omega.
\end{array}\right.
$$
Then $u\leq 0$ in $\Omega$.
\end{lem}
\begin{tfa}\rm Call $\Omega^+=\{x\in\R^n:u>0\}$:
$\Omega^+$ is an open set contained in $\Omega$.
Assume by contradiction that $\Omega^+\neq\varnothing$.
By continuity of $u$ in $\Omega^+$, 
there exists $x_0\in\Omega^+$ such that
$u(x_0)=\max\{u(x):x\in\Omega^+\}$, but this point $x_0$
will be also a global maximum for $u$ since outside $\Omega^+$
the function $u$ is nonpositive.
Thus
$$
\Ds u(x_0)=\A(n,s)\,
PV\int_{\R^n}\frac{u(x_0)-u(y)}{|x_0-y|^{n+2s}}>0
$$
contradicting our hypotheses. Therefore $\Omega^+$ is empty.
\end{tfa}

By splitting $g$ into its positive and negative part,
it suffices to prove Theorem \ref{pointwise} in the case where $g\geq0$.
So, from now on we will deal with nonnegative boundary data $g:\C\Omega\rightarrow[0,+\infty)$
which are measurable functions with
\begin{equation}\label{g}
0\leq-\int_{\C\Omega}g(y)\cdot\Ds G_\Omega(x,y)\;dy<+\infty
\qquad x\in\Omega.
\end{equation}
Note that, in view of equation \eqref{est-pois},
\begin{equation}\label{est-pois2}
0\leq-\Ds G_\Omega(x,y)\leq C\int_{\C\Omega}\frac{\delta(y)^s}{\left|y-z\right|^{n+2s}}\;dz
\leq C\int_{\delta(y)}^{+\infty}\frac{d\rho}{\rho^{1+s}}=\frac{C}{s}\frac{1}{\delta(y)^s},
\qquad x\in\Omega,\ y\in\C\Omega.
\end{equation}
\medskip

\begin{proof}[Proof of Theorem \ref{pointwise}]
We split the proof by building the solution associated to each 
datum $f$, $g$ and $h$ separately.

\paragraph{First case: $f,\,h\equiv 0$.}

We present here a readaptation of \cite[Lemma 1.13]{landkof}.

The function $u$ defined by equation \eqref{linearsol} is 
continuous in $\Omega$ as an application of the Dominated Convergence Theorem
and inequality \eqref{est-pois2}.
The continuity up to the boundary is postponed to Paragraph \ref{cont-sharm-sec}.

For the sake of clarity we divide the proof in four steps:
for special forms of $g$, for $g$ regular enough, for $g$ bounded
and finally for any other $g$.

\it Step 1. \rm Suppose we have a measure $\nu$, such that $\nu(\Omega)=0$ and
$$
g(x)=\int_{\C\Omega}\Gamma_s(x-y)\;d\nu(y)\quad x\in\C\Omega.
$$
Then set
$$
\widetilde{u}(x)=\int_{\R^n}\Gamma_s(x-y)\;d\nu(y)\quad\forall x\in\R^n:
$$
\begin{description}
\item[-] $\widetilde{u}=g$, since $\nu$ is supported in $\C\Omega$,
\item[-] $\widetilde{u}=u$ in $\Omega$, where $u$ is given by
$$
u(x)=-\int_{\C\Omega}g(y)\,\Ds G_\Omega(x,y)\;dy,\quad x\in\Omega.
$$
Indeed,
\begin{eqnarray*}
-\int_{\C\Omega}g(y)\cdot\Ds G_\Omega(x,y)\;dy & = & 
-\int_{\C\Omega}\Ds G_\Omega(x,y)\left(\int_{\C\Omega}\Gamma_s(z-y)\;d\nu(z)\right)\;dy \\ & = & 
-\int_{\C\Omega}\left(\int_{\C\Omega}\Gamma_s(z-y)\cdot\Ds G_\Omega(x,y)\;dy\right)\;d\nu(z) \\ 
\hbox{\small [in view of equation \eqref{gammas-eq}]}
& = & 
\int_{\C\Omega}\Gamma_s(z-y)\;d\nu(z)
\quad = \quad \widetilde{u}(x)
\end{eqnarray*}
\item[-] $u$ is $s$-harmonic in $\Omega$:
\begin{eqnarray*}
\left(\eta_r*u\right)(x) & = & 
\int_{\C B_r(x)}\eta_r(y-x)\cdot u(y)\;dy \\
& = & 
\int_{\C B_r(x)}\eta_r(y-x)\left(\int_{\R^n}\Gamma_s(z-y)\;d\nu(z)\right)\;dy \\
& = & 
\int_{\R^n}\left(\int_{\C B_r(x)}\Gamma_s(z-y)\,\eta_r(y-x)\;dy\right)\;d\nu(z)
\quad = \quad \int_{\R^n}\Gamma_s(x-z)\;d\nu(z) 
\ = \ u(x)
\end{eqnarray*}
by choosing $0<r<\delta(x)\leq|z-x|$
and exploiting the $s$-harmonicity of $\Gamma_s$.
\end{description}

\it Step 2. \rm If $g\in C^\infty(\overline{\C\Omega})$ and supp$\,g$ is bounded,
then $g$ admits an extension $\tilde{g}\in C^\infty_c(\R^n)$ and (see \cite[Lemma 1.1]{landkof})
there exists an absolutely continuous measure $\nu$, with 
density $\Psi$, such that 
$$
\tilde{g}(x)=\int_{\C\Omega}\Gamma_s(x-y)\;d\nu(y)
\quad\hbox{for any $x\in\R^n$}.
$$
Denote by $\nu_\Omega$ the measure obtained by restricting $\nu$ to $\Omega$, i.e.
$\nu_\Omega(A)=\nu(A\cap\Omega)$ for any measurable $A$ and $\nu_\Omega$
has density $\Psi_\Omega=\Psi\chi_\Omega$, and 
$$
\nu_\Omega'(y)=\left\lbrace\begin{array}{ll}
\int_\Omega\Ds G_\Omega(x,y)\;d\nu_\Omega(x)
=-\int_\Omega\Ds G_\Omega(x,y)\,\Psi_\Omega(x)\;dx & y\in\C\Omega \\
0 & y\in\Omega
\end{array}\right.:
$$
the integral is well-defined because $\Psi_\Omega\in L^1(\Omega)$
while $\Ds G_\Omega(\cdot,y)\in C(\overline\Omega)$ for any fixed $y\in\C\Omega$
as a consequence of $G_\Omega(\cdot,y)\in C(\overline\Omega)$ and equation \eqref{green-pois}.
Define $\gamma=\nu-\nu_\Omega+\nu_\Omega'$ which is a measure supported in $\C\Omega$.
Then, when $x\in\C\Omega$,
$$
\int_{\R^n}\Gamma_s(x-y)\;d\nu_\Omega'(y)=
-\int_\Omega\left(\int_{\C\Omega}\Gamma_s(x-y)\,\Ds G_\Omega(z,y)\;dy\right)
\;d\nu_\Omega(z)=
\int_\Omega\Gamma_s(x-z)\;d\nu_\Omega(z)
$$ 
where we have used \eqref{gammas-eq}. Therefore
$$
\int_{\C\Omega}\Gamma_s(x-y)\;d\gamma(y)
=\int_{\C\Omega}\Gamma_s(x-y)\;d\nu(y)=\tilde{g}(x)
$$
so that we can apply the previous step of the proof.

\it Step 3. \rm For $g\in L^\infty(\C\Omega)$,
consider a sequence 
$\{g_N\}_{n\in\N}\subseteq C^\infty(\overline{\C\Omega})$ uniformly bounded 
and converging pointwisely to $g$.
The corresponding sequence of $s$-harmonic functions $u_N$ converges to $u$,
since
$$
u_N(x)=-\int_{\C\Omega}g_N(y)\,\Ds G_\Omega(x,y)\;dy
\xrightarrow[N\uparrow+\infty]{}
u(x)=-\int_{\C\Omega}g(y)\,\Ds G_\Omega(x,y)\;dy
$$
by Dominated Convergence.
Then, again by the Dominated Convergence theorem we have
$$
u_N\ =\ \eta_\delta*u_N\longrightarrow\eta_\delta*u,\quad \hbox{in }\Omega
$$
therefore
$$
u\ =\ \lim u_N\ =\ \eta_\delta*u,\quad \hbox{in }\Omega,
$$
i.e. $u$ is $s$-harmonic in $\Omega$.

\it Step 4. \rm For a general measurable nonnegative $g$ it suffices now to consider
an increasing sequence $g_N$ converging to $g$, e.g. $g_N=\min\{g,N\}$.
Then the corresponding sequence of $s$-harmonic functions $u_N$ converges to $u$.
Moreover, the sequence $\{u_N\}_N$ is increasing:
$$
u_{N+1}(x)=-\int_{\C\Omega}g_{N+1}(y)\,\Ds G_\Omega(x,y)\;dy\geq
-\int_{\C\Omega}g_N(y)\,\Ds G_\Omega(x,y)\;dy=u_N(x).
$$
Then, thanks to the Monotone Convergence theorem we have
$$
u_N\ =\ \eta_\delta*u_N\longrightarrow\eta_\delta*u,\quad \hbox{in }\Omega
$$
therefore
$$
u=\lim u_N=\eta_\delta*u,\quad \hbox{in }\Omega.
$$

\it Uniqueness. \rm Finally, if $g\in C(\overline{\C\Omega})$,
the solution we have built is the only solution in $C(\overline{\Omega})$
as an application of Lemma \ref{maxprinc}.


\begin{rmk}\rm Suppose to have $g:\C\Omega\rightarrow[0,+\infty)$
for which \eqref{g} fails and there is a set $\mathcal{O}\subseteq\Omega$,
$|\mathcal{O}|>0$, in which
$$
-\int_{\C\Omega}g(y)\,\Ds G_\Omega(x,y)\;dy=+\infty,\qquad x\in\mathcal{O}.
$$
It is not possible in this case to have a pointwise solution of
$$
\left\lbrace\begin{array}{ll}
\Ds u=0 & \hbox{ in }\Omega, \\
u=g & \hbox{ in }\C\Omega.
\end{array}\right.
$$
Indeed, if we set $g_N=\min\{g,N\}$, $N\in\N$, then
$g_N$ converges monotonically to $g$, and 
$$
u(x)\geq-\int_{\C\Omega}g_N(y)\,\Ds G_\Omega(x,y)\;dy
\xrightarrow{\ N\uparrow+\infty\ }+\infty
$$
for all $x\in\mathcal{O}$.
\end{rmk}

\paragraph{Second case: $g,\,h\equiv 0$.}

Use the construction of $G_\Omega$ to write for $x\in\Omega$
$$
u(x)=\int_{\Omega}f(y)\,G_\Omega(y,x)\;dy=
\int_{\Omega}f(y)\,\Gamma_s(y-x)\;dy-\int_{\Omega}f(y)\,H(y,x)\;dy:
$$
the first addend is a function $u_1(x)$ which solves 
$\Ds u_1=f\chi_\Omega$ in $\R^n$, let us turn to the second one:
\begin{eqnarray*}
u_2(x)\ :=\ \int_{\Omega}f(y)H(y,x)\;dy & = & 
-\int_{\Omega}f(y)\left[
\int_{\C\Omega}\Gamma_s(y-z)\Ds H(x,z)\;dz\right]\;dy \\
& = & -\int_{\C\Omega}\Ds H(x,z)\left[
\int_{\Omega}f(y)\Gamma_s(y-z)\;dy\right]\;dz \\
& = & -\int_{\C\Omega}\Ds H(x,z)\,u_1(z)\;dz \ =\ -\int_{\C\Omega}u_1(z)\,\Ds G_\Omega(x,z)\;dz.
\end{eqnarray*}
According to the \it Step 4 \rm above, $u_2$ solves
$$
\left\lbrace\begin{array}{ll}
\Ds u_2=0 & \hbox{ in }\Omega \\
u_2=u_1 & \hbox{ in }\C\Omega
\end{array}\right.
$$
therefore $u=u_1-u_2$ and 
\begin{equation}\label{56}
\left\lbrace\begin{array}{ll}
\Ds u=\Ds u_1-\Ds u_2=f & \hbox{ in }\Omega \\
u=u_1-u_2=0 & \hbox{ in }\C\Omega.
\end{array}\right.
\end{equation}
Finally, $u\in C^{2s+\alpha}(\Omega)$ thanks to \cite[Proposition 2.8]{silvestre},
while for inequality
$$
|u(x)|\leq C \Arrowvert f\Arrowvert_\infty\, \delta(x)^s
$$
we refer to \cite[Proposition 1.1]{rosserra}.

\begin{rmk}\rm Note that these computations give an alternative integral representation
to the one provided in equation \eqref{green-repr}
for $u$, meaning that we have both
$$
\int_{\Omega} G_\Omega(x,y)\,\Ds u(y)\;dy
\ =\ u(x)\ =\ 
\int_{\Omega} G_\Omega(y,x)\,\Ds u(y)\;dy,
$$
so we must conclude that $G_\Omega(x,y)=G_\Omega(y,x)$ for $x,y\in\Omega,\ x\neq y$.
\end{rmk}

\paragraph{Third case: $f,\,g\equiv 0$.}

The function $u(x)=\int_{\partial\Omega}M(x,\theta)\,h(\theta)\;d\mathcal{H}(\theta)$
is $s$-harmonic: to show this we use both the construction of $M(x,\theta)$
and the mean value formula \eqref{meanform}. 
Using the notations of \eqref{h-approx}, for any $0<r<\delta(x)$
there exists $z\in \overline{B_r(x)}$
$$
u_\eps(x)=\int_{\C B_r}\eta_r(y)\,u_\eps(x-y)\;dy + \gamma(n,s,r)\Ds u_\eps(z)
=\int_{\C B_r}\eta_r(y)\,u_\eps(x-y)\;dy + \gamma(n,s,r)\,f_\eps(z)
$$
where the equality $\Ds u_\eps=f_\eps$ holds throughout $\Omega$ in view of 
\eqref{56}.
Letting $\eps\downarrow 0$ we have both
$$
u_\eps(x)\xrightarrow{\eps\downarrow 0}u(x)=\int_{\partial\Omega}M(x,\theta)\,h(\theta)\;d\mathcal{H}(\theta)
$$
and 
$$
f_\eps(z)\xrightarrow{\eps\downarrow 0} 0,\qquad\hbox{ for any }z\in\Omega.
$$
This implies that we have equality
$$
u(x)=\int_{\C B_r}\eta_r(y)\,u(x-y)\;dy,
$$
i.e. $u$ is $s$-harmonic.
\end{proof}

\begin{lem}[An explicit example on the ball]\label{expl-sol}
The functions
$$
u_\sigma(x)=\left\lbrace\begin{array}{ll}
\displaystyle\frac{c(n,s)}{\left(1-|x|^2\right)^\sigma} & |x|<1 \\
\displaystyle\frac{c(n,s+\sigma)}{\left(1-|x|^2\right)^\sigma} & |x|>1 \\
\end{array}\right.\quad 0<\sigma<1-s,\quad\qquad
u_{1-s}(x)=\left\lbrace\begin{array}{ll}
\displaystyle\frac{c(n,s)}{\left(1-|x|^2\right)^{1-s}} & |x|<1 \\
0 & |x|>1
\end{array}\right.
$$
are $s$-harmonic in the ball $B=B_1(0)$, where $c(n,s)$ is given by \eqref{cns2}.
\end{lem}
\begin{tfa}\rm
According to \cite[equation (1.6.11')]{landkof} and in view to
the computations due to Riesz \cite{riesz}, the Poisson kernel
for the ball $B$ of radius $1$ and centered at $0$ has the
explicit expression
\begin{equation}\label{pois-ball}
-\Ds G_B^s(x,y)=\frac{c(n,s)}{|x-y|^n}
\left(\frac{1-|x|^2}{|y|^2-1}\right)^s,\qquad
c(n,s)=\frac{\Gamma(n/2)\,\sin(\pi s)}{\pi^{1+n/2}}.
\end{equation}
We construct here the $s$-harmonic function
induced by data
$$
h(\theta)=0,\qquad
g_\sigma(y)=\frac{c(n,s+\sigma)}{\left(|y|^2-1\right)^\sigma},\quad 0<\sigma<1-s.
$$
Indeed, it can be explicitly computed
\begin{eqnarray*}
\left(1-|x|^2\right)^\sigma u_\sigma(x) & = & 
-\left(1-|x|^2\right)^\sigma\int_{\C B}g(y)\cdot\Ds G_B^s(x,y)\;dy \\
& = & 
\int_{\C B}\frac{c(n,s)}{|x-y|^n}\cdot
\frac{{(1-|x|^2)}^{s+\sigma}}{{(|y|^2-1)}^s}\:g(y)\;dy
\ = \ 
\int_{\C B}\frac{c(n,s)\,c(n,s+\sigma)}{|x-y|^n}\cdot\frac{{(1-|x|^2)}^{s+\sigma}}{{(|y|^2-1)}^{s+\sigma}}\;dy \\
& = & 
-c(n,s)\int_{\C B}(-\Delta)^{s+\sigma}G_B^{s+\sigma}(x,y)\;dy
\quad = \quad c(n,s)
\end{eqnarray*}
therefore the function
$$
u_\sigma(x)=\left\lbrace\begin{array}{ll}
\displaystyle\frac{c(n,s)}{\left(1-|x|^2\right)^\sigma} & x\in B \\
\displaystyle\frac{c(n,s+\sigma)}{\left(1-|x|^2\right)^\sigma} & x\in\C B \\
\end{array}\right.
$$
solves the problem
$$
\left\lbrace\begin{array}{ll}
\Ds u_\sigma=0 & \hbox{ in }B \\
\displaystyle u_\sigma(x)=g_\sigma(x)=\frac{c(n,s+\sigma)}{\left(|x|^2-1\right)^\sigma} & \hbox{ in }\C B.
\end{array}\right.
$$
We are interested in letting $\sigma\rightarrow 1-s$.
Obviously,
$$
u_\sigma(x)\ \xrightarrow{\sigma\rightarrow 1-s}\ u(x)=
\left\lbrace\begin{array}{ll}
\displaystyle \frac{c(n,s)}{\left(1-|x|^2\right)^{1-s}} & \hbox{ in }B\\
\displaystyle 0 & \hbox{ in }\C B
\end{array}\right.
$$
everywhere in $\R^n\setminus\partial B$. $s$-harmonicity is preserved,
since for $x\in B$ and any $r\in(0,1-|x|)$,
\begin{multline*}
u(x)=\lim_{\sigma\rightarrow 1-s}u_\sigma(x)=
\lim_{\sigma\rightarrow 1-s}\int_{\C B_r}u_\sigma(y)\,\eta_r(x-y)\;dy= \\
=\lim_{\sigma\rightarrow 1-s}\int_{\C B_r\cap B}u_\sigma(y)\,\eta_r(x-y)\;dy+
\lim_{\sigma\rightarrow 1-s}\int_{\C B}g_\sigma(y)\,\eta_r(x-y)\;dy= \\
=\int_{\C B_r\cap B}u(y)\,\eta_r(x-y)\;dy=
\int_{\C B_r}u(y)\,\eta_r(x-y)\;dy
\end{multline*}
since $g_\sigma\xrightarrow[]{\sigma\rightarrow 1-s}0$ in $L^1(\C B)$,
while $\eta_r(\,\cdot\,-x)$ is bounded in $\C B$ for $0<r<\delta(x)$.
For any $\psi\in C^\infty_c(\Omega)$ we have
$$
\int_B u_\sigma\psi=-\int_{\C B}g_\sigma(x)\int_B \psi(z)\,\Ds G_B^s(z,x)\;dz\;dx=
-\int_{\C B}\frac{c(n,s+\sigma)}{\left(|x|^2-1\right)^\sigma}\int_B \psi(z)\,\Ds G_B^s(z,x)\;dz\;dx.
$$
Then on the one hand
$$
\int_B u_\sigma\psi\xrightarrow{\sigma\rightarrow 1-s}\int_B u\psi=c(n,s)\int_{B}\frac{\psi(x)}{\left(1-|x|^2\right)^{1-s}}\;dx.
$$
On the other hand
\begin{multline*}
 -\int_{\C B}\frac{c(n,s+\sigma)}{\left(|x|^2-1\right)^\sigma}\int_B \psi(z)\,\Ds G_B^s(z,x)\;dz\;dx = \\
=
\int_{\C B}\frac{c(n,s+\sigma)}{\left(|x|^2-1\right)^\sigma}\left[\int_B \psi(z)\frac{c(n,s)}{|z-x|^n}\left(\frac{1-|z|^2}{|x|^2-1}\right)^s\;dz\right]\;dx 
=\\=
\int_{\C B}\frac{c(n,s+\sigma)\,c(n,s)}{\left(|x|^2-1\right)^{\sigma+s}}\int_B\left[\psi(z)\frac{\left(1-|z|^2\right)^s}{|z-x|^n}\;dz\right]\;dx.
\end{multline*}
Note that the function $\int_B \psi(z)\cdot\frac{(1-|z|^2)^s}{|z-x|^n}\;dz$
is $C(\overline{\C B})$ in the $x$ variable, since $\psi\in C^\infty_c(B)$.
Splitting $x\in\C B$ in spherical coordinates, 
i.e. $x=\rho\theta$, $\rho=|x|\in(1,+\infty)$ and $|\theta|=1$, and denoting by $\phi\in C(\overline{\Omega})$
the function satisfying $\Ds\phi|_\Omega=\psi$, $\phi=0$ in $\C\Omega$,
\begin{equation}\label{678}
-c(n,s+\sigma)\int_{\C B}\frac{\Ds\phi(x)}{\left(|x|^2-1\right)^\sigma}\;dx =
 \int_1^{+\infty}\frac{c(n,s+\sigma)\,c(n,s)}{\left(\rho^2-1\right)^{\sigma+s}}\cdot\psi_1(\rho)\rho^{n-1}\;d\rho
\end{equation}
where $\psi_1(\rho)=\int_{\partial B}\int_B \psi(z)\,\frac{\left(1-|z|^2\right)^s}{|z-\rho\theta|^n}\;dz\;d\mathcal{H}(\theta)$
is continuous on $[1,+\infty)$ and 
has a decay at infinity which is comparable to that of $\rho^{-n}$.
Therefore, as $\sigma\rightarrow1-s$,
\begin{multline*}
-c(n,s+\sigma)\int_{\C B}\frac{\Ds\phi(x)}{\left(|x|^2-1\right)^\sigma}\;dx \longrightarrow
c(n,s)\,c(n,1/2)\,\frac{\psi_1(1)}{2}=\\
=\frac{c(n,s)\,c(n,1/2)}{2}\int_{\partial B}\int_B \psi(z)\cdot\frac{\left(1-|z|^2\right)^s}{|z-\theta|^n}\;dz\;d\mathcal{H}(\theta).
\end{multline*}
Indeed, 
by the definition of $c(n,s)$ in \eqref{cns}, it is 
$$
-c(n,s+\sigma)\int_1^{+\infty}\frac{\rho\;d\rho}{\left(\rho^2-1\right)^{\sigma+s}}=-\frac{c(n,s+\sigma)}{2(1-s-\sigma)}
\xrightarrow[]{\sigma\uparrow 1-s}\frac{c(n,1/2)}{2}
$$
and, since in \eqref{678} the product $\psi_1(\rho)\,\rho^{n-2}\in C([1,+\infty))$,
$$
-c(n,s+\sigma)\int_{\C B}\frac{\Ds\phi(x)}{\left(|x|^2-1\right)^\sigma}\;dx \longrightarrow
c(n,s)\,c(n,1/2)\,\frac{\psi_1(1)}{2}.
$$
So
$$
\int_B u\psi
=\frac{c(n,s)\,c(n,1/2)}{2}\int_{\partial B}\int_B \psi(z)\cdot\frac{\left(1-|z|^2\right)^s}{|z-\theta|^n}\;dz\;d\mathcal{H}(\theta),
$$
i.e. 
$$
u(x)=\frac{c(n,s)\,c(n,1/2)}{2}\int_{\partial B}\frac{\left(1-|x|^2\right)^s}{|x-\theta|^n}\;d\mathcal{H}(\theta)
$$
and indeed the kernel $M_B(x,\theta)$ for the ball is
$$
M_B(x,\theta)=
\lim_{\stackrel{\hbox{\scriptsize $y\in\Omega$}}{y\rightarrow\theta\in\partial\Omega}}\frac{G_B^s(x,y)}{{\delta(y)}^s}
=\frac{c(n,s)}{2}\cdot\frac{\left(1-|x|^2\right)^s}{|x-\theta|^n}.
$$
and so $u$ solves 
$$
\left\lbrace\begin{array}{ll}
\Ds u=0 & \hbox{ in } B \\
u=0 &\hbox{ in }\C B \\
Eu=c(n,1/2) & \hbox{ on }\partial B.
\end{array}\right.
$$
\end{tfa}

\subsection{The linear Dirichlet problem: an $L^1$ theory}

We define $L^1$ solutions for the Dirichlet problem,
in the spirit of Stampacchia \cite{stampacchia}.
A proper functional space in which to consider 
test functions is the following.

\begin{lem}[Test function space]\label{testspace}
For any $\psi\in C^\infty_c(\Omega)$ 
the solution $\phi$ of
$$
\left\lbrace\begin{array}{ll}
\Ds \phi = \psi & \hbox{ in }\Omega \\
\phi = 0 & \hbox{ in }\C\Omega \\
E\phi = 0 & \hbox{ on }\partial\Omega
\end{array}\right.
$$
satisfies the following
\begin{enumerate}
\item $\Ds\phi\in L^1(\C\Omega)$ and for any $x\in\C\Omega$
\begin{equation}\label{repr-lapl}
\Ds\phi(x)=\int_\Omega\psi(z)\,\Ds G_\Omega(z,x)\;dz,
\end{equation}
\item for any $x\in\Omega,\,\theta\in\partial\Omega$
\begin{equation}\label{repr-deriv}
\int_\Omega M_\Omega(x,\theta)\psi(x)\;dx=D_s\phi(\theta).
\end{equation}
\end{enumerate}
\end{lem}
\begin{tfa}\rm The solutions $\phi$ is given by 
$\phi(x)=\int_{\R^n}\psi(y)\,G_\Omega(y,x)\;dy\in C^{2s+\eps}(\Omega)\cap C(\overline{\Omega})$.
Also, when $x\in\C\Omega$,
\begin{multline*}
\Ds\phi(x)=-\A(n,s)\,\int_\Omega\frac{\phi(y)}{|y-x|^{n+2s}}\;dy=
-\int_\Omega\frac{\A(n,s)}{|y-x|^{n+2s}}\left[\int_\Omega\psi(z)\,G_\Omega(z,y)\;dz\right]\;dy=\\
=-\int_\Omega\psi(z)\left[\A(n,s)\,\int_\Omega\frac{G_\Omega(z,y)}{|y-x|^{n+2s}}\;dy\right]\;dz=
\int_\Omega\psi(z)\,\Ds G_\Omega(z,x)\;dz.
\end{multline*}
Thanks to the integrability of $\Ds G_\Omega(z,x)$ in $\C\Omega$,
also $\Ds \phi\in L^1(\C\Omega)$,
so that $\phi$ is an admissible function for the integration by parts formula
\eqref{intpartsprop}.
Moreover,
$$
\int_\Omega M_\Omega(x,\theta)\psi(x)\;dx=D_s\phi(\theta).
$$
Indeed,
$$
\int_\Omega M_\Omega(x,\theta)\psi(x)\;dx =
\int_\Omega \psi(x)\,
\lim_{\stackrel{\hbox{\scriptsize $z\!\in\!\Omega$}}{z\rightarrow\theta}}\frac{G_\Omega(x,z)}{{\delta(z)}^s}\;dx
= \lim_{\stackrel{\hbox{\scriptsize $z\!\in\!\Omega$}}{z\rightarrow\theta}}\int_\Omega 
\frac{\psi(x)\,G_\Omega(x,z)}{{\delta(z)}^s}\;dx =
\lim_{\stackrel{\hbox{\scriptsize $z\!\in\!\Omega$}}{z\rightarrow\theta}}\frac{\phi(z)}{{\delta(z)}^s} =
D_s\phi(\theta).
$$\end{tfa}

Then, our space of test functions will be 
$$
\T(\Omega)=\left\lbrace\phi\in C(\R^n):
\left\lbrace\begin{array}{ll}
\Ds \phi=\psi & \hbox{ in }\Omega \\
\phi=0 & \hbox{ in }\C\Omega \\
E\phi=0 & \hbox{ in }\partial\Omega
\end{array}\right.,\ 
\hbox{ when }\psi\in C^\infty_0(\Omega)\right\rbrace.
$$
Note that the map $D_s$ is well-defined from $\T(\Omega)$ to $C(\partial\Omega)$
as a consequence of the results in \cite[Theorem 1.2]{rosserra}.
We give the following

\begin{defi} Given three Radon measures $\lambda\in\mathcal{M}(\Omega)$, 
$\mu\in\mathcal{M}(\C\Omega)$ and $\nu\in\mathcal{M}(\partial\Omega)$,
we say that a function $u\in L^1(\Omega)$
is a solution of 
$$
\left\lbrace\begin{array}{ll}
\Ds u=\lambda & \hbox{ in }\Omega \\
u=\mu & \hbox{ in }\C\Omega \\
Eu=\nu & \hbox{ on }\partial\Omega
\end{array}\right.
$$
if for every $\phi\in\T(\Omega)$ it is
\begin{equation}\label{weakdefi}
\int_\Omega u(x)\Ds \phi(x)\;dx\ =\ \int_\Omega\phi(x)\;d\lambda(x)
-\int_{\C\Omega}\Ds \phi(x)\;d\mu(x)+\int_{\partial\Omega}D_s\phi(\theta)\;d\nu(\theta)
\end{equation}
where $\displaystyle D_s\phi(\theta)=\lim_{\stackrel{\hbox{\scriptsize $x\in\Omega$}}{x\rightarrow\theta}}\frac{\phi(x)}{{\delta(x)}^s}$.
\end{defi}

\begin{prop}\label{strsol-weaksol}\rm Solutions provided by Theorem \ref{pointwise}
are solutions in the above $L^1$ sense.
\end{prop}
\begin{tfa}\rm Indeed, if 
$$
u(x)=\int_\Omega f(y)\,G_\Omega(y,x)\;dy-\int_{\C\Omega}g(y)\,\Ds G_\Omega(x,y)\;dy
+\int_{\partial\Omega}M(x,\theta)\,h(\theta)\;d\mathcal{H}(\theta),
$$
then for any $\phi\in\T(\Omega)$,
\begin{eqnarray*}
& & \int_\Omega u(x)\,\Ds\phi(x)\;dx = \\ 
& & = \int_\Omega \left[\int_\Omega f(y)\,G_\Omega(y,x)\;dy-\int_{\C\Omega}g(y)\,\Ds G_\Omega(x,y)\;dy
	+\int_{\partial\Omega}M(x,\theta)\,h(\theta)\;d\mathcal{H}(\theta)\right]\,\Ds\phi(x)\;dx \\ 
& & =\int_\Omega f(y)\underbrace{\left[\int_\Omega \Ds\phi(x)\,G_\Omega(y,x)\;dx\right]}_{=\ \phi(y)}\;dy-
	\int_{\C\Omega}g(y)\underbrace{\left[\int_\Omega\Ds\phi(x)\,\Ds G_\Omega(x,y)\;dx\right]}_{=\ \Ds\phi(y)}\;dy \\
& & \quad +\int_{\partial\Omega}h(\theta)\underbrace{\left[\int_\Omega M(x,\theta)\Ds\phi(x)\right]
 	}_{=\ D_s\phi(\theta)}d\mathcal{H}(\theta)
\end{eqnarray*}
where we have used both \eqref{repr-lapl} and \eqref{repr-deriv}.
\end{tfa}

\begin{proof}[Proof of Theorem \ref{existence-weak2}]
In case $\nu=0$, we claim that the solution is given by formula
$$
u(x)=\int_\Omega G_\Omega(y,x)\;d\lambda(y)-\int_{\C\Omega}\Ds G_\Omega(x,y)\;d\mu(y).
$$
Take $\phi\in\T(\Omega)$:
\begin{multline*}
\int_\Omega u(x)\,\Ds\phi(x)\;dx=
\int_\Omega \left[\int_\Omega G_\Omega(y,x)\;d\lambda(y)-\int_{\C\Omega}\Ds G_\Omega(x,y)\;d\mu(y)\right]\,\Ds\phi(x)\;dx=\\
=\int_\Omega\left[\int_\Omega \Ds\phi(x)\,G_\Omega(y,x)\;dx\right]\;d\lambda(y)-
\int_{\C\Omega}\left[\int_\Omega\Ds\phi(x)\,\Ds G_\Omega(x,y)\;dx\right]\;d\mu(y)=\\
=\int_\Omega\phi(y)\;d\lambda(y)-\int_{\C\Omega}\Ds\phi(y)\;d\mu(y)
\end{multline*}
again using \eqref{repr-lapl}.
Then, we claim that the function
$$
u(x)=\int_{\partial\Omega}M(x,\theta)\;d\nu(\theta)
$$
solves problem
$$
\left\lbrace\begin{array}{ll}
\Ds u=0 & \hbox{ in }\Omega, \\
u=0 & \hbox{ in }\C\Omega, \\
Eu=\nu & \hbox{ on }\partial\Omega.
\end{array}\right.
$$
This, along with the first part of the proof, proves our thesis.
Take $\phi\in\T(\Omega)$ and call $\psi=\Ds\phi|_\Omega\in C^\infty_c(\Omega)$.
Then
\begin{eqnarray*}
\int_\Omega\left(\int_{\partial\Omega}M(x,\theta)\;d\nu(\theta)\right)\psi(x)\;dx
& = & 
\int_{\partial\Omega}\left(\int_\Omega M(x,\theta)\psi(x)\;dx\right)\;d\nu(\theta) \\
& = & 
\int_{\partial\Omega}\left(\int_\Omega \psi(x)\,
\lim_{\stackrel{\hbox{\scriptsize $z\in\Omega$}}{z\rightarrow\theta}}\frac{G_\Omega(x,z)}{{\delta(z)}^s}\;dx\right)\;d\nu(\theta) \\
& = & 
\int_{\partial\Omega}\lim_{\stackrel{\hbox{\scriptsize $z\in\Omega$}}{z\rightarrow\theta}}\left(\int_\Omega 
\frac{\psi(x)\,G_\Omega(x,z)}{{\delta(z)}^s}\;dx\right)\;d\nu(\theta) \\
& = & 
\int_{\partial\Omega}\lim_{\stackrel{\hbox{\scriptsize $z\in\Omega$}}{z\rightarrow\theta}}\frac{\phi(z)}{{\delta(z)}^s}\;d\nu(\theta) 
\quad = \quad
\int_{\partial\Omega}D_s\phi(\theta)\;d\nu(\theta).
\end{eqnarray*}

The uniqueness is due to Lemma \ref{max-princweak2} below. 
Theorem \ref{reg-l1sol} below proves the estimate on the $L^1$ norm of the solution.
\end{proof}

\begin{lem}[Maximum Principle]\label{max-princweak2} Let $u\in L^1(\Omega)$ be a solution of
$$
\left\lbrace\begin{array}{ll}
\Ds u\leq 0 & \hbox{ in }\Omega, \\
u\leq 0 & \hbox{ in }\C\Omega, \\
Eu\leq 0 & \hbox{ on }\partial\Omega.
\end{array}\right.
$$
Then $u\leq 0$ a.e. in $\R^n$.
\end{lem}
\begin{tfa}\rm Take $\psi\in C^\infty_c(\Omega)$, $\psi\geq 0$ and
the associated $\phi\in\T(\Omega)$ for which $\Ds\phi|_\Omega=\psi$:
it is $\phi\geq 0$ in $\Omega$, in view of Lemma \ref{maxprinc}
This implies that for $y\in\C\Omega$ it is
$$
\Ds\phi(y)=-\A(n,s)\,\int_\Omega\frac{\phi(z)}{|y-z|^n}\;dz\leq0,
$$
and also $D_s\phi\geq 0$ throughout $\partial\Omega$.
In particular 
$$
\int_\Omega u\psi\leq 0
$$
as a consequence of \eqref{weakdefi}.
\end{tfa}

\subsection{Regularity theory}

\begin{theo}\label{reg-l1sol} Given three Radon measures $\lambda\in\mathcal{M}(\Omega)$, 
$\mu\in\mathcal{M}(\C\Omega)$ and $\nu\in\mathcal{M}(\partial\Omega)$,
consider the solution $u\in L^1(\Omega)$ of problem
$$
\left\lbrace\begin{array}{ll}
\Ds u=\lambda & \hbox{ in }\Omega, \\
u=\mu & \hbox{ in }\C\Omega, \\
Eu=\nu & \hbox{ on }\partial\Omega.
\end{array}\right.
$$
Then 
$$
\Arrowvert u\Arrowvert_{L^1(\Omega)}\leq C\left(
\Arrowvert\lambda\Arrowvert_{\mathcal{M}(\Omega,{\delta(x)}^s\,dx)}+
\Arrowvert\mu\Arrowvert_{\mathcal{M}(\C\Omega,{\delta(x)}^{-s}\wedge{\delta(x)}^{-n-2s}\,dx)}+
\Arrowvert\nu\Arrowvert_{\mathcal{M}(\partial\Omega)}\right).
$$
\end{theo}
\begin{tfa}\rm 
Consider $\zeta$ to be the solution of 
$$
\left\lbrace\begin{array}{ll}
\Ds\zeta=1 & \hbox{ in }\Omega \\
\zeta=0 & \hbox{ in }\C\Omega \\
E\zeta=0 & \hbox{ on }\partial\Omega
\end{array}\right.
$$
which we know to satisfy $0\leq\zeta(x)\leq C\,\delta(x)^s$ in $\Omega$,
see \cite{rosserra}.
Note also that, by approximating $\zeta$ with functions in $\T(\Omega)$ and
by \ref{repr-lapl}, for $x\in\C\Omega$
$$
\Ds\zeta(x)=\int_\Omega\Ds G_\Omega(z,x)\;dz,\qquad x\in\C\Omega,
$$
and therefore, when $x\in\C\Omega$ and $\delta(x)<1$,
$$
0\leq-\Ds\zeta(x)\leq \int_\Omega\frac{C\,\delta(y)^s}{|y-x|^{n+2s}}\;dy\leq
C\int_{\delta(x)}^{+\infty}\frac{dt}{t^{1+s}}=\frac{C}{\delta(x)^s},
$$
while for $x\in\C\Omega$ and $\delta(x)\geq1$
$$
0\leq-\Ds\zeta(x)\leq \int_\Omega\frac{C\,\delta(y)^s}{|y-x|^{n+2s}}\;dy\leq
\frac{C}{\delta(x)^{n+2s}}\int_\Omega\delta(y)^s\;dy.
$$
Furthermore, $D_s\zeta$ is a well-defined function on $\partial\Omega$,
again thanks to the results in \cite[Proposition 1.1]{rosserra}. Indeed, consider
an increasing sequence ${\{\psi_k\}}_{k\in\N}\subseteq C^\infty_c(\Omega)$,
such that $0\leq\psi_k\leq 1$, $\psi_k\uparrow 1$ in $\Omega$.
Call $\phi_k$ the function in $\T(\Omega)$ associated with $\psi_k$,
i.e. $\Ds\phi_k|_\Omega=\psi_k$.
In this setting
$$
0\leq\lim_{\stackrel{\hbox{\scriptsize $z\!\in\!\Omega$}}{z\rightarrow\theta}}
\frac{\zeta(z)-\phi_k(z)}{{\delta(z)}^s}=
\lim_{\stackrel{\hbox{\scriptsize $z\!\in\!\Omega$}}{y\rightarrow\theta}}
\int_\Omega\frac{G_\Omega(y,z)}{{\delta(z)}^s}\left(1-\psi_k(y)\right)\;dy
\leq\int_\Omega M_\Omega(y,\theta)\left(1-\psi_k(y)\right)\;dy
\xrightarrow{k\rightarrow+\infty}0,
$$
hence $D_s\zeta(\theta)$ is well-defined and it is
$$
D_s\zeta(\theta)=\int_\Omega M_\Omega(x,\theta)\;dx.
$$
Finally, we underline how $D_s\zeta\in L^\infty(\Omega)$,
thanks to \eqref{chen-martin}.

We split the rest of the proof by using the integral representation of $u$:
\begin{description}
\item[-] the mass induced by the right-hand side is
$$
\int_\Omega\left|\int_\Omega G_\Omega(y,x)\;d\lambda(y)\right|\;dx\leq
\int_\Omega\int_\Omega G_\Omega(y,x)\;dx\;d|\lambda|(y)\leq \int_\Omega\zeta(y)\;d|\lambda|(y)
\leq C\int_\Omega\delta(y)^s\;d|\lambda|(y),
$$
\item[-] the one induced by the external datum
\begin{multline*}
\int_\Omega\left|\int_{\C\Omega} \Ds G_\Omega(x,y)\;d\mu(y)\right|\;dx\leq
\int_{\C\Omega}-\int_\Omega \Ds G_\Omega(x,y)\;dx\;d|\mu|(y)=\\
=\int_{\C\Omega}-\Ds\zeta(y)\;d|\mu|(y)\leq
\int_{\C\Omega}\left(\frac{C}{\delta(y)^s}\wedge\frac{C}{\delta(y)^{n+2s}}\right)\;d|\mu|(y),
\end{multline*}
\item[-] finally the mass due to the boundary behavior
$$
\int_\Omega\left|\int_{\partial\Omega} M_\Omega(x,\theta)\;d\nu(\theta)\right|\;dx\leq
\int_{\partial\Omega}\int_\Omega M_\Omega(x,\theta)\;dx\;d|\nu|(\theta)=
\int_{\partial\Omega}D_s\zeta(\theta)\;d|\nu|(\theta)
\leq \Arrowvert D_s\zeta\Arrowvert_\infty\,|\nu|(\partial\Omega).
$$
\end{description}
Note that the smoothness of the domain is needed 
only to make the last point go through,
and we can repeat the proof in case $\nu=0$
without requiring $\partial\Omega\in C^{1,1}$.
\end{tfa}

To gain higher integrability on a solution,
the first step we take is the following

\begin{lem}\label{reg-t} For any $q>n/s$ there exists 
$C=C(n,s,\Omega,q)$ such that for all $\phi\in\T(\Omega)$
$$
\left\Arrowvert\frac{\phi}{\delta^s}\right\Arrowvert_{L^\infty(\Omega)}
\leq C\, \Arrowvert\Ds\phi\Arrowvert_{L^q(\Omega)}.
$$
\end{lem}
\begin{tfa}\rm Call as usual $\psi=\Ds\phi|_\Omega\in C^\infty_c(\Omega)$.
Let us work formally
$$
\left\Arrowvert\frac{\phi}{\delta^s}\right\Arrowvert_{L^\infty(\Omega)}
\leq \sup_{x\in\Omega}\,\frac{1}{\delta(x)^s}\int_\Omega G_\Omega(x,y)\,|\psi(y)|\;dy
\leq \sup_{x\in\Omega}\,\frac{\left\Arrowvert G_\Omega(x,\cdot)\right\Arrowvert_{L^p(\Omega)}}{\delta(x)^s}
\cdot \Arrowvert\Ds\phi\Arrowvert_{L^q(\Omega)},
$$
where $\frac{1}{p}+\frac{1}{q}=1$. Thus we need only to understand 
for what values of $p$ we are not writing a trivial inequality.
The main tool here is inequality
$$
G_\Omega(x,y)\leq \frac{c\,\delta(x)^s}{|x-y|^n}\left(|x-y|^s\wedge\delta(y)^s\right),
$$
which holds in $C^{1,1}$ domains, see \cite[equation 2.14]{chen}.
We have then
$$
\int_\Omega\left|\frac{G_\Omega(x,y)}{\delta(x)^s}\right|^p\;dy\leq
c\int_\Omega \frac{1}{|x-y|^{np}}\left(|x-y|^{sp}\wedge\delta(y)^{sp}\right)\;dy
\leq c\int_\Omega \frac{1}{|x-y|^{(n-s)p}}
$$
which is uniformly bounded in $x$ for $p<\frac{n}{n-s}$.
This condition on $p$ becomes $q>\frac{n}{s}$.
\end{tfa}

In view of this last lemma,
we are able to provide the following theorem,
which is the fractional counterpart of a classical result
(see e.g. \cite[Proposition A.9.1]{dupaigne-book}).

\begin{theo}\label{reg} For any $p<\frac{n}{n-s}$, there exists a constant $C$ such that
any solution $u\in L^1(\Omega)$ of problem
$$
\left\lbrace\begin{array}{ll}
\Ds u=\lambda & \hbox{ in }\Omega \\
u=0 & \hbox{ in }\C\Omega \\
Eu=0 & \hbox{ on }\partial\Omega
\end{array}\right.
$$
has a finite $L^p$-norm controlled by
$$
\Arrowvert u\Arrowvert_{L^p(\Omega)}\leq C\,
\Arrowvert\lambda\Arrowvert_{L^1(\Omega,{\delta(x)}^s dx)}
$$
\end{theo}
\begin{tfa}\rm For any $\psi\in C^\infty_c(\Omega)$, let $\phi\in\T(\Omega)$ be
chosen in such a way that $\Ds\phi|_\Omega=\psi$. We have
$$
\left|\int_\Omega u(x)\,\psi(x)\;dx\right|=
\left|\int_\Omega \phi(x)\;d\lambda(x)\right|\leq
\int_\Omega\frac{|\phi(x)|}{\delta(x)^s}\cdot\delta(x)^s\;d|\lambda|(x)\leq
C\Arrowvert\psi\Arrowvert_{L^q(\Omega)}\int_\Omega\delta(x)^s\;d|\lambda|(x).
$$
where $q$ is the conjugate exponent of $p$, according to Lemma \ref{reg-t}.
By density of $C^\infty_c(\Omega)$ in $L^q(\Omega)$ and 
the isometry between $L^p(\Omega)$ and the dual space of $L^q(\Omega)$,
we obtain our thesis.
\end{tfa}

\subsection{Asymptotic behaviour at the boundary}\label{bblin-sec}

\subsubsection{Right-hand side blowing up at the boundary: proof of \eqref{udelta}}
\label{rhs-blowing}

In this paragraph we study the boundary behaviour
of the solution $u$ to the problem
$$
\left\lbrace\begin{array}{ll}
\displaystyle\Ds u(x)=\frac{1}{{\delta(x)}^\beta} & \hbox{ in }\Omega,\ 0<\beta<1+s \\
u=0 & \hbox{ in }\C\Omega \\
Eu=0 & \hbox{ on }\partial\Omega.
\end{array}\right.
$$
then
$$
u(x)=\int_B\frac{G_B(x,y)}{{\delta(y)}^\beta}\;dy
$$
and by using \eqref{est-green}, up to multiplicative constants,
$$
u(x)\leq\int_B \frac{\left[{|x-y|}^2\wedge\delta(x)\delta(y)\right]^s}{{|x-y|}^n}
\cdot\frac{dy}{{\delta(y)}^\beta}.
$$
Set
$$
\eps:=\delta(x)\hbox{ and }x=(1-\eps)e_1,\qquad
\theta:=\frac{y}{|y|},\,r:=\delta(y)\hbox{ and }y=(1-r)\theta,
$$
and rewrite
$$
u(x)\leq\int_{S^{n-1}}\int_0^1 \frac{\left[{|(1-\eps)e_1-(1-r)\theta|}^2\wedge\eps r\right]^s}{{|(1-\eps)e_1-(1-r)\theta|}^n}
\cdot\frac{{(1-r)}^{n-1}}{{r}^\beta}\;dr\;d\mathcal{H}^{n-1}(\theta).
$$
Split the angular variable in $\theta=(\theta_1,\theta')$ where
$\theta_1=\langle\theta,e_1\rangle$ and $\theta_1^2+|\theta'|^2=1$: then, for a general $F$,
$$
\int_{S^{n-1}} F(\theta)d\mathcal{H}^{n-1}(\theta)=
\int_{S^{n-2}}\left(\int_{-1}^1\left(1-\theta_1^2\right)^{(n-3)/2} F(\theta_1,\theta')\;d\theta_1\right)\;d\mathcal{H}^{n-2}(\theta'),
$$
and we write
\begin{multline*}
u(x)\leq \\
\int_{S^{n-2}}\int_0^1\left(1-\theta_1^2\right)^{(n-3)/2}\int_0^1
\frac{\left\{[(1-\eps)^2+(1-r)^2-2(1-\eps)(1-r)\theta_1]\wedge\eps r\right\}^s(1-r)^{n-1}}{{[(1-\eps)^2+(1-r)^2-2(1-\eps)(1-r)\theta_1]}^{n/2}\;{r}^\beta}
\,dr\,d\theta_1\,d\mathcal{H}^{n-2}(\theta') \\
=\mathcal{H}^{n-2}(S^{n-2})\int_0^1\left(1-\theta_1^2\right)^{(n-3)/2}\int_0^1
\frac{\left\{[(1-\eps)^2+(1-r)^2-2(1-\eps)(1-r)\theta_1]\wedge\eps r\right\}^s{(1-r)}^{n-1}}{{[(1-\eps)^2+(1-r)^2-2(1-\eps)(1-r)\theta_1]}^{n/2}\;{r}^\beta}\,dr\,d\theta_1.
\end{multline*}
From now on we will drop all multiplicative constants
and all inequalities will have to be interpreted to hold up to constants.
Let us apply a first change of variables
$$
t=\frac{1-r}{1-\eps}\quad\longleftrightarrow\quad r=1-(1-\eps)t
$$
to obtain
\begin{align}
u(x)\leq\int_0^1\left(1-\theta_1^2\right)^{(n-3)/2}\int_0^{1/(1-\eps)}
\frac{\left\{[(1-\eps)^2+(1-\eps)^2t^2-2(1-\eps)^2t\theta_1]\wedge\eps(1-(1-\eps)t)\right\}^s}
{{[(1-\eps)^2+(1-\eps)^2t^2-2(1-\eps)^2t\theta_1]}^{n/2}}
\times \nonumber\\
\times \frac{(1-\eps)^nt^{n-1}}{{(1-(1-\eps)t)}^\beta}\;dt\;d\theta_1 \nonumber \\
=\int_0^1\left(1-\theta_1^2\right)^{(n-3)/2}\int_0^{1/(1-\eps)}
\frac{\left[(1-\eps)^2(1+t^2-2t\theta_1)\wedge\eps(1-(1-\eps)t)\right]^s}
{{[1+t^2-2t\theta_1]}^{n/2}}\cdot\frac{t^{n-1}}{{(1-(1-\eps)t)}^\beta}\;dt\;d\theta_1 \nonumber \\
\leq\int_0^{1/(1-\eps)}\int_0^1
\frac{\left\{[(1-t)^2+2t\sigma]\wedge\eps[1-(1-\eps)t]\right\}^s}
{{[(1-t)^2+2t\sigma]}^{n/2}}\,{\sigma}^{(n-3)/2}
\;d\sigma\;\frac{t^{n-1}}{{(1-(1-\eps)t)}^\beta}\;dt 	\label{67},
\end{align}
where $\sigma=1-\theta_1$.
We compute now the integral in the variable $\sigma$.
Let $\sigma'$ be defined by equality
$$
(1-t)^2+2t\sigma'=\eps[1-(1-\eps)t]
\qquad\Longleftrightarrow\qquad
\sigma'=\frac{\eps[1-(1-\eps)t]-(1-t)^2}{2t}
$$
and let
$$
\sigma^*=\max\{\sigma',0\}.
$$
The quantity $\sigma^*$ equals $0$ if and only if
$$
\eps[1-(1-\eps)t]\leq(1-t)^2
$$
and it is easy to verify that this happens whenever
$$
t\leq t_1(\eps):=1-\frac{\eps-\eps^2+\eps\sqrt{\eps^2-2\eps+5}}{2}
\qquad\hbox{ or }\qquad
t\geq t_2(\eps):=1+\frac{-\eps+\eps^2+\eps\sqrt{\eps^2-2\eps+5}}{2}.
$$
\begin{rmk}\rm $t_2(\eps)<\frac{1}{1-\eps}$, since for small $\eps>0$
\begin{multline*}
1+\frac{-\eps+\eps^2+\eps\sqrt{\eps^2-2\eps+5}}{2}<\frac{1}{1-\eps}
\Longleftrightarrow
\frac{-1+\eps+\sqrt{\eps^2-2\eps+5}}{2}<\frac{1}{1-\eps}
\Longleftrightarrow\\
\Longleftrightarrow
\sqrt{\eps^2-2\eps+5}<\frac{2}{1-\eps}+1-\eps.
\end{multline*}
Also, $0<t_1(\eps)<1$. Finally $\sigma'\leq 1$, since
$$
\eps[1-(1-\eps)t]-(1-t)^2\leq 2t
\qquad\Longleftrightarrow\qquad
\eps[1-(1-\eps)t]\leq 1+t^2
$$
which is true for any positive $\eps<1$.
\end{rmk}

Hence
$$
\sigma^*=\max\{\sigma',0\}=\left\lbrace
\begin{array}{ll}
\displaystyle
\frac{\eps[1-(1-\eps)t]-(1-t)^2}{2t} &
\hbox{ when }\; \displaystyle
t\in\left(t_1(\eps),t_2(\eps)\right)\subseteq\left(0,\frac{1}{1-\eps}\right), \\
0 & \hbox{ when }\; \displaystyle
t\in\left(0,t_1(\eps)\right)\cup\left(t_2(\eps),\frac{1}{1-\eps}\right).
\end{array}\right.
$$
We split now integral \eqref{67} into four pieces, as following
\begin{equation}\label{splitted}
\int_0^{\frac{1}{1-\eps}}dt\int_0^1d\sigma\quad=\quad
\int_{t_1(\eps)}^{t_2(\eps)}dt\int_0^{\sigma^*}d\sigma\ 
+\ \int_{t_1(\eps)}^{t_2(\eps)}dt\int_{\sigma^*}^1d\sigma\ 
+\ \int_0^{t_1(\eps)}dt\int_0^1d\sigma\ 
+\ \int_{t_2(\eps)}^\frac{1}{1-\eps}dt\int_0^1d\sigma,
\end{equation}
and we treat each of them separately:
\begin{itemize}
\item for the first one we have
\begin{multline*}
\int_0^{\sigma^*}\frac{{\sigma}^{(n-3)/2}}{{[(1-t)^2+2t\sigma]}^{n/2-s}}\;d\sigma
\leq\frac{1}{{(2t)}^{(n-3)/2}}\int_0^{\sigma^*}
{[(1-t)^2+2t\sigma]}^{s-3/2}\;d\sigma\leq\\
\leq \frac{1}{{t}^{(n-1)/2}}\frac{1}{s-\frac{1}{2}}
\left[\eps^{s-1/2}\left(1-(1-\eps)t\right)^{s-1/2}-|1-t|^{2s-1}\right]
\end{multline*}
and therefore the first integral is less than
\begin{equation}\label{1}
\int_{t_1(\eps)}^{t_2(\eps)}
\frac{t^{(n-1)/2}}{\left(1-(1-\eps)t\right)^\beta}\cdot
\frac{1}{s-\frac{1}{2}}
\left[\eps^{s-1/2}\left(1-(1-\eps)t\right)^{s-1/2}-|1-t|^{2s-1}\right]\;dt;
\end{equation}
note now that $t_2(\eps)-t_1(\eps)\sim\eps$ and $\frac{1}{1-\eps}-1\sim\eps$
and thus \eqref{1} is of magnitude
$$
\eps^{-\beta+2s};
$$
\item for the second integral we have
\begin{multline*}
\int_{\sigma^*}^1\frac{{\sigma}^{(n-3)/2}}{{[(1-t)^2+2t\sigma]}^{n/2}}\;d\sigma
\leq\frac{1}{{(2t)}^{(n-3)/2}}\int_0^{\sigma^*}
{[(1-t)^2+2t\sigma]}^{-3/2}\;d\sigma\leq\\
\leq \frac{1}{{t}^{(n-1)/2}}
\left[\eps^{-1/2}\left(1-(1-\eps)t\right)^{-1/2}-(1+t^2)^{-1/2}\right]
\end{multline*}
and therefore the second integral is less than
\begin{equation}\label{2}
\int_{t_1(\eps)}^{t_2(\eps)}
\frac{t^{(n-1)/2}}{\left(1-(1-\eps)t\right)^\beta}\cdot
\left[\eps^{-1/2}\left(1-(1-\eps)t\right)^{-1/2}-(1+t^2)^{-1/2}\right]\;dt;
\end{equation}
and \eqref{2} is of magnitude
$$
\eps^{-\beta+2s};
$$
\item for the third integral we have
$$
\int_0^1\frac{{\sigma}^{(n-3)/2}}{{[(1-t)^2+2t\sigma]}^{n/2}}\;d\sigma
\leq \frac{1}{t^{(n-3)/2}}\int_0^1\frac{d\sigma}{\left[(1-t)^2+2t\sigma\right]^{3/2}}
\leq \frac{1}{t^{(n-1)/2}}\cdot\frac{1}{|1-t|}
$$
so that the third integral is less than
\begin{equation}\label{3}
\begin{split}
\int_0^{t_1(\eps)}
\frac{t^{(n-1)/2}\,\eps^s}{\left(1-(1-\eps)t\right)^{\beta-s}}
\cdot\frac{dt}{1-t}\leq
\eps^s\int_0^{t_1(\eps)}
\frac{dt}{(1-t)\,\left(1-(1-\eps)t\right)^{\beta-s}}\leq\\
\leq\eps^s\int_0^{t_1(\eps)}
\frac{dt}{\left(1-t\right)^{\beta-s+1}}=
\frac{\eps^s\left(1-t_1(\eps)\right)^{\beta-s}}{\beta-s}
-\frac{\eps^s}{\beta-s}
\end{split}
\end{equation}
if $\beta\neq s$
and \eqref{3} is of magnitude
\begin{eqnarray*}
& \eps^s & \hbox{ for }0\leq\beta< s,\\
& \eps^s\log\frac{1}{\eps} & \hbox{ for }\beta=s,\\
& \eps^{-\beta+2s} & \hbox{ for }s<\beta<1+s;
\end{eqnarray*}
\item for the fourth integral we have
$$
\int_0^1\frac{{\sigma}^{(n-3)/2}}{{[(1-t)^2+2t\sigma]}^{n/2}}\;d\sigma
\leq \frac{1}{t^{(n-1)/2}}\cdot\frac{1}{|1-t|}
$$
so that the fourth integral is less than
\begin{equation}\label{4}
\int_{t_2(\eps)}^{\frac{1}{1-\eps}}
\frac{t^{(n-1)/2}\,\eps^s}{\left(1-(1-\eps)t\right)^{\beta-s}}
\cdot\frac{dt}{t-1}\leq
\eps^s\int_{t_2(\eps)}^{\frac{1}{1-\eps}}
\frac{dt}{(t-1)\,\left(1-(1-\eps)t\right)^{\beta-s}}
\end{equation}
and \eqref{4} is of magnitude
\begin{eqnarray*}
\eps^s\int_{t_2(\eps)}^{\frac{1}{1-\eps}}\frac{dt}{\left(t-1\right)^{\beta-s+1}} & \sim\eps^s &
\hbox{ for }0\leq\beta< s, \\
\eps^s\int_{t_2(\eps)}^{\frac{1}{1-\eps}}\frac{dt}{t-1}
& \sim\eps^s\log\frac{1}{\eps} & \hbox{ for }\beta=s, \\
\frac{\eps^s}{t_2(\eps)-1}\int_{t_2(\eps)}^{\frac{1}{1-\eps}}
\frac{dt}{\left(1-(1-\eps)t\right)^{\beta-s}}
& \sim\eps^{-\beta+2s} &
\hbox{ for }s<\beta<1+s.
\end{eqnarray*}
\end{itemize}
Resuming the information collected so far,
what we have gained is that, up to constants,
\begin{equation}\label{rhs-estimate}
u(x)\leq \left\lbrace
\begin{array}{ll}
\displaystyle
{\delta(x)}^s &
\hbox{ for }0\leq\beta< s, \\ 
\displaystyle
{\delta(x)}^s\log\frac{1}{\delta(x)} & \hbox{ for }\beta=s,\\
\displaystyle
{\delta(x)}^{-\beta+2s} &
\hbox{ for }s<\beta<1+s.
\end{array}\right.
\end{equation}

This establish an upper bound for the solutions.
Note now that the integral \eqref{67} works also as
a lower bound, of course up to constants.
Using the split expression \eqref{splitted} we entail
\begin{multline*}
u(x)\geq\int_0^{\frac{1}{1-\eps}}
\int_0^1\frac{\left\{[(1-t)^2+1+2t\sigma]\wedge\eps[1-(1-\eps)t]\right\}^s}
{{[(1-t)^2+2t\sigma]}^{n/2}}\,{\sigma}^{(n-3)/2}
\;d\sigma\;\frac{t^{n-1}}{{(1-(1-\eps)t)}^\beta}
\;dt\geq \\ \geq
\int_0^{t_1(\eps)}\frac{t^{n-1}\,\eps^s}{\left(1-(1-\eps)t\right)^{\beta-s}}
\int_0^1\frac{{\sigma}^{(n-3)/2}}{\left[(1-t)^2+2t\sigma\right]^{n/2}}\;d\sigma\;dt,
\end{multline*}
where we have used only the expression with the third 
integral in \eqref{splitted}.
We claim that
\begin{equation}\label{45}
\int_0^1\frac{{\sigma}^{(n-3)/2}}{\left[(1-t)^2+2t\sigma\right]^{n/2}}
\;d\sigma\geq\frac{1}{1-t}\geq\frac{1}{1-(1-\eps)t},
\end{equation}
where the inequality is intended to hold up to constants.
In case \eqref{45} holds and $\beta\neq s$ we have
\begin{multline*}
\int_0^{t_1(\eps)}\frac{t^{n-1}\,\eps^s}{\left(1-(1-\eps)t\right)^{\beta-s}}
\int_0^1\frac{{\sigma}^{(n-3)/2}}{\left[(1-t)^2+2t\sigma\right]^{n/2}}\;d\sigma\;dt
\geq \int_0^{t_1(\eps)}\frac{t^{n-1}\,\eps^s}{\left(1-(1-\eps)t\right)^{\beta-s+1}}\;dt=\\
=\left.\frac{1}{(1-\eps)\,(\beta-s)}\cdot
\frac{t^{n-1}\,\eps^s}{\left(1-(1-\eps)t\right)^{\beta-s}}\right\arrowvert_0^{t_1(\eps)}
-\frac{n-1}{(1-\eps)(\beta-s)}\int_0^{t_1(\eps)}\frac{t^{n-2}\,\eps^s}{\left(1-(1-\eps)t\right)^{\beta-s}}\;dt.
\end{multline*}
The integral on the second line is a bounded quantity as $\eps\downarrow 0$
since $\beta<1+s$. Now, we are left with
$$
u(x)\geq\frac{\eps^{-\beta+2s}}{\beta-s}-\frac{\eps^s}{\beta-s}
$$
so that
$$
u(x)\geq\left\lbrace\begin{array}{ll}
\displaystyle\eps^s & \hbox{ when }0\leq\beta< s \\
\displaystyle\eps^s\log\frac{1}{\eps} & \hbox{ when }\beta =s \\
\displaystyle\eps^{-\beta+2s} & \hbox{ when }s<\beta<1+s.
\end{array}\right.
$$
We still have to prove \eqref{45}: note that an integration by parts yields, for $n\geq4$,
$$
\int_0^1\frac{{\sigma}^{(n-3)/2}}{\left[(1-t)^2+2t\sigma\right]^{n/2}}
\:d\sigma=
\frac{1}{2t\,(1-n/2)}\cdot\frac{1}{\left[(1-t)^2+2t\right]^{n/2-1}}
+\frac{1}{2t\,(n/2-1)}\int_0^1\frac{{\sigma}^{(n-5)/2}}{\left[(1-t)^2+2t\sigma\right]^{n/2-1}}\:d\sigma
$$
so that we can show \eqref{45} only in dimensions $n=2,\,3$
and deduce the same conclusions for any other value of $n$
by integrating by parts a suitable number of times.
For $n=2$
\begin{eqnarray*}
\int_0^1\frac{d\sigma}{\sqrt{\sigma}\left[(1-t)^2+2t\sigma\right]}
& = & 2\int_0^1\frac{d\xi}{(1-t)^2+2t\xi^2} \quad =\quad
\frac{2}{(1-t)^2}\int_0^1\frac{d\xi}{1+\frac{2t}{(1-t)^2}\,\xi^2}\\
& = & \frac{\pi}{\sqrt{2t}(1-t)}.
\end{eqnarray*}
For $n=3$
$$
\int_0^1\frac{d\sigma}{\left[(1-t)^2+2t\sigma\right]^{3/2}}=
\frac{1}{4t}\left(\frac{1}{1-t}-\sqrt{\frac{1}{1+t^2}}\right)
$$
which completely proves our claim \eqref{45}.

So far we have worked only on spherical domains.
In a general domain $\Omega$ with $C^{1,1}$ boundary,
split the problem in two by setting $u=u_1+u_2$,
for $0<\beta<1+s$ and a small $\delta_0>0$
$$
\left\lbrace\begin{array}{ll}
\displaystyle\Ds u_1(x)=\frac{\chi_{\{\delta\geq\delta_0\}}(x)}{{\delta(x)}^\beta} & \hbox{ in }\Omega, \\
u_1=0 & \hbox{ in }\C\Omega \\
Eu_1=0 & \hbox{ on }\partial\Omega
\end{array}\right.
\qquad\qquad
\left\lbrace\begin{array}{ll}
\displaystyle\Ds u_2(x)=\frac{\chi_{\{\delta<\delta_0\}}(x)}{{\delta(x)}^\beta} & \hbox{ in }\Omega, \\
u_2=0 & \hbox{ in }\C\Omega \\
Eu_2=0 & \hbox{ on }\partial\Omega
\end{array}\right.
$$
Note that $u_1\in C^s(\R^n)$, see \cite[Proposition 1.1]{rosserra}.
Since $\partial\Omega\in C^{1,1}$, we choose now $\delta_0$ sufficiently small in order to have
that for any $y\in\Omega$ with $\delta(y)<\delta_0$ it is uniquely determined $\theta=\theta(y)\in\partial\Omega$
such that $|y-\theta|=\delta(y)$. Then in inequalities (see \eqref{est-green})
\begin{multline*}
\int_{\{\delta(y)<\delta_0\}}\frac{\left[{|x-y|}^2\,\wedge\,\delta(x)\delta(y)\right]^s}{c_2\,|x-y|^n\,{\delta(y)}^\beta}\;dy
\leq u_2(x)=\int_{\{\delta(y)<\delta_0\}}\frac{G_\Omega(x,y)}{{\delta(y)}^\beta}\;dy\leq \\
\leq\int_{\{\delta(y)<\delta_0\}}\frac{c_2\,\left[{|x-y|}^2\,\wedge\,\delta(x)\delta(y)\right]^s}{|x-y|^n\,{\delta(y)}^\beta}\;dy
\end{multline*}
set
$$
\begin{array}{l}
\nu(\theta),\ \theta\in\partial\Omega,\ \hbox{the outward unit normal to }\partial\Omega\hbox{ at }\theta, \\
\eps:=\delta(x),\ \theta^*=\theta(x),\ \hbox{ and }x=\theta^*-\eps\,\nu(\theta^*), \\
\theta=\theta(y),\ r:=\delta(y),\ \hbox{ and }y=\theta-r\,\nu(\theta).
\end{array}
$$
Then, using the Fubini's Theorem, we write
\begin{equation}\label{09}
\int_{\{\delta(y)<\delta_0\}}\frac{\left[{|x-y|}^2\,\wedge\,\delta(x)\delta(y)\right]^s}{|x-y|^n\,{\delta(y)}^\beta}\;dy
=\int_0^{\delta_0}\int_{\partial\Omega}\frac{\left[{|\theta^*-\eps\,\nu(\theta^*)-\theta+r\,\nu(\theta)|}^2\,\wedge\,\eps r\right]^s}
{{|\theta^*-\eps\,\nu(\theta^*)-\theta+r\,\nu(\theta)|}^n\,{r}^\beta}\;d\mathcal{H}(\theta)\;dr.
\end{equation}
Split the integration on $\partial\Omega$ into the integration on 
$\Gamma:=\{\theta\in\partial\Omega:|\theta-\theta^*|<\delta_1\}$ and $\partial\Omega\setminus\Gamma$,
and choose $\delta_1>0$ small enough to have a $C^{1,1}$ diffeomorphism 
$$
\begin{array}{rcl}
\varphi:\widetilde\Gamma\subseteq S^{n-1} & \longrightarrow & \Gamma \\
\omega & \longmapsto & \theta=\varphi(\omega).
\end{array}
$$
We build now
$$
\begin{array}{rcl}
\overline\varphi:\widetilde\Gamma\times(0,\delta_0)\subseteq B_1 & \longrightarrow & \{y\in\Omega:\delta(y)<\delta_0,\ |\theta(y)-\theta^*|<\delta_1\} \\
(1-\delta)\omega=\omega-\delta\omega & \longmapsto & y=\theta-\delta\,\nu(\theta)=\varphi(\omega)-\delta\,\nu(\varphi(\omega)),
\end{array}
$$
where we suppose $e_1\in\widetilde\Gamma$ and $\phi(e_1)=\theta^*$.
With this change of variables \eqref{09} becomes
$$
\int_{\{\delta(y)<\delta_0\}}\frac{\left[{|x-y|}^2\,\wedge\,\delta(x)\delta(y)\right]^s}{|x-y|^n\,{\delta(y)}^\beta}\;dy
=\int_{0}^{\delta_0}\int_{\widetilde\Gamma}\frac{\left[{|e_1-\eps\,e_1-\omega+r\,\omega|}^2\,\wedge\,\eps r\right]^s}
{{|e_1-\eps\,e_1-\omega+r\,\omega|}^n\,{r}^\beta}\;|D\overline\varphi(\omega)|\;d\mathcal{H}(\omega)\;dr
$$
which, since $|D\overline\varphi(\omega)|$ is a bounded continuous quantity far from $0$,
is bounded below and above in terms of
$$
\int_{0}^{\delta_0}\int_{\widetilde\Gamma}\frac{\left[{|(1-\eps)e_1-(1-r)\omega|}^2\,\wedge\,\eps r\right]^s}
{{|(1-\eps)e_1-(1-r)\omega|}^n\,{r}^\beta}\;d\mathcal{H}(\omega)\;dr,
$$ 
i.e. we are brought back to the spherical case.

\subsubsection{Boundary continuity of $s$-harmonic functions}\label{cont-sharm-sec}

Consider $g:\C\Omega\rightarrow\R$ and
$$
u(x)=-\int_{\C\Omega}g(y)\Ds G_\Omega(x,y)\;dy
$$
and think of letting $x\rightarrow\theta\in\partial\Omega$.
Suppose that for any small $\eps>0$ there exists $\delta>0$
such that $|g(y)-g(\theta)|<\eps$ for any $y\in\C\Omega\cap B_\delta(\theta)$.
Then
\begin{multline*}
|u(x)-g(\theta)|=\left|\int_{\C\Omega}(g(\theta)-g(y))\Ds G_\Omega(x,y)\;dy\right|\leq\\
\leq \int_{\C\Omega\cap B_\delta(\theta)}\left|(g(\theta)-g(y))\Ds G_\Omega(x,y)\right|\;dy+
\int_{\C\Omega\setminus B_\delta(\theta)}\left|(g(\theta)-g(y))\Ds G_\Omega(x,y)\right|\;dy.
\end{multline*}
The first addend satisfies
$$
\int_{\C\Omega\cap B_\delta(\theta)}\left|(g(\theta)-g(y))\Ds G_\Omega(x,y)\right|\;dy
\leq -\eps \int_{\C\Omega\cap B_\delta(\theta)}\Ds G_\Omega(x,y)\;dy\leq\eps.
$$
For the second one we exploit \eqref{est-poisson}:
$$
\int_{\C\Omega\setminus B_\delta(\theta)}\left|(g(\theta)-g(y))\Ds G_\Omega(x,y)\right|\;dy\leq
{\delta(x)}^s\int_{\C\Omega\setminus B_\delta(\theta)}\frac{c_1|(g(\theta)-g(y))|}{{\delta(y)}^s\left(1+\delta(y)\right)^s{|x-y|}^n}
$$
which converges to $0$ as $x\rightarrow\theta\in\partial\Omega$.
So we have that
$$
\lim_{x\rightarrow\theta}|u(x)-g(\theta)|\leq\eps,
$$
and by arbitrarily choosing $\eps$, we conclude 
$$
\lim_{x\rightarrow\theta}u(x)=g(\theta).
$$

\subsubsection{Explosion rate of large $s$-harmonic functions: proof of \eqref{uleqg}}
\label{expl-sharm-sec}

We study here the rate of divergence of 
$$
u(x)=-\int_{\C\Omega}g(y)\:\Ds G_\Omega(x,y)\;dy
\quad\hbox{ as }x\rightarrow\partial\Omega,\ x\in\Omega
$$
which is the solution to 
$$
\left\lbrace\begin{array}{ll}
\Ds u=0 & \hbox{ in }\Omega \\
u=g & \hbox{ in }\C\Omega \\
Eu=0 & \hbox{ on }\partial\Omega
\end{array}\right.
$$
in case $g$ explodes at $\partial\Omega$.

\begin{rmk}\rm The asymptotic behaviour of $u$ depends only on
the values of $g$ near the boundary, since we can split
$$
g=g\chi_{\{d<\eta\}}+g\chi_{\{d\geq\eta\}}
$$
and the second addend has a null contribution on the boundary,
in view of Paragraph \ref{cont-sharm-sec}.
Therefore in our computations we will suppose
that $g(y)=0$ for $\delta(y)>\eta$.
\end{rmk}

In the further assumption that $g$ explodes like a power, i.e.
there exist $\eta, k, K>0$ for which
$$
\frac{k}{\delta(y)^\tau}\leq g(y)\leq\frac{K}{\delta(y)^\sigma},
\qquad \hbox{ for } 0\leq\tau\leq\sigma<1-s,\ 0<\delta(y)<\eta.
$$
(the choice $\sigma<1-s$ is in order to have \eqref{g},
see \eqref{est-poisson} above)
our proof doesn't require heavy computations
and it is as follows.

Dropping multiplicative constants in inequalities and for $\Omega_\eta=\{y\in\C\Omega:\delta(y)<\eta\}$:
\begin{multline*}
\delta(x)^\sigma u(x)=
-\delta(x)^\sigma\int_{\C\Omega}g(y)\cdot\Ds G_\Omega(x,y)\;dy \\
\leq\int_{\Omega_\eta}\frac{\delta(x)^{s+\sigma}}{\delta(y)^{s+\sigma}\,\left(1+\delta(y)\right)^s\,|x-y|^n}\;dy 
\leq-\int_{\C\Omega}\chi_{\Omega_\eta}(y)\cdot{(-\Delta)}^{s+\sigma}G_\Omega^{s+\sigma}(x,y)\;dy 
\ \leq\ 1
\end{multline*}
Similarly one can treat also the lower bound:
\begin{multline*}
\delta(x)^\tau u(x)=
-\delta(x)^\tau\int_{\C\Omega}g(y)\cdot\Ds G_\Omega(x,y)\;dy \\
\geq\int_{\Omega_\eta}\frac{\delta(x)^{s+\tau}}{\delta(y)^{s+\tau}\,\left(1+\delta(y)\right)^s\,|x-y|^n}\;dy 
\geq-\int_{\C\Omega}\chi_{\Omega_\eta}(y)\cdot{(-\Delta)}^{s+\tau}G_\Omega^{s+\tau}(x,y)\;dy 
\ \xrightarrow[x\rightarrow\partial\Omega]{}1.
\end{multline*}
The limit we have computed above is the continuity up to the boundary
of $\widehat{u}$ solution of
$$
\left\lbrace\begin{array}{ll}
{(-\Delta)}^{s+\tau}\widehat{u}=0 & \hbox{ in }\Omega, \\
\widehat{u}=\chi_{\Omega_\eta} & \hbox{ in }\C\Omega, \\
E\widehat{u}=0 & \hbox{ on }\partial\Omega.
\end{array}\right.
$$

\begin{rmk}\rm Both the upper and the lower estimate are optimal,
thanks to what have been shown in Example \ref{expl-sol}.
\end{rmk}

In the case of a general boundary datum $g$
we start from the case $\Omega=B$,
recalling that in this setting, according to \cite[equation (1.6.11')]{landkof},
$$
-\Ds G_B(x,y)=\frac{c(n,s)}{{|x-y|}^n}\left(\frac{1-|x|^2}{|y|^2-1}\right)^s
$$
and therefore
$$
u(x)=\int_{\C B}\frac{c(n,s)}{{|x-y|}^n}\left(\frac{1-|x|^2}{|y|^2-1}\right)^s\:g(y)\;dy.
$$
Suppose without loss of generality
$$
\begin{array}{l}
x=(1-\eps)\,e_1\quad \eps=\delta(x) \\
y=(1+r)\,\theta\quad r=\delta(y),\quad\theta=\frac{y}{|y|},\quad\theta_1=e_1\cdot\theta
\end{array}
$$
so that
$$
u(x)=\int_0^{\eta}\left[
\int_{\partial B}
\frac{c(n,s)}{{|(1-\eps)^2+(1+r)^2-2(1-\eps)(1+r)\theta_1|}^{n/2}}
\left(\frac{\eps(2-\eps)}{r(2+r)}\right)^s\:g(r\theta)
\;d\mathcal{H}(\theta)
\right](1+r)^{n-1}\;dr.
$$
Denote now by $\overleftarrow{g}(r)=\sup_{\delta(x)=r}g(x)$.
Splitting the integral in the $\theta$ variable into two integrals
in the variables $(\theta_1,\theta')$ where $\theta_1^2+|\theta'|^2=|\theta|^2=1$,
up to constants we obtain
\begin{eqnarray}
u(x) & \leq & 
\int_0^{\eta}\left[
\int_{-1}^1
\frac{(1-\theta_1^2)^{(n-3)/2}}{{|(1-\eps)^2+(1+r)^2-2(1-\eps)(1+r)\theta_1|}^{n/2}}
\cdot\frac{\eps^s}{r^s}
\;d\theta_1\right]\overline{g}(r)\,(1+r)^{n-1}\;dr \nonumber \\
& \leq & 
\eps^s\int_0^{\eta}\left[
\int_{-1}^1
\frac{(1-\theta_1^2)^{(n-3)/2}}{{|(1-\eps)^2+(1+r)^2-2(1-\eps)(1+r)\theta_1|}^{n/2}}
\;d\theta_1\right]\frac{\overline{g}(r)}{r^s}\;dr. \label{34}
\end{eqnarray}
Define $M:=\frac{1+r}{1-\eps}>1$ and look at the inner integral:
$$
\int_{-1}^1
\frac{(1-\theta_1^2)^{(n-3)/2}}{{|1+M^2-2M\theta_1|}^{n/2}}
\;d\theta_1
=
\int_{-1}^1
\frac{(1-\theta_1^2)^{(n-3)/2}}{{|1-\theta_1^2+(M-\theta_1)^2|}^{n/2}}
\;d\theta_1
\leq
\int_{-1}^1
\left|1-\theta_1^2+(M-\theta_1)^2\right|^{-3/2}
\;d\theta_1.
$$
The integral from $-1$ to $0$ contributes by a bounded quantity
so that we are left with
\begin{eqnarray*}{rcl}
u(x) & \leq & 
\int_0^1
\left|1-\theta_1+(M-\theta_1)^2\right|^{-3/2}
\;d\theta_1
\ =\ \int_0^1
\left|\tau+(M-1+\tau)^2\right|^{-3/2}
\;d\tau \\
& \leq & 
\int_0^1
\left|\tau+(M-1)^2\right|^{-3/2}
\;d\tau
\ =\ 
-2\left.\left(\tau+(M-1)^2\right)^{-1/2}\right|_{\tau=0}^1 
\ \leq \ 
\frac{1}{M-1}\ =\ \frac{1-\eps}{r+\eps}.
\end{eqnarray*}
Thus
$$
u(x)\ \leq\ \eps^s\int_0^{\eta}
\frac{\overline{g}(r)}{r^s}\cdot\frac{1-\eps}{r+\eps}\;dr
\ \leq\ \eps^s\int_0^{\eps}
\frac{\overline{g}(r)}{r^s}\cdot\frac{1}{r+\eps}\;dr+
\int_1^{\eta/\eps}\frac{\overline{g}(\eps\tau)}{\tau^s}\cdot\frac{1}{1+\tau}\;d\tau.
$$
Our claim  now is that this last expression
is controlled by $\overline{g}(\eps)$ as $\eps\downarrow0$.
Since $\overline{g}$ is exploding in $0$, for small $\eps$ it is
$\overline{g}(\tau\eps)\leq\overline{g}(\eps)$ for $\tau>1$ and
$$
\frac{1}{\overline{g}(\eps)}\int_1^{\eta/\eps}\frac{\overline{g}(\eps\tau)}{\tau^s}\cdot\frac{1}{1+\tau}\;d\tau
\leq \int_1^{+\infty}\frac{1}{\tau^s}\cdot\frac{1}{1+\tau}\;d\tau.
$$
For the other integral
$$
\frac{\eps^s}{\overline{g}(\eps)}\int_0^{\eps}
\frac{\overline{g}(r)}{r^s}\cdot\frac{1}{r+\eps}\;dr\leq
\frac{\eps^s}{\eps\,\overline{g}(\eps)}\int_0^{\eps}
\frac{\overline{g}(r)}{r^s}\;dr.
$$
To compute the limit as $\eps\downarrow 0$
we use a Taylor expansion:
$$
\frac{1}{\eps}\int_0^\eps G(r)\;dr
=G(\eps)+G'(\eps)\,\frac{\eps^2}{2}+o(\eps^2),
$$
where we have denoted by $G(\eps)=\eps^{-s}\,\overline{g}(s)$.
Thus
$$
\frac{\eps^s}{\overline{g}(\eps)}\int_0^{\eps}
\frac{\overline{g}(r)}{r^s}\cdot\frac{1}{r+\eps}\;dr\leq
1+\frac{G'(\eps)}{2\,G(\eps)}\,\eps+o\left(\frac{\eps}{G(\eps)}\right).
$$
We are going to show now that
$$
\frac{|G'(\eps)|}{G(\eps)}\,\eps<1.
$$
Indeed
$$
-\int_\eps^\eta\frac{G'(\xi)}{G(\xi)}\;d\xi<\int_\eps^\eta\frac{d\xi}{\xi}
\qquad\Longleftrightarrow\qquad G(\eps)<G(\eta)\,\frac{\eta}{\eps}
$$
which is guaranteed by the fact that $G$ is integrable in a neighbourhood of $0$.

These computations show that, in the case of the ball,
the explosion rate of the $s$-harmonic function induced 
by a large boundary datum is the almost the same as
the rate of the datum itself.

Note now that up to \eqref{34} the same computations provide a lower estimate for $u$
if we substitute $\overline{g}$ with $\underline{g}(r)=\inf_{\delta(x)=r}g(x)$.
Then
$$
\int_{-1}^1
\frac{(1-\theta_1^2)^{(n-3)/2}}{{|1+M^2-2M\theta_1|}^{n/2}}
\;d\theta_1
=
\int_{-1}^1
\frac{(1-\theta_1^2)^{(n-3)/2}}{{|1-\theta_1^2+(M-\theta_1)^2|}^{n/2}}
\;d\theta_1
\geq
\int_0^1\frac{\sigma^{(n-3)/2}}{\left|\sigma+(M-1+\sigma)^2\right|^{n/2}}\;d\sigma
\geq \frac{1}{M-1}
$$
where the last inequality is \eqref{45}.
Finally we need only to repeat the above computations replacing
$\overline{g}$ with $\underline{g}$ and other minor modifications.

In the case of a general smooth domain,
we can reduce to the spherical case 
as we did to conclude Paragraph \ref{rhs-blowing}.

\section{The semilinear fractional Dirichlet problem}\label{nonlin-sec}

\subsection{The method of sub- and supersolution: proof of Theorem \ref{nl-cs}}
The proof is a simple readaptation of the result
by Cl\'ement and Sweers \cite{clem-sweers}.

\it Existence. \rm
We can reduce the problem to homogeneous boundary condition,
indeed by considering the solution of 
$$
\left\lbrace\begin{array}{ll}
\Ds v=0 & \hbox{ in }\Omega \\
v=g & \hbox{ in }\C\Omega \\
Ev=0 & \hbox{ on }\partial\Omega
\end{array}\right.
$$
we can think of solving the problem
$$
\left\lbrace\begin{array}{ll}
\Ds u=-f(x,v+u) & \hbox{ in }\Omega \\
u=0 & \hbox{ in }\C\Omega \\
Eu=0 & \hbox{ on }\partial\Omega
\end{array}\right.
$$
therefore from now on we will suppose $g\equiv 0$.
Note also that since $v$ is continuous and bounded then
$(x,t)\mapsto f(x,v(x)+t)$ satisfies \it f.1) \rm too.

Modify $f$ by defining
$$
F(x,u)=\left\lbrace\begin{array}{ll}
f(x,\overline{u}) & \hbox{if }u>\overline{u}(x) \\
f(x,u) & \hbox{if }\underline{u}(x)\leq u\leq \overline{u}(x) \\
f(x,\underline{u}) & \hbox{if }u<\underline{u}(x)
\end{array}\right.\qquad\hbox{ for every }x\in\overline{\Omega},\, u\in\R:
$$
the function $F(x,u)$ is continuous and bounded on $\Omega\times\R$,
by hypothesis \it f.1)\rm. We can write a solution of 
$$
(\ast)\qquad\left\lbrace\begin{array}{ll}
\Ds u=-F(x,u) & \hbox{ in }\Omega \\
u=0 & \hbox{ in }\C\Omega \\
Eu=0 & \hbox{ on }\partial\Omega
\end{array}\right.
$$
as a fixed-point of the map obtained as the composition
$$
\begin{array}{ccccl}
L^\infty(\Omega) & 
\longrightarrow & L^\infty(\Omega) & 
\longrightarrow & L^\infty(\Omega) \\
u & \longmapsto & -F(x,u(x)) & \longmapsto & w \:\hbox{ s.t. }\Ds w=-F(x,u(x))\hbox{ in }\Omega,\ w=0\hbox{ in }\C\Omega,
\ Ew=0\hbox{ on }\partial\Omega.
\end{array}
$$
The first map sends $L^\infty(\Omega)$ 
in a bounded subset of $L^\infty(\Omega)$,
by continuity of $F$ and boundedness of $\underline{u},\,\overline{u}$.
The second map is compact since $w\in C^s(\R^n)$, thanks to
the results in \cite[Proposition 1.1]{rosserra}.
Then the composition admits a fixed point in view of the Shauder Fixed Point Theorem.

Note that a solution to the original problem lying between $\underline{u}$ and $\overline{u}$ in $\Omega$, 
is also a solution of $(\ast)$.
Moreover, \it any \rm solution of $(\ast)$ is between $\underline{u}$ and $\overline{u}$.
Indeed consider $A:=\{x\in\Omega:u(x)>\overline{u}(x)\}$, which is open by the continuity of $u$ and $\overline{u}$.
For any $\psi\in C^\infty_c(A)$, $\psi\geq 0$,
with the corresponding $\phi\in\T(A)$:
$$
\int_A \overline{u}(x)\,\psi(x)\;dx \geq -\int_A F(x,\overline{u}(x))\,\phi(x)\;dx\ =
\ -\int_A F(x,u(x))\,\phi(x)\;dx=\int_A u(x)\,\psi(x)\;dx
$$
which implies $u\leq\overline{u}$ in $A$,
by positivity of $\psi$, proving $A=\varnothing$.

\it Uniqueness. \rm If we have two continuous solutions $u$ and $w$
$$
\left\lbrace\begin{array}{ll}
\Ds u=-f(x,u) & \hbox{ in }\Omega \\
u=g & \hbox{ in }\C\Omega \\
Eu=0 & \hbox{ in }\partial\Omega
\end{array}\right.
\qquad\qquad
\left\lbrace\begin{array}{ll}
\Ds w=-f(x,w) & \hbox{ in }\Omega \\
w=g & \hbox{ in }\C\Omega \\
Ew=0 & \hbox{ in }\partial\Omega
\end{array}\right.
$$
then for the difference $u-w$ it is 
$$
\left\lbrace\begin{array}{ll}
\Ds(u-w)=-f(x,u)+f(x,w) & \hbox{ in }\Omega \\
u-w=0 & \hbox{ in }\C\Omega \\
Eu-Ew=0 & \hbox{ in }\partial\Omega.
\end{array}\right.
$$
Defining $\Omega_1=\{x\in\Omega:w(x)<u(x)\}$, thanks to the monotony of $f$,
$$
\left\lbrace\begin{array}{ll}
\Ds(u-w)\leq 0 & \hbox{ in }\Omega_1 \\
u-w\leq 0 & \hbox{ in }\C\Omega_1 \\
Eu-Ew=0 & \hbox{ in }\partial\Omega_1
\end{array}\right.
$$
but then, according to Lemma \ref{max-princweak2}, $u\leq w$ in $\Omega_1$.
This means $\Omega_1$ is empty. By reversing the roles of $u$ and $w$,
we deduce $u=w$ in $\Omega$.

\it Minimal solution. \rm
We refer the reader to the proof in \cite[Corollary 2.2]{KO-dupaigne}.

\subsection{Proof of Theorem \ref{large-building}}

In the case of negative right-hand side,
Theorem \ref{large-building} follows from Theorem \ref{-sign}.
So, assume the right-hand side is positive and consider 
$$
\left\lbrace\begin{array}{l}
\Ds v= f(x,v) \quad \hbox{ in }\Omega, \\ \displaystyle
\lim_{\stackrel{\hbox{\scriptsize $x\!\rightarrow\!\partial\Omega$}}{x\in\Omega}}v(x)=+\infty.
\end{array}\right.
$$
We look for a suitable shape $g$ of $v$ outside $\Omega$
and exploding at $\partial\Omega$:
the large $s$-harmonic function $v_0$ induced by $g$
in $\Omega$ will be a subsolution of our equation,
and in particular will imply that the blow-up condition 
at $\partial\Omega$ is fulfilled.
Then, in order to prove the existence part,
we need a supersolution.

Consider $F:\R\rightarrow\R$ continuous, increasing
and such that $F(t)\geq f(x,t)$ for any $t\geq 0$:
for example,
$$
F(t)={t}^{\frac{4s}{1-s}}+\max_{0\leq r\leq t}\left[\max_{x\in\Omega}f(x,r)\right].
$$
Choose
$$
g(x):=\frac{F^{-1}(I(x))-1}{\overline{c}},
$$
where $\overline{c}=\overline{c}(n,s,\Omega)$ is the constant of equation
\eqref{uleqg} giving the upper control
of large $s$-harmonic functions in terms of the boundary datum
(see Paragraph \ref{expl-sharm-sec}) and
$$
I(x)=\min\left\lbrace\A(n,s)\int_{\C\Omega}\frac{dy}{{|z-y|}^{n+2s}}:z\in\Omega,\ 
\delta(z)=\delta(x)\right\rbrace,\qquad x\in\R^n,\ \delta(x)\leq\max_{x'\in\Omega}\delta(x')
$$
and $I(x)=0$ when $\displaystyle x\in\C\Omega,\ \delta(x)>\max_{x'\in\Omega}\delta(x')$.
Note that when $\delta(x)$ is small
$$
I(x)\leq \A(n,s)\,\omega_{n-1}\int_{\delta(x)}^{+\infty}\frac{d\rho}{{\rho}^{1+2s}}=\frac{\A(n,s)\,\omega_{n-1}}{2s}\cdot\frac{1}{{\delta(x)}^{2s}}.
$$
Such $g$ satisfies hypothesis \eqref{g}, since when $\delta(x)$ is small
$$
F(t)\geq {t}^{\frac{4s}{1-s}}\quad\Longrightarrow\quad
F^{-1}(t)\leq {t}^{\frac{1-s}{4s}}\quad\Longrightarrow\quad
F^{-1}\left(I(x)\right)\leq\left(\frac{\A(n,s)\,\omega_{n-1}}{2s}\right)^{\frac{1-s}{4s}}\frac{1}{{\delta(x)}^{(1-s)/2}}.
$$

Call $v_0$ the solution to
$$
\left\lbrace\begin{array}{ll}
\Ds v_0 = 0 & \hbox{ in }\Omega, \\
v_0=g & \hbox{ in }\C\Omega, \\
Ev_0=0 & \hbox{ on }\partial\Omega.
\end{array}\right.
$$
Denote by $w:=v-v_0$: our claim
is that problem
$$
\left\lbrace\begin{array}{ll}
\Ds w=f(x,v_0+w) & \hbox{ in }\Omega \\
w=0 & \hbox{ in }\C\Omega \\
Ew=0 & \hbox{ on }\partial\Omega
\end{array}\right.
$$
admits a solution $w$.
Indeed, we have a subsolution which is
the function constant to $0$ in $\Omega$
and $\chi_\Omega$ turns out to be a supersolution.
To show this we consider the problem
$$
\left\lbrace\begin{array}{ll}
\Ds w\geq F(v_0+w) & \hbox{ in }\Omega \\
w=0 & \hbox{ in }\C\Omega \\
Ew=0 & \hbox{ on }\partial\Omega
\end{array}\right.
$$
and observe that, when $x\in\Omega$
$$
\Ds\chi_\Omega(x)=\A(n,s)\int_{\C\Omega}\frac{dy}{{|x-y|}^{n+2s}}\geq I(x)
$$
and
$$
F(v_0(x)+1)\leq F(\overline{c}g(x)+1)=F(F^{-1}(I(x)))=I(x).
$$
Finally, the property $F(v_0+w)\geq f(x,v_0+w)$
concludes the construction of the supersolution.
Then Lemma \ref{sub+super} below provides 
the existence of a solution.

\subsection{Damping term: proof of Theorem \ref{-sign}}

For any $N\in\N$, denote by $g_N=\min\{g,N\}$. Also, with the notation of equation \eqref{h-approx},
for a small parameter $r>0$ denote by
$$
f_r(y)=f_r(\rho,\theta)=h(\theta)\,\frac{\varphi(\rho/r)}{K_r},
\qquad K_r=\frac{1}{1+s}\int_{\Omega_r}\varphi(\delta(y)/r)\,{\delta(y)}^s\;dy,
$$
and recall that this is an approximation of the $h$ boundary datum.
Finally call $u_{N,r}$ the minimal solution of
$$
\left\lbrace\begin{array}{ll}
\Ds u_{N,r}(x)=-f(x,u_{N,r}(x))+f_r(x) & \hbox{ in }\Omega \\
u_{N,r}=g_N & \hbox{ in }\C\Omega \\
Eu_{N,r}=0 & \hbox{ on }\partial\Omega
\end{array}\right.
$$
provided by Theorem \ref{nl-cs}.
Note that for any $r>0$, 
the sequence $\{u_{N,r}\}_{N\in\N}$ we obtain is increasing in $N$:
indeed, $u_{N+1,r}$ is a supersolution for the problem defining $u_{N,r}$,
since it has larger boundary values and the minimality property
on $u_{N,r}$ gives $u_{N,r}\leq u_{N+1,r}$.
Moreover, $\{u_{N,r}\}_{N\in\N}$ is bounded by the 
function $u_r^0$ associated with 
the linear problem with data $g$ and $f_r$, i.e.
$$
\left\lbrace\begin{array}{ll}
\Ds u_r^0=f_r & \hbox{ in }\Omega, \\
u_r^0=g & \hbox{ in }\C\Omega, \\
Eu_r^0=0 & \hbox{ on }\partial\Omega.
\end{array}\right.
$$
Therefore $u_{N,r}$ admits a pointwise limit in $\R^n$. Call $u_r$ this limit: 
obviously $u_r=g$ in $\C\Omega$. Take any nonnegative $\phi\in\T(\Omega)$
with $0\leq\psi=\Ds\phi|_\Omega\in C^\infty_c(\Omega)$: then
\begin{multline*}
\int_\Omega[f(x,u_r)-f_r]\,\phi\leq\liminf_{N\uparrow+\infty}\int_\Omega[f(x,u_{N,r})-f_r]\,\phi
=-\limsup_{N\uparrow+\infty}\int_\Omega u_{N,r}\psi-\int_{\C\Omega}g\,\Ds\phi=\\
=-\int_\Omega u_r\psi-\int_{\C\Omega}g\,\Ds\phi,
\end{multline*}
where we have used the Fatou lemma and the continuity 
of the map $t\mapsto f(x,t)$.
This means that $u_r$ is a subsolution.
We are left to prove that $u_r$ is also a supersolution.
Call $\Omega'=$ supp$\psi\subset\subset\Omega$ 
and build a sequence ${\{\Omega_k\}}_{k\in\N}$
such that $\Omega'\subseteq\Omega_k\subseteq\Omega$ and $\Omega_k\nearrow\Omega$.
Since $\psi\in C^\infty_c(\Omega_k)$ for any $k$, then the we can build 
the sequence of functions $\phi_k\in\T(\Omega_k)$ induced by $\psi$:
this sequence is increasing and converges pointwisely to $\phi$.
Moreover, for any $k$, since $\Ds\phi_k\leq0$ in $\C\Omega_k$
\begin{multline*}
\int_\Omega[f(x,u_r)-f_r]\,\phi_k=\lim_{N\uparrow+\infty}\int_{\Omega_k}[f(x,u_{N,r})-f_r]\,\phi_k
=\lim_{N\uparrow+\infty}\left(-\int_{\Omega_k}u_{N,r}\psi-\int_{\C\Omega_k}u_{N,r}\,\Ds\phi_k\right)\geq\\
\geq\lim_{N\uparrow+\infty}\left(-\int_{\Omega}u_{N,r}\psi-\int_{\C\Omega}g_N\,\Ds\phi_k\right)
=-\int_\Omega u_r\psi-\int_{\C\Omega}g\,\Ds\phi_k:
\end{multline*}
letting both sides of the inequality pass to the limit as $k\uparrow+\infty$
we obtain
$$
\int_\Omega[-f(x,u_r)-f_r]\,\phi\geq-\int_\Omega u_r\psi-\int_{\C\Omega}g\,\Ds\phi,
$$
because recall that for $x\in\C\Omega\subseteq\C\Omega_k$
$$
-\Ds\phi_k(x)=\int_{\Omega_k}\frac{\phi_k(y)}{{|x-y|}^{n+2s}}\;dy
$$
increases to $-\Ds\phi(x)$.

This means that $u_r$ is both a sub- and a supersolution
and it solves
\begin{equation}\label{ur}
\left\lbrace\begin{array}{ll}
\Ds u_r(x)=-f(x,u_r(x))+f_r(x) & \hbox{ in }\Omega, \\
u=g & \hbox{ in }\C\Omega, \\
Eu=0 & \hbox{ on }\partial\Omega.
\end{array}\right.
\end{equation}

\begin{rmk}\rm Note that we have just solved all problems 
with null $h$ boundary datum, i.e.
$$
\left\lbrace\begin{array}{ll}
\Ds u(x)=-f(x,u(x)) & \hbox{ in }\Omega \\
u=g & \hbox{ in }\C\Omega \\
Eu=0 & \hbox{ on }\partial\Omega
\end{array}\right.
$$
since, in this case, it is obviously $f_r\equiv 0$.
\end{rmk}

We want now to push $r\downarrow0$, under the additional assumption
$$
f(x,t)\leq a_1+a_2{t}^p,\quad \hbox{ for }t>0,\ p<\frac{1+s}{1-s}.
$$ 
We claim that the family ${\{u_r\}}_r$ is uniformly bounded and
equicontinuous on every compact subset of $\Omega$:
there exist then a $u\in L^1(\Omega)$ (since $|u_r(x)|\leq C\delta(x)^{s-1}$ for a $C$ independent on $r$) 
and a sequence ${\{r_k\}}_{k\in\N}$
such that $r_k\rightarrow 0$ as $k\rightarrow+\infty$ and
$u_{r_k}\rightarrow u$ a.e. in $\Omega$ and uniformly
on compact subsets. Then for any $\phi\in\T(\Omega)$
$$
\begin{array}{ccccc}
\displaystyle
\int_\Omega u_{r_k}\Ds\phi & = & \displaystyle
-\int_\Omega f(x,u_{r_k})\,\phi &
\displaystyle +\int_\Omega f_r\,\phi &
\displaystyle -\int_{\C\Omega}g\,\Ds\phi \\
\downarrow & & \downarrow & \downarrow & \\
\displaystyle
\int_\Omega u\Ds\phi & = & \displaystyle
-\int_\Omega f(x,u)\,\phi &
\displaystyle +\int_{\partial\Omega} h\,D_s\phi &
\displaystyle -\int_{\C\Omega}g\,\Ds\phi
\end{array}
$$
where the convergence $\int_\Omega f(x,u_{r_k})\phi\rightarrow\int_\Omega f(x,u)\phi$
holds since
$$
f(x,u_{r_k})|\phi|\leq\left(a_1+a_2 u_{r_k}^p\right)\,C{\delta(x)}^s
\leq\left(a_1+a_2 u_0^p\right)\,C{\delta}^s\leq
\left(a_1+a_2 {\delta}^{p(s-1)}\right)\,C{\delta}^s
$$
and $p(s-1)+s>-1$ by hypothesis.
We still have to prove that our claim is true.
The uniform boundedness on compact subsets of ${\{u_r\}}_r$ is a
consequence of inequalities
$$
0\leq u_r \leq v_r \xrightarrow{r\downarrow 0} u_0,\quad\hbox{ for any }r,
$$
where
$$
\left\lbrace\begin{array}{ll}
\Ds v_r=f_r & \hbox{ in }\Omega \\
v_r=g & \hbox{ in }\C\Omega \\
Ev_r=0 & \hbox{ on }\partial\Omega,
\end{array}\right.
\qquad\qquad
\left\lbrace\begin{array}{ll}
\Ds u_0=0 & \hbox{ in }\Omega \\
u_0=g & \hbox{ in }\C\Omega \\
Eu_0=h & \hbox{ on }\partial\Omega,
\end{array}\right.
$$
and the convergence of $v_r$ to $u_0$ is uniform
in compact subsets of $\Omega$.
Then
\begin{multline*}
|u_r(x)-u_r(z)|  = 
\left|\int_\Omega\left(G_\Omega(x,y)-G_\Omega(z,y)\right)\,[-f(y,u_r(y))+f_r(y)]\;dy\right| \leq \\
\leq \int_\Omega\left|G_\Omega(x,y)-G_\Omega(z,y)\right|\left(a_1+C\,a_2{\delta(y)}^{p(s-1)}\right)\;dy 
+\int_\Omega\left|G_\Omega(x,y)-G_\Omega(z,y)\right|\,C\,{\delta(y)}^{-1}\;dy \\
\end{multline*}
implies the equicontinuity.
\bigskip

Note that this proof exploits the negativity of the right-hand side 
only in considering the $s$-harmonic function induced by
$g$ and $h$ as a supersolution of problem 
$$
\left\lbrace\begin{array}{ll}
-\Ds u=f(x,u) & \hbox{ in }\Omega \\
u=g & \hbox{ in }\C\Omega \\
Eu=h & \hbox{ in }\partial\Omega.
\end{array}\right.
$$
With minor modifications to the proof we can state

\begin{lem}\label{sub+super}
Let $f:\overline{\Omega}\times\R\rightarrow[0,+\infty)$ be a function 
satisfying f.1) and f.2).
Let $g:\C\Omega\rightarrow\R^+$ be a measurable function
satisfying \eqref{gintro} and $h\in C(\partial\Omega),\ h\geq0$.
Assume the nonlinear problem
$$
\left\lbrace\begin{array}{ll}
\Ds u=f(x,u) & \hbox{ in }\Omega \\
u=g & \hbox{ in }\C\Omega \\
Eu=h & \hbox{ on }\partial\Omega
\end{array}\right.
$$
admits a subsolution $\underline{u}\in L^1(\Omega)$ and a supersolution $\overline{u}\in L^1(\Omega)$.
Assume also $\underline{u}\leq\overline{u}$ in $\Omega$.
Then the above nonlinear problem has a weak solution $u\in L^1(\Omega)$ 
satisfying
$$
\underline{u}\leq u\leq\overline{u}.
$$
\end{lem}
\begin{tfa} Replace, in the above proof, the function
$u^0$ with the supersolution $\overline{u}$.
\end{tfa}

\subsection{Sublinear nonnegative nonlinearity: proof of Theorem \ref{sublinear}}


We first prove a Lemma which will make the proof easily go through.

\begin{lem} There exists $m=m(\Lambda)>0$ sufficiently large for which 
any problem of the form
$$
\left\lbrace\begin{array}{llc}
\Ds u(x)=\Lambda (u(x)) & \hbox{ in }\Omega, &  \\
u=g & \hbox{ in }\C\Omega, & g\geq m>0, \\
Eu=h & \hbox{ on }\partial\Omega,
\end{array}\right.
$$
is solvable.
\end{lem}

\begin{tfa}\rm We can equivalently solve the integral equation
$$
u(x)\ =\ u_0(x)+\int_\Omega G_\Omega(x,y)\,\Lambda (u(y))\;dy,
$$
where $u_0$ is the s-harmonic function induced by $g$ and $h$ in $\Omega$.

Define the map
$$
\begin{array}{rcl}
\mathcal{K}\ :\ L^1(\Omega) & \longrightarrow & L^1(\Omega) \\
u(x) & \longmapsto & \displaystyle
u_0(x)+\int_\Omega G_\Omega(x,y)\,\Lambda (u(y))\;dy
\end{array}
$$
The condition $g\geq m$ in $\C\Omega$ entails $u_0\geq m$ in $\Omega$;
also, for any $w\in L^1(\Omega),\ w\geq 0$ implies $\mathcal{K} w\geq u_0\geq m$,
therefore $\mathcal{K}$ sends the subset $D_m:=\{w\in L^1(\Omega):w\geq m\}$ of $L^1(\Omega)$
into itself. Moreover, for $u,v\in D_m$
\begin{multline*}
\int_\Omega\left|\mathcal{K}u(x)-\mathcal{K}v(x)\right|\;dx
\leq
\int_\Omega\left|\Lambda (u(y))-\Lambda (v(y))\right|\int_\Omega G_\Omega(x,y)\;dx\;dy
\leq\\
\leq\Arrowvert\zeta\Arrowvert_\infty
\sup_{x\in\Omega}\Lambda' (u(x))\int_\Omega\left| u(y)-v(y)\right|\;dy
\leq\Arrowvert\zeta\Arrowvert_\infty
\Lambda'(m)\int_\Omega\left| u(y)-v(y)\right|\;dy
\end{multline*}
where $\zeta(y)=\int_\Omega G_\Omega(x,y)\;dx$.
Now, if $m$ is very large, we have 
$$
\Arrowvert\zeta\Arrowvert_\infty\Lambda'(m)<1,
$$
i.e. $\mathcal{K}$ is a contraction on $D_m$,
and $\mathcal{K}$ has a fixed point in $D_m$.
\end{tfa}

In general, for the problem
$$
\left\lbrace\begin{array}{ll}
\Ds u=f(x,u) & \hbox{ in }\Omega, \\
u=g & \hbox{ in }\C\Omega, \\
Eu=h & \hbox{ on }\partial\Omega,
\end{array}\right.
$$
we have a subsolution which is the $s$-harmonic function satisfying 
the boundary conditions.

But we are now able to provide a supersolution:
this can be done by setting $g_m=\max\{g,m\}$
and by solving, for some large value of $m$,
$$
\left\lbrace\begin{array}{ll}
\Ds\overline{u}=\Lambda(\overline{u})\geq f(x,\overline{u}) & \hbox{ in }\Omega,  \\
\overline{u}=g_m\geq g & \hbox{ in }\C\Omega,\\
E\overline{u}=h & \hbox{ on }\partial\Omega.
\end{array}\right.
$$

It is sufficient to apply the classical iteration scheme starting from
the $s$-harmonic function $u_0$ and with iteration step
$$
\hbox{for any }k\in\N\qquad
\left\lbrace\begin{array}{ll}
\Ds u_k=f(x,u_{k-1}(x)) & \hbox{ in }\Omega, \\
u_k=g & \hbox{ in }\C\Omega,\\
Eu_k=h & \hbox{ on }\partial\Omega
\end{array}\right.
$$
In such a way we build an increasing sequence ${\{u_k\}_{k\in\N}}\subseteq L^1(\Omega)$
which is uniformly bounded from above by $\overline{u}$.
Indeed, on the one hand we have that 
$$
\left\lbrace\begin{array}{ll}
\Ds (u_1-u_0)=f(x,u_0(x))\geq 0 & \hbox{ in }\Omega \\
u_1-u_0=0 & \hbox{ in }\C\Omega \\
Eu_1-Eu_0=0 & \hbox{ on }\partial\Omega
\end{array}\right.
$$
entails that $u_1-u_0\geq0$, while on the other hand
an induction argument relying on the monotony of $t\mapsto f(x,t)$
finally shows that $u_k$ increases.
Call $u(x):=\lim_{k}u_k(x)$, which is finite in view of 
the upper bound furnished by $\overline{u}$.
Then
$$
u(x)=\lim_{k\rightarrow+\infty}u_k(x)=u_0(x)+
\lim_{k\rightarrow+\infty}\int_\Omega f(y,u_{k-1}(y))\,G_\Omega(x,y)\; dy=
u_0(x)+\int_\Omega f(y,u(y))\,G_\Omega(x,y)\; dy.
$$

\subsection{Superlinear nonnegative nonlinearity: proof of Theorem \ref{superlinear}}

We give the proof for problem 
$$
\left\lbrace\begin{array}{ll}
\Ds u = f(x,u) & \hbox{ in }\Omega \\
u(x)=\delta(x)^{-\beta} & \hbox{ in }\C\Omega,\ 0<\beta<1+s \\
Eu=0 & \hbox{ on }\partial\Omega
\end{array}\right.
$$
while for the other one it is sufficient to replace $\beta$ with $1-s$
and repeat the same computations.

To treat the case of a general nonlinearity
we use again the equivalent integral equation
$$
u(x)\ =\ u_0(x)+\lambda\int_\Omega G_\Omega(x,y)\,f(y,u(y))\;dy,
$$
where $u_0$ is the s-harmonic function induced in $\Omega$ by the boundary data.
In this case the computations in Section \ref{bblin-sec} on the 
rate of explosion at the boundary turn out to be very useful.
Indeed on the one hand we have that $u_0$ inherits its explosion
from the boundary data $g$ and $h$: briefly, in our case
\begin{equation*}
g(x)= \frac{1}{{\delta(x)}^\beta},\quad 0<\beta<1-s
\qquad  \longrightarrow \qquad
\frac{\underline{c}}{{\delta(x)}^\beta}\leq u_0(x)\leq \frac{\overline{c}}{{\delta(x)}^\beta}.
\end{equation*}
Since $u_0$ is a subsolution,
our first goal is to build a supersolution and we build it of the form
$$
\overline{u}=u_0+\zeta,
$$
where
$$
\left\lbrace\begin{array}{ll}
\Ds\zeta={\delta(x)}^{-\gamma} & \hbox{ in }\Omega,\quad\gamma>0\\
\zeta=0 & \hbox{ in }\C\Omega, \\
E\zeta=0 & \hbox{ on }\partial\Omega.
\end{array}\right.
$$
Recall that, \eqref{udelta} says
$$
\zeta(x)\geq \left\lbrace\begin{array}{ll}\displaystyle
c_1{\delta(x)}^s & \hbox{ if }0\leq\gamma<s, \\ \displaystyle 
c_3{\delta(x)}^s\,\log\frac{1}{\delta(x)} & \hbox{ if }\gamma=s, \\ \displaystyle 
c_5{\delta(x)}^{-\gamma+2s} & \hbox{ if }s<\gamma<1+s.
\end{array}\right.
$$
The function $\overline{u}$ is a supersolution if
$$
\overline{u}(x)\geq u_0(x)+\lambda\int_\Omega G_\Omega(x,y)\,f(y,\overline{u}(y))\;dy,
$$
or, equivalently formulated
\begin{equation}\label{12}
\zeta(x)\geq\lambda\int_\Omega G_\Omega(x,y)\,f(y,u_0(y)+\zeta(y))\;dy.
\end{equation}

If $f(x,t)$ has an algebraic behavior\footnote{for $p<1$ we are actually in the case of the previous paragraph}
$$
f(x,t)\leq a_1+a_2\,{t}^p,\quad a_1,a_2>0,\ p\geq 1
$$
then
$$
f(y,u_0(y)+\zeta(y))\leq a_1+a_2\,\left(u_0(y)+\zeta(y)\right)^p\leq
\left\lbrace\begin{array}{ll}
\displaystyle\frac{C}{{\delta(x)}^{p\beta}} & \hbox{ if }\gamma-2s\leq\beta, \\
\displaystyle\frac{C}{{\delta(x)}^{p(\gamma-2s)}} & \hbox{ if }\gamma-2s>\beta.
\end{array}\right.
$$
In case $\gamma-2s\leq\beta$ we have
$$
\int_\Omega G_\Omega(x,y)\,f(y,u_0(y)+\zeta(y))\;dy\leq 
\left\lbrace\begin{array}{ll}
\displaystyle c_2\,C{\delta(x)}^s & \hbox{ if }p\beta<s \\
\displaystyle c_4\,C{\delta(x)}^s\log\frac{1}{\delta(x)} & \hbox{ if }p\beta= s \\
\displaystyle\frac{c_6\,C}{{\delta(x)}^{p\beta-2s}} & \hbox{ if }s< p\beta<1+s 
\end{array}\right.
$$
again in view of Paragraph \ref{rhs-blowing}, so that
we can choose $\gamma=p\beta$ provided $p\beta-2s\leq\beta$.

If this is not the case then it means we need powers $\gamma$
satisfying $\gamma-2s>\beta$. If $\gamma-2s>\beta$ we have
$$
\int_\Omega G_\Omega(x,y)\,f(y,u_0(y)+\zeta(y))\;dy\leq 
\left\lbrace\begin{array}{ll}
\displaystyle c_2\,C{\delta(x)}^s & \hbox{ if }p(\gamma-2s)<s \\
\displaystyle c_4\,C{\delta(x)}^s\log\frac{1}{\delta(x)} & \hbox{ if }p(\gamma-2s)= s \\
\displaystyle\frac{c_6\,C}{{\delta(x)}^{p(\gamma-2s)-2s}} & \hbox{ if }s< p(\gamma-2s)<1+s 
\end{array}\right.
$$
and a suitable choice for $\gamma$ would be 
$$
\gamma=\max\left\lbrace \frac{2s\,p}{p-1},\beta+2s+\eps\right\rbrace
$$
which fulfills both inequalities
$$
\gamma-2s>\beta,\qquad\gamma\geq p(\gamma-2s).
$$
This is an admissible choice for $\gamma$ provided $\gamma<1+s$, i.e. only if
$$
p>\frac{1+s}{1-s};
$$
in case $p$ doesn't satisfy this lower bound then
$$
p\leq\frac{1+s}{1-s} \quad \Longrightarrow \quad p\beta\leq \frac{1+s}{1-s}\,\beta\leq \beta+2s
$$
and we are in the previous case.

Finally, if $q\beta>1+s$ then a solution $u$ should satisfy, whenever $\delta(x)<1$,
$$
f(x,u(x))\geq b {u(x)}^q \geq b{u_0(x)}^q \geq c{\delta(x)}^{-q\beta}
$$
which would imply
$$
\int_\Omega G_\Omega(x,y)\,f(y,u(y))\;dy=+\infty,\quad x\in\Omega,
$$
which means that the problem is not solvable.


\subsection{Complete blow-up: proof of Theorem \ref{complete-blowup}}

Let us first prove the theorem in the case of null boundary data.
The first claim is that
$$
\int_\Omega f_k(x,u_k(x))\,{\delta(x)}^s\;dx\ \xrightarrow{k\uparrow+\infty}\ +\infty.
$$
Suppose by contradiction that the sequence 
of integrals is bounded by a constant $C$.
Consider an increasing sequence of nonnegative
${\{\psi_N\}}_{N\in\N}\subseteq C^\infty_c(\Omega)$
such that $\psi_N\uparrow1$ in $\Omega$
and choose $\phi_N\in\T(\Omega)$ in such a way that
$\Ds\phi_N=\psi_N$ holds in $\Omega$. Then
$$
\int_\Omega u_k\,\psi_N=\int_\Omega f_k(x,u_k)\,\phi_N\leq c\int_\Omega f_k(x,u_k(x))\,{\delta(x)}^s\;dx
$$
for some constant $c>0$ not depending on $N$, see \cite[Proposition 1.1]{rosserra}.
By letting $N\uparrow+\infty$, we deduce that $u_k$ is a bounded sequence in $L^1(\Omega)$.
Take now $\underline{u}_k$ as the minimal solution to the $k$-th nonlinear problem:
since $\Ds\underline{u}_{k+1}=f_{k+1}(x,\underline{u}_{k+1})\geq f_k(x,\underline{u}_{k+1})$
then $\underline{u}_k$ is an increasing sequence and it admits 
a pointwise limit $u$. Also, this $u$ is limit also in the $L^1$-norm,
since $\underline{u}_k$ is bounded in this norm.
But then, for any $\phi\in\T(\Omega)$, we have
$$
\int_\Omega u\,\Ds\phi=\lim_{k\uparrow+\infty}\int_\Omega\underline{u}_k\,\Ds\phi=
\lim_{k\uparrow+\infty}\int_\Omega f_k(x,\underline{u}_k)\,\phi=
\int_\Omega f(x,u)\,\phi,
$$
that is $u\in L^1(\Omega)$ would be a weak solution, a contradiction.

Our second claim is that
$$
u_k(x)\geq c\left[\int_\Omega f_k(y,u_k(y))\,{\delta(y)}^s\;dy\right]\,{\delta(x)}^s,
$$
for some constant $c>0$ independent of $k$.
To do this, we exploit \eqref{est-green}. Call
$\Omega_1(x)=\{y\in\Omega:|y-x|\leq\delta(x)\delta(y)\}$,
$\Omega_2(x)=\{y\in\Omega:|y-x|>\delta(x)\delta(y)\}$
and $d(\Omega)=\sup\{|y-x|:x,\,y\in\Omega\}$:
$$
\begin{array}{rcl}
u_k(x) & = & \displaystyle
\int_\Omega G_\Omega(x,y)\,f_k(y,u_k(y))\;dy  \\
& & \\
& \geq & \displaystyle
c_2 \int_\Omega\left[{|x-y|}^{2s}\wedge{\delta(x)}^s{\delta(y)}^s\right]f_k(y,u_k(y))\;\frac{dy}{{|x-y|}^n}  \\
& & \\
& = & c_2 \displaystyle
\int_{\Omega_1(x)}f_k(y,u_k(y))\;\frac{dy}{{|x-y|}^{n-2s}} +
c_2\,{\delta(x)}^s \int_{\Omega_2(x)}{\delta(y)}^s f_k(y,u_k(y))\;\frac{dy}{{|x-y|}^n}  \\
& & \\
& \geq & \displaystyle
\frac{c_2}{{d(\Omega)}^{2s}} \int_{\Omega_1(x)}{\delta(x)}^s{\delta(y)}^s f_k(y,u_k(y))\;\frac{dy}{{|x-y|}^{n-2s}} +
\frac{c_2}{{d(\Omega)}^n}\,{\delta(x)}^s \int_{\Omega_2(x)}{\delta(y)}^s f_k(y,u_k(y))\;dy  \\
& & \\
& \geq & \displaystyle
\frac{c_2}{{d(\Omega)}^n}\,{\delta(x)}^s \int_\Omega{\delta(y)}^s f_k(y,u_k(y))\;dy  
\end{array}
$$
and we see then how we have complete blow-up.

In the case of nonhomogeneous boundary conditions,
we consider the $s$-harmonic function $u_0$ induced by data $g$ and $h$,
and we denote by $F(x,t)=f(x,u_0(x)+t)$ for $x\in\Omega$, $t\geq 0$.
By hypothesis we have then that there is no weak solution to
$$
\left\lbrace\begin{array}{ll}
\Ds v = F(x,v) & \hbox{ in }\Omega, \\
v=0 & \hbox{ in }\C\Omega, \\
Ev=0 & \hbox{ on }\partial\Omega.
\end{array}\right.
$$
Since any monotone approximation on $f$ is also a monotone approximation of $F$,
then there is complete blow-up in the problem for $v$
and this bears the complete blow-up for the problem on $u$.

\appendix

\section{Proof of Proposition \ref{integrbypartsform}}

Assume first that $u\in\mathcal{S}$ and $v=0$ in $\R^n\setminus\Omega$, $v\in C^{2s+\eps}(\Omega)\cap C(\overline{\Omega})$
and $\Ds v\in L^1(\R^n)$; 
then we can regularize $v$, 
via the convolution with a mollifier ${\{\alpha_k(x)=k^n\alpha(kx):\alpha(x)=0\hbox{ for }|x|\geq1\}}_{k\in\N}$ 
in order to obtain a sequence
${\{v_k:=\alpha_k*v\}}_{k\in\N}\subseteq C^\infty_c(\R^n)\subseteq\mathcal{S}$
converging uniformly to $v$ in $\R^n$.
Also,
$$
\Ds v_k=v*\Ds\alpha_k
$$
indeed 
\begin{multline*}
\Ds v_k(x)=\A(n,s)\,PV\int_{\R^n}\frac{v_k(x)-v_k(y)}{{|x-y|}^{n+2s}}\;dy=\\=
\A(n,s)\,\lim_{\eps\downarrow0}\int_{\C B_\eps(x)}\frac{\int_{\R^n}v(z)[\alpha_k(x-z)-\alpha_k(y-z)]\;dz}{{|x-y|}^{n+2s}}\;dy=\\=
\A(n,s)\lim_{\eps\downarrow 0}\int_{\R^n}v(z)\int_{\C B_\eps(x)}\frac{\alpha_k(x-z)-\alpha_k(y-z)}{{|x-y|}^{n+2s}}\;dy\;dz=\\=
\A(n,s)\lim_{\eps\downarrow 0}\int_{\R^n}v(z)\int_{\C B_\eps(x-z)}\frac{\alpha_k(x-z)-\alpha_k(y)}{{|x-z-y|}^{n+2s}}\;dy\;dz
\end{multline*}
where we have
$$
\A(n,s)\int_{\C B_\eps(x)}\frac{\alpha_k(x)-\alpha_k(y)}{{|x-y|}^{n+2s}}\;dy
\xrightarrow[\eps\downarrow0]{}\Ds\alpha_k(x),\quad\hbox{ uniformly in }\R^n
$$
since, see \cite[Lemma 3.2]{hitchhiker},
\begin{eqnarray}
 & & \left|\Ds\alpha_k(x)-\A(n,s)\int_{\C B_\eps(x)}\frac{\alpha_k(x)-\alpha_k(y)}{{|x-y|}^{n+2s}}\;dy\right|= \nonumber \\
& & =\left|\A(n,s)\,PV\int_{B_\eps(x)}\frac{\alpha_k(x)-\alpha_k(y)}{{|x-y|}^{n+2s}}\;dy\right| 
=\left|\frac{\A(n,s)}{2}\int_{B_\eps}\frac{\alpha_k(x+y)+\alpha_k(x-y)-2\alpha_k(x)}{{|y|}^{n+2s}}\;dy\right| \nonumber \\
& & \leq \frac{\A(n,s)\,\Arrowvert\alpha_k\Arrowvert_{C^2(B_\eps(x))}}{2}\int_{B_\eps}\frac{dy}{{|y|}^{n+2s-2}} \label{111} \\
& & \leq \frac{\A(n,s)\,\Arrowvert\alpha_k\Arrowvert_{C^2(\R^n)}}{2}\int_{B_\eps}\frac{dy}{{|y|}^{n+2s-2}}
\xrightarrow[\eps\downarrow0]{}0 \nonumber.
\end{eqnarray}
Following the very same proof up to \eqref{111}, since $v\in C^{2s+\eps}(\Omega)$, it is also possible to prove
$$
\Ds v_k(x)=\left(\alpha_k*\Ds v\right)(x),\quad\hbox{ for }\delta(x)>\frac{1}{k}.
$$
Since $\Ds v\in C(\R^n\setminus\Omega)$, see \cite[Proposition 2.4]{silvestre}, we infer that
$$
\Ds v_k(x)\xrightarrow[k\uparrow+\infty]{}\Ds v,\quad\hbox{ for every } x\in\R^n\setminus\partial\Omega.
$$

We give now a pointwise estimate on $\Ds\alpha_k(x),\ x\in B_{1/k}$.
Since $\alpha_k\in C^\infty_c(\R^n)$, we can write (see \cite[Paragraph 2.1]{silvestre})
$$
\Ds\alpha_k(x)=\left[{(-\Delta)}^{s-1}\circ(-\Delta)\right]\alpha_k(x)={(-\Delta)}^{s-1}[-k^{n+2}\Delta\alpha(kx)]
=\A(n,-s)\int_{\R^n}\frac{k^{n+2}\Delta\alpha(ky)}{{|x-y|}^{n+2s-2}}\;dy.
$$
With a change of variable we entail
$$
\Ds\alpha_k(x)=\A(n,-s)\,k^{n+2s}\int_B\frac{\Delta\alpha(y)}{{|k\,x-y|}^{n+2s-2}}\;dy
$$
and therefore
\begin{equation}\label{dsalfa}
|\Ds\alpha_k(x)|\leq \frac{\omega_{n-1}}{2\,(1-s)}\, k^{n+2s}\Arrowvert\Delta\alpha\Arrowvert_{L^\infty(\R^n)}.
\end{equation}
Indeed,
$$
\int_B\frac{dy}{{|z-y|}^{n+2s-2}}
$$
is a bounded function of $z$, having in $z=0$ its maximum $\frac{\omega_{n-1}}{2\,(1-s)}$.
Indeed,
\begin{multline*}
\int_B\frac{dy}{{|z-y|}^{n+2s-2}}=\int_{B\cap B(z)}\frac{dy}{{|z-y|}^{n+2s-2}}+\int_{B\setminus B(z)}\frac{dy}{{|z-y|}^{n+2s-2}}=\\
=\int_{B\cap B(z)}\frac{dy}{{|y|}^{n+2s-2}}+\int_{B\setminus B(z)}\frac{dy}{{|z-y|}^{n+2s-2}}
\leq \int_{B\cap B(z)}\frac{dy}{{|y|}^{n+2s-2}}+\int_{B\setminus B(z)}\frac{dy}{{|y|}^{n+2s-2}}
= \int_B\frac{dy}{{|y|}^{n+2s-2}}.
\end{multline*}
Using \eqref{dsalfa}, the $L^1$-norm of $\Ds v_k$ can be estimated by
\begin{multline*}
\begin{split}
\int_{\R^n}|\Ds v_k(x)|\;dx=\int_{\{\delta(x)<1/k\}}|(v*\Ds\alpha_k)(x)|\;dx
+\int_{\{\delta(x)\geq1/k\}}|(\alpha_k*\Ds v)(x)|\;dx\leq\\\leq
\int_{\{\delta(x)<1/k\}}\int_{B_{1/k}(x)}|v(y)\Ds\alpha_k(x-y)|\;dy\;dx\ +\qquad\qquad\qquad\qquad\qquad\\
+\int_{\{\delta(x)\geq1/k\}}\int_{B_{1/k}}|\alpha_k(y)\Ds v(x-y)|\;dy\;dx\leq\\\leq
\int_{\{\delta(x)<1/k\}}\int_{B_{1/k}(x)}C{\delta(y)}^sk^{n+2s}\Arrowvert\Delta\alpha\Arrowvert_{L^\infty(\R^n)}\;dy\;dx\ +\qquad\qquad\qquad\qquad\qquad\\
+\int_{B_{1/k}}\alpha_k(y)\int_{\R^n}|\Ds v(x-y)|\;dx\;dy\leq\\
\frac{C\,\Arrowvert\Delta\alpha\Arrowvert_{L^\infty(\R^n)}}{k^{1-s}}+\Arrowvert\Ds v\Arrowvert_{L^1(\R^n)},
\end{split}
\end{multline*}
so that, by the Fatou's Lemma we have
$$
\int_{\R^n}|\Ds v|\leq \liminf_{k\uparrow+\infty}\int_{\R^n}|\Ds v_k|\leq
\limsup_{k\uparrow+\infty}\int_{\R^n}|\Ds v_k|\leq \int_{\R^n}|\Ds v|
$$
which means
$$
\Arrowvert\Ds v_k\Arrowvert_{L^1(\R^n)}\xrightarrow[k\uparrow+\infty]{}\Arrowvert\Ds v\Arrowvert_{L^1(\R^n)}.
$$
Apply the Fatou's lemma to $|\Ds v_k|+|\Ds v|-|\Ds v_k-\Ds v|\geq 0$ to deduce
\begin{multline*}
2\int_{\R^n}|\Ds v|\leq \liminf_{k\uparrow+\infty}\int_{\R^n}(|\Ds v_k|+|\Ds v|-|\Ds v_k-\Ds v|)=\\
=2\int_{\R^n}|\Ds v|-\limsup_{k\uparrow+\infty}\int_{\R^n}|\Ds v_k-\Ds v|
\end{multline*}
and conclude 
$$
\Arrowvert\Ds v_k-\Ds v\Arrowvert_{L^1(\R^n)}\xrightarrow[k\uparrow+\infty]{}0
$$
and, for any $u\in\mathcal{S}$,
$$
\int_{\R^n}u\Ds v_k\xrightarrow[k\uparrow+\infty]{}\int_{\R^n}u\Ds v.
$$
Note now that
$$
\int_{\Omega_k}v_k\Ds u=\int_{\R^n}v_k\Ds u\xrightarrow[]{k\uparrow 0}\int_{\R^n}v\Ds u=\int_{\Omega}v\Ds u
$$
since $\Arrowvert v_k-v\Arrowvert_{L^\infty(\R^n)}\xrightarrow[]{k\uparrow 0}0$, and this concludes the proof.


\section{The Liouville theorem for $s$-harmonic functions}

Following the proof of the classic Liouville Theorem 
due to Nelson \cite{nelson}, it is possible to prove the analogous
result for the fractional Laplacian.

\begin{theo}\label{liouv} Let $u:\R^n\rightarrow\R$ be a function
which is $s$-harmonic throughout $\R^n$.
Then, if $u$ is bounded in $\R^n$, it is constant.
\end{theo}
\begin{tfa} Take two arbitrary points $x_1,\ x_2\in\R^n$.
Both satisfy for all $r>0$
$$
u(x_1)=\int_{\C B_r(x_1)}\eta_r(y-x_1)\,u(y)\;dy,
\qquad
u(x_2)=\int_{\C B_r(x_2)}\eta_r(y-x_2)\,u(y)\;dy
$$
Denote by $M:=\sup_{\R^n}|u|$, 
which is finite by hypothesis:
\begin{eqnarray*}
|u(x_1)-u(x_2)| & = & 
\left\arrowvert\int_{\C B_r(x_1)}\eta_r(y-x_1)\,u(y)\;dy-
\int_{\C B_r(x_2)}\eta_r(y-x_2)\,u(y)\;dy\right\arrowvert \\
& \leq & 
\int_{\C B_r(x_1)\cap B_r(x_2)}\eta_r(y-x_1)\,M\;dy+
\int_{\C B_r(x_2)\cap B_r(x_1)}\eta_r(y-x_2)\,M\;dy\ + \\
& & 
+\int_{\C B_r(x_1)\cap\C B_r(x_2)}\left\arrowvert
\eta_r(y-x_1)-\eta_r(y-x_1)\right\arrowvert\,M\;dy
\end{eqnarray*}
Define $\delta:=|x_1-x_2|$. The first addend (and similarly the second)
vanish as $r\rightarrow+\infty$:
\begin{eqnarray*}
\int_{\C B_r(x_1)\cap B_r(x_2)}\eta_r(y-x_1)\;dy & \leq & 
\int_{B_{r+\delta}\setminus B_r}\frac{c(n,s)\,r^{2s}}{|y|^n\,(|y|^2-r^2)^s}\;dy \\
& = & 
\omega_{n-1}c(n,s)\,r^{2s}\int_r^{r+\delta}\frac{d\rho}{\rho\,(\rho^2-r^2)^s} \\
& \leq & 
\omega_{n-1}c(n,s)\,\frac{1}{r^{1-s}}
\int_r^{r+\delta}\frac{d\rho}{(\rho-r)^s}\xrightarrow[]{r\uparrow+\infty}0. 
\end{eqnarray*}
The third one is more delicate.
\begin{eqnarray*}
\int_{\C B_r(x_1)\cap\C B_r(x_2)}\left\arrowvert
\frac{r^{2s}}{|y-x_1|^n\,(|y-x_1|^2-r^2)^s}-
\frac{r^{2s}}{|y-x_2|^n\,(|y-x_2|^2-r^2)^s}\right\arrowvert\;dy \\
=\int_{\C B(x_1/r)\cap\C B(x_2/r)}\left\arrowvert
\frac{1}{|y-\frac{x_1}{r}|^n\,(|y-\frac{x_1}{r}|^2-1)^s}-
\frac{1}{|y-\frac{x_2}{r}|^n\,(|y-\frac{x_2}{r}|^2-1)^s}\right\arrowvert\;dy \\
=\int_{\C B\cap\C B(\frac{x_2-x_1}{r})}\left\arrowvert
\frac{1}{|y|^n\,(|y|^2-1)^s}-
\frac{1}{|y-\frac{x_2-x_1}{r}|^n\,(|y-\frac{x_2-x_1}{r}|^2-1)^s}\right\arrowvert\;dy \\
\left[\hbox{\small setting 
$x_r=(x_2-x_1)/r$, $|x_r|=\delta/r\rightarrow 0$ as $r\rightarrow+\infty$}\right] \\
=\int_{\C B\cap\C B(x_r)}\left\arrowvert
\frac{1}{|y|^n\,(|y|^2-1)^s}-
\frac{1}{|y-x_r|^n\,(|y-x_r|^2-1)^s}\right\arrowvert\;dy \\
\left[\hbox{\small taking wlog $x_r=\frac{\delta e_1}{r}$
and defining $H_r=\{x\in\R^n:x_1>\frac{\delta}{2r},\ |x-x_r|>1\}$}\right] \\
=2\int_{H_r}\left[
\frac{1}{|y-x_r|^n\,(|y-x_r|^2-1)^s}
-\frac{1}{|y|^n\,(|y|^2-1)^s}\right]\;dy \\
\leq 2\,\omega_{n-1}
\int_1^{+\infty}\left[
\frac{1}{\rho\,(\rho^2-1)^s}
-\frac{\rho^{n-1}}{(\rho+\frac{\delta}{r})^n\,((\rho+\frac{\delta}{r})^2-1)^s}
\right]\;d\rho \\
\leq 2\,\omega_{n-1}
\int_1^{+\infty}\frac{1}{\rho\,(1+\frac{\delta}{r\rho})^n}
\left[
\frac{1}{(\rho^2-1)^s}
-\frac{\rho^{n-1}}
{((\rho+\frac{\delta}{r})^2-1)^s}\right]\;d\rho\; 
\xrightarrow[]{r\uparrow+\infty}0
\end{eqnarray*}
thanks to the Monotone Convergence Theorem.
\end{tfa}

\section{Asymptotics 
as $s\uparrow 1$}\label{meanform-app}

First of all, the proof of Theorem \ref{meanvalue} implies
$$
\lim_{s\uparrow 1}\gamma(n,s,r)=\gamma(n,1,r)=
\frac{\Gamma(n/2)}{4\,\Gamma\left(\frac{n+2}{2}\right)\,\Gamma(2)}\;r^2
=\frac{r^2}{4}\cdot\frac{\Gamma(n/2)}{\frac{n}{2}\,\Gamma(n/2)}=\frac{r^2}{2n}.
$$
Also
$$
1=\int_{\C B_r}\eta_r(y)\;dy=
\int_{\C B_r}\frac{c(n,s)\,r^{2s}}{|y|^n\left(|y|^2-r^2\right)^s}\;dy
\qquad\hbox{where }
c(n,s)=\frac{2\sin(\pi s)}{\pi\,\omega_{n-1}}
$$
where we have denoted by $\omega_{n-1}=|\partial B|_{n-1}$,
the $(n-1)$-dimensional Hausdorff measure of the unit sphere. Then
\begin{multline*}
1\ =\ \frac{2\sin(\pi s)}{\pi\,\omega_{n-1}}
\int_{\C B_r}\frac{r^{2s}}{|y|^n\left(|y|^2-r^2\right)^s}\;dy 
\ =\ \frac{2\sin(\pi s)}{\pi\,\omega_{n-1}}
\int_{\C B}\frac{1}{|y|^n\left(|y|^2-1\right)^s}\;dy \ =\\
\ =\ \frac{2\sin(\pi s)}{\pi\,\omega_{n-1}}
\int_{\partial B}\left[
\int_1^{+\infty}\frac{d\rho}{\rho\,\left(\rho^2-1\right)^s}
\right]\;d\mathcal{H}^{n-1}(\theta) \ =\ 
\frac{2\sin(\pi s)}{\pi}\int_1^{+\infty}\frac{d\rho}{\rho\,\left(\rho^2-1\right)^s}.
\end{multline*}
With similar computation
$$
\int_{\C B_r}\eta_r(y)\:u(x-y)\;dy\ =\ 
\frac{2\sin(\pi s)}{\pi}\int_{\partial B}\left[
\int_1^{+\infty}\frac{u(x-r\rho\theta)}{\rho\,\left(\rho^2-1\right)^s}\;d\rho
\right]\;d\mathcal{H}^{n-1}(\theta).
$$
Fix a $\theta\in\partial B$ and consider the difference
$$
\left|\frac{2\sin(\pi s)}{\pi}\int_1^{+\infty}\frac{u(x-r\rho\theta)}{\rho\,\left(\rho^2-1\right)^s}\;d\rho
\ -\ u(x-r\theta)\right|\ \leq\ 
\frac{2\sin(\pi s)}{\pi}\int_1^{+\infty}\frac{|u(x-r\rho\theta)-u(x-r\theta)|}{\rho\,\left(\rho^2-1\right)^s}\;d\rho:
$$
since we are handling $C^2$ functions 
we can push a bit further the estimate to deduce
\begin{multline*}
\left|\frac{2\sin(\pi s)}{\pi}\int_1^{+\infty}\frac{u(x-r\rho\theta)}{\rho\,\left(\rho^2-1\right)^s}\;d\rho
\ -\ u(x-r\theta)\right|\ \leq\ \\
\leq\ \frac{2\sin(\pi s)}{\pi}\int_1^{1+\delta}\frac{C\,\left(\rho-1\right)^{1-s}}{\rho\,\left(\rho+1\right)^s}\;d\rho
+\frac{2\sin(\pi s)}{\pi}\int_{1+\delta}^{+\infty}\frac{|u(x-r\rho\theta)-u(x-r\theta)|}{\rho\,\left(\rho^2-1\right)^s}\;d\rho
\xrightarrow[s\uparrow 1]{}0
\end{multline*}
therefore
$$
\int_{\C B_r}\eta_r(y)\:u(x-y)\;dy\ 
\xrightarrow[s\uparrow 1]{}\ 
\frac{1}{\omega_{n-1}\,r^{n-1}}\int_{\partial B_r}u(x-y)\;d\mathcal{H}^{n-1}(y).
$$
The choice of the point $z\in \overline{B_r}$ in \eqref{meanform} depends 
on the value of $s$, but since these are all points belonging to a compact set,
we can build $\{s_k\}_{k\in\N}\subseteq(0,1)$, $s_k\rightarrow 1$ as $k\rightarrow+\infty$,
such that $z(s_k)\rightarrow z_0\in \overline{B_r}$.
Since it is known (see \cite[Proposition 4.4]{hitchhiker}) that
$$
\lim_{s\uparrow 1}(-\Delta)^su=-\Delta u,
$$
then
$$
|(-\Delta)^su\,(z(s_k))+\Delta u(z_0)|\leq
|(-\Delta)^su\,(z(s_k))-(-\Delta)^su\,(z_0)|+
|(-\Delta)^su\,(z_0)+\Delta u(z_0)|
\xrightarrow{k\rightarrow+\infty}0.
$$
Finally we have proven that
\begin{multline*}
u(x)\ =\ \int_{\C B_r}\eta_r(y)\:u(x-y)\;dy + \gamma(n,s,r)(-\Delta)^su(z(s_k))
\\ \xrightarrow{k\rightarrow+\infty}\ 
\frac{1}{\omega_{n-1}\,{r}^{n-1}}\int_{\partial B_r}u(x-y)\;d\mathcal{H}^{n-1}(y)
 - \frac{r^2}{2n}\,\Delta u(z_0)
\end{multline*}
which is a known formula for $C^2$ functions,
see e.g. \cite[Proposition A.1.2]{dupaigne-book}.

\subsection*{Acknowledgements}
It is a pleasure and an honour to thank professor Louis Dupaigne
for his affectionate help, punctual suggestions, 
passionate questions, restless encouragement
and the attentive reading of this paper.

\nocite{*}
\bibliographystyle{plain}
\bibliography{biblio}

\begin{thebibliography}{10}

\bibitem{axler}
Sheldon Axler, Paul Bourdon, and Wade Ramey.
\newblock {\em Harmonic function theory}, volume 137 of {\em Graduate Texts in
  Mathematics}.
\newblock Springer-Verlag, New York, second edition, 2001.

\bibitem{bandle}
Catherine Bandle.
\newblock Asymptotic behavior of large solutions of elliptic equations.
\newblock {\em An. Univ. Craiova Ser. Mat. Inform.}, 32:1--8, 2005.

\bibitem{bogdan-bhp}
Krzysztof Bogdan.
\newblock The boundary {H}arnack principle for the fractional {L}aplacian.
\newblock {\em Studia Math.}, 123(1):43--80, 1997.

\bibitem{bogdan-repr-sharm}
Krzysztof Bogdan.
\newblock Representation of {$\alpha$}-harmonic functions in {L}ipschitz
  domains.
\newblock {\em Hiroshima Math. J.}, 29(2):227--243, 1999.

\bibitem{sharm}
Krzysztof Bogdan, Tomasz Byczkowski, Tadeusz Kulczycki, Michal Ryznar, Renming
  Song, and Zoran Vondra{\v{c}}ek.
\newblock {\em Potential analysis of stable processes and its extensions},
  volume 1980 of {\em Lecture Notes in Mathematics}.
\newblock Springer-Verlag, Berlin, 2009.
\newblock Edited by Piotr Graczyk and Andrzej Stos.

\bibitem{extension}
Luis Caffarelli and Luis Silvestre.
\newblock An extension problem related to the fractional {L}aplacian.
\newblock {\em Comm. Partial Differential Equations}, 32(7-9):1245--1260, 2007.

\bibitem{chen-felmer}
Huyan Chen, Patricio Felmer, and Alexander Quaas.
\newblock Large solutions to elliptic equations involving the fractional
  laplacian.
\newblock {\em preprint}, 2013.

\bibitem{chen-veron}
Huyan Chen and Laurent V\'eron.
\newblock Semilinear fractional elliptic equations involving measures.
\newblock {\em preprint at http://arxiv.org/pdf/1305.0945v2.pdf}, 2013.

\bibitem{chen}
Zhen-Qing Chen.
\newblock Multidimensional symmetric stable processes.
\newblock {\em Korean J. Comput. Appl. Math.}, 6:227--266, 1999.

\bibitem{clem-sweers}
Philippe Cl{\'e}ment and Guido Sweers.
\newblock Getting a solution between sub- and supersolutions without monotone
  iteration.
\newblock {\em Rend. Istit. Mat. Univ. Trieste}, 19(2):189--194, 1987.

\bibitem{ball-dupaigne}
Ovidiu Costin and Louis Dupaigne.
\newblock Boundary blow-up solutions in the unit ball: asymptotics, uniqueness
  and symmetry.
\newblock {\em J. Differential Equations}, 249(4):931--964, 2010.

\bibitem{uniq-dupaigne}
Ovidiu Costin, Louis Dupaigne, and Olivier Goubet.
\newblock Uniqueness of large solutions.
\newblock {\em J. Math. Anal. Appl.}, 395(2):806--812, 2012.

\bibitem{dherslegall}
Jean-St{\'e}phane Dhersin and Jean-Fran{\c{c}}ois Le~Gall.
\newblock Wiener's test for super-{B}rownian motion and the {B}rownian snake.
\newblock {\em Probab. Theory Related Fields}, 108(1):103--129, 1997.

\bibitem{hitchhiker}
Eleonora Di~Nezza, Giampiero Palatucci, and Enrico Valdinoci.
\newblock Hitchhiker's guide to the fractional {S}obolev spaces.
\newblock {\em Bull. Sci. Math.}, 136(5):521--573, 2012.

\bibitem{KO-dupaigne}
Serge Dumont, Louis Dupaigne, Olivier Goubet, and Vicentiu R{\u{a}}dulescu.
\newblock Back to the {K}eller-{O}sserman condition for boundary blow-up
  solutions.
\newblock {\em Adv. Nonlinear Stud.}, 7(2):271--298, 2007.

\bibitem{dupaigne-book}
Louis Dupaigne.
\newblock {\em Stable solutions of elliptic partial differential equations},
  volume 143 of {\em Chapman \& Hall/CRC Monographs and Surveys in Pure and
  Applied Mathematics}.
\newblock Chapman \& Hall/CRC, Boca Raton, FL, 2011.

\bibitem{felmer-quaas}
Patricio Felmer and Alexander Quaas.
\newblock Boundary blow up solutions for fractional elliptic equations.
\newblock {\em Asymptot. Anal.}, 78(3):123--144, 2012.

\bibitem{karlsen}
Kenneth~H. Karlsen, Francesco Petitta, and Suleyman Ulusoy.
\newblock A duality approach to the fractional {L}aplacian with measure data.
\newblock {\em Publ. Mat.}, 55(1):151--161, 2011.

\bibitem{keller}
Joseph~B. Keller.
\newblock On solutions of {$\Delta u=f(u)$}.
\newblock {\em Comm. Pure Appl. Math.}, 10:503--510, 1957.

\bibitem{polish}
Tomasz Klimsiak and Andrzej Rozkosz.
\newblock Dirichlet forms and semilinear elliptic equations with measure data.
\newblock {\em preprint at http://arxiv.org/pdf/1207.2263v2}, 2013.

\bibitem{landkof}
N.~S. Landkof.
\newblock {\em Foundations of modern potential theory}.
\newblock Springer-Verlag, New York, 1972.
\newblock Translated from the Russian by A. P. Doohovskoy, Die Grundlehren der
  mathematischen Wissenschaften, Band 180.

\bibitem{marcus-veron}
Moshe Marcus and Laurent V\'eron.
\newblock Existence and uniqueness results for large solutions of general
  nonlinear elliptic equations.
\newblock {\em J. of Evol. Eq.}, 3(1):637--652, 2004.

\bibitem{ponce}
Marcelo Montenegro and Augusto~C. Ponce.
\newblock The sub-supersolution method for weak solutions.
\newblock {\em Proc. Amer. Math. Soc.}, 136(7):2429--2438, 2008.

\bibitem{mselati}
Beno{\^{\i}}t Mselati.
\newblock Classification and probabilistic representation of the positive
  solutions of a semilinear elliptic equation.
\newblock {\em Mem. Amer. Math. Soc.}, 168(798), 2004.

\bibitem{nelson}
Edward Nelson.
\newblock A proof of {L}iouville's theorem.
\newblock {\em Proc. Amer. Math. Soc.}, 12:995, 1961.

\bibitem{osserman}
Robert Osserman.
\newblock On the inequality {$\Delta u\geq f(u)$}.
\newblock {\em Pacific J. Math.}, 7:1641--1647, 1957.

\bibitem{riesz}
Marcel Riesz.
\newblock Int\'egrales de {R}iemann--{L}ioville et potentiels.
\newblock {\em Acta Sci. Math.}, 9:1--42, 1938.

\bibitem{rosserra}
Xavier Ros-Oton and Joaquim Serra.
\newblock The {D}irichlet problem for the fractional {L}aplacian: regularity up
  to the boundary.
\newblock {\em \rm preprint at arXiv:1207.5985v1}, 2012.

\bibitem{silvestre}
Luis Silvestre.
\newblock Regularity of the obstacle problem for a fractional power of the
  {L}aplace operator.
\newblock {\em Comm. Pure Appl. Math.}, 60(1):67--112, 2007.

\bibitem{stampacchia}
Guido Stampacchia.
\newblock {\em \'{E}quations elliptiques du second ordre \`a coefficients
  discontinus}.
\newblock S\'eminaire de Math\'ematiques Sup\'erieures, No. 16 (\'Et\'e, 1965).
  Les Presses de l'Universit\'e de Montr\'eal, Montreal, Que., 1966.

\end{thebibliography}

\end{document}